\documentclass[11pt]{article}
\usepackage{amssymb, amsmath, amsthm, amsfonts,enumerate}
\usepackage{amsmath,setspace,scalefnt}
\usepackage[usenames,dvipsnames,svgnames,table]{xcolor}
\usepackage{graphicx,tikz,caption,subcaption}
\usetikzlibrary{patterns}
\usetikzlibrary{shapes}
\usepackage{fullpage}
\usepackage{hyperref}
\usepackage{enumitem,verbatim}
\usetikzlibrary{decorations.pathmorphing}
\usepackage{subcaption}

\tikzset{snake it/.style={decorate, decoration=snake}}
\usetikzlibrary{calc}

\definecolor{gray}{rgb}{0.25, 0.25, 0.25}

\newtheorem{theorem}{Theorem}[section]
\newtheorem{lemma}[theorem]{Lemma}
\newtheorem{cor}[theorem]{Corollary}

\theoremstyle{definition}

\theoremstyle{plain}
\newtheorem{claim}[theorem]{Claim}

\newtheorem{prop}[theorem]{Proposition}
\theoremstyle{definition}

\theoremstyle{definition}

\theoremstyle{definition}

\theoremstyle{definition}
\newtheorem{defn}[theorem]{Definition}
\theoremstyle{definition}

\theoremstyle{definition}

\newenvironment{poc}{\begin{proof}[Proof of claim]}{\end{proof}}

\newcommand{\ep}{\varepsilon}
\newcommand{\eps}{\varepsilon}

\newcommand{\al}{\alpha}
\newcommand{\de}{\delta}

\newcommand{\De}{\Delta}

\newcommand{\cH}{\mathcal{H}}
\newcommand{\cS}{\mathcal{S}}

\newcommand{\cC}{\mathcal{C}}
\newcommand{\cA}{\mathcal{A}}
\newcommand{\cP}{\mathcal{P}}

\newcommand{\N}{\mathbb{N}}

\newcommand{\tk}{\mathsf{TK}}

\newcommand{\bN}{\ensuremath{\mathbb{N}}}

\newcommand{\emref}[1]{\emph{\ref{#1}}}

\newcounter{propcounter}

\newcommand{\makenote}[2]
{

\smallskip

\noindent

\fbox{
\begin{minipage}{0.95\textwidth}

\def\temp{#1}
\ifx\temp\empty
\def\forlabel{$\bullet$}
\else
\def\forlabel{{\bfseries #1:}}
\fi

\begin{itemize}[label = \forlabel]
  #2
\end{itemize}
\end{minipage}
}

\smallskip
}

\title{A solution to Erd\H{o}s and Hajnal's odd cycle problem}
 \author{
Hong Liu
\thanks{Extremal Combinatorics and Probability Group (ECOPRO), Institute for Basic Science (IBS), Daejeon, South Korea. Email: {\tt hongliu@ibs.re.kr}. Supported by the Institute for Basic Science (IBS-R029-C4) and the UK Research and Innovation Future Leaders Fellowship MR/S016325/1.}
 \quad\quad
 Richard Montgomery
 \thanks{Mathematics Institute, University of Warwick, Coventry, CV4 7AL, UK. Email:
{\tt richard.montgomery@warwick.ac.uk}. Supported by the European Research Council (ERC) under the European Union Horizon 2020 research and innovation programme (grant agreement No. 947978) and the Leverhulme trust.}
 }

\begin{document}
\maketitle

\begin{abstract} In 1981, Erd\H{o}s and Hajnal asked whether the sum of the reciprocals of the odd cycle lengths in a graph with infinite chromatic number is necessarily infinite. Let $\mathcal{C}(G)$ be the set of cycle lengths in a graph $G$ and let $\mathcal{C}_\text{odd}(G)$ be the set of odd numbers in $\mathcal{C}(G)$. We prove that, if $G$ has chromatic number $k$, then $\sum_{\ell\in \mathcal{C}_\text{odd}(G)}1/\ell\geq (1/2-o_k(1))\log k$. This solves Erd\H{o}s and Hajnal's odd cycle problem, and, furthermore, this bound is asymptotically optimal.

In 1984, Erd\H{o}s asked whether there is some $d$ such that each graph with chromatic number at least $d$ (or perhaps even only average degree at least $d$) has a cycle whose length is a power of 2. We show that an average degree condition is sufficient for this problem, solving it with methods that apply to a wide range of sequences in addition to the powers of 2.

Finally, we use our methods to show that, for every $k$, there is some $d$ so that every graph with average degree at least $d$ has a subdivision of the complete graph $K_k$ in which each edge is subdivided the same number of times. This confirms a conjecture of Thomassen from 1984.
\end{abstract}

\section{Introduction}\label{sec:intro}
Does the chromatic number or the average degree imply anything about the cycle lengths of a graph? For any fixed $k$, we can never infer the presence of a cycle with length $k$, as a graph may have arbitrarily high chromatic number yet no such cycle (as Erd\H{o}s famously showed in 1959~\cite{probmethod}). Can we, however, say something about the density of cycle lengths or infer the existence of a cycle with length in some given infinite set of integers?

We will consider these two questions for both the even cycles and the odd cycles of a graph. From an average degree condition, we can only say something about the even cycles of a graph, as, of course, a graph with high average degree may have no odd cycles. The natural corresponding condition to impose for odd cycles is one on the chromatic number (see, for example, the survey by Verstra\"ete~\cite{SurVer}).

In this paper, we give the first constructions for even cycles with precise lengths using only an average degree condition. We then develop our methods further to find many odd cycle lengths in graphs with a chromatic number condition. This development, while itself novel, relies on the strength of our results on even cycles and cannot be done using the previous results guaranteeing different even cycle lengths. 

We now discuss further, and state our results on, even cycles in graphs with an average degree condition (in Section~\ref{sec:even}) and odd cycles in graphs with a chromatic number condition (in Section~\ref{sec:odd}). We then give a further application of our work which confirms a conjecture of Thomassen~\cite{Tho84} on subdivisions (in Section~\ref{sec:sub}). For questions on cycle lengths in graphs under additional conditions, and the background on them, we refer the reader to the comprehensive survey by Verstra\"ete~\cite{SurVer}.

\subsection{Average degree and even cycle lengths}\label{sec:even}
In 1966, Erd\H{o}s and Hajnal~\cite{EH66} suggested the harmonic sum of the cycle lengths in a graph as a measure of the density of its cycle lengths. In particular, letting  $\mathcal{C}(G)$ be the set of cycle lengths in a graph $G$, Erd\H{o}s and Hajnal~\cite{EH66} asked whether
\begin{equation}\label{eq:harmonic}
\sum_{\ell\in \mathcal{C}(G)}\frac{1}{\ell}\to \infty \quad \text{ as }\quad \chi(G)\to \infty,
\end{equation}
where $\chi(G)$ is the chromatic number of $G$. As Erd\H{o}s later wrote~\cite{erdos1975-42}, they felt that \eqref{eq:harmonic} should even hold under the weaker condition $d(G)\to \infty$, where $d(G)$ is the average degree of the graph~$G$. In 1984, Gy\'arf\'as, Koml\'os and Szemer\'edi~\cite{gyarfas1984distribution} confirmed this stronger conjecture by proving that any graph $G$ with average degree $d$ has $\sum_{\ell\in \cC(G)}1/\ell=\Omega_d(\log d)$.
If $G$ is a complete balanced bipartite graph with average degree $d$, then $\sum_{\ell\in \cC(G)}1/\ell=(1/2+o_d(1))\log d$, so this lower bound is tight up to the implicit constant. Here and throughout the paper, we write $\log$ for the natural logarithm. Erd\H{o}s~\cite{erdos1975-42} had previously stated, in 1975, that it was likely that $(1/2+o_d(1))\log d$ is the correct asymptotic lower bound, over all graphs $G$ with $d(G)\geq d$, and this remained an open problem.

We say an increasing sequence $\sigma_1,\sigma_2,\sigma_3,\ldots$ of integers is \emph{unavoidable with high average degree} if there is some $d$ such that any graph with average degree at least $d$ has a cycle with length in $\sigma_1,\sigma_2,\sigma_3,\ldots$. In 1977, Bollob\'as~\cite{bollobas1977cycles} confirmed a conjecture of Burr and Erd\H{o}s by showing that, if $\sigma_i$ forms an arithmetic progression containing even numbers, then $\sigma_i$ is unavoidable with high average degree.
Solving a problem of Erd\H{o}s, in 2005 Verstra\"ete~\cite{Ver05} showed that some unavoidable sequence with density 0 must exist, without giving an explicit sequence.

In 2008, Sudakov and Verstra\"ete~\cite{SV08} showed that many increasing sequences of integers are unavoidable in all but (potentially) exceptionally sparse graphs. In particular, they showed that, for any exponentially bounded increasing sequence of even integers $\sigma_i$ (that is, where $\sigma_{i+1}\le C\sigma_{i}$ for each $i\in \N$ and some fixed $C>1$), any $n$-vertex graph $G$ with $\sigma_i\notin \mathcal{C}(G)$ for each $i\in \N$ must have average degree at most $e^{O(\log^*n)}$, where $\log^\ast n$ is the iterated logarithm function.
In 1984, Erd\H{o}s~\cite{erdos1984-21} had asked whether the powers of 2 are unavoidable with high average degree, but, despite the results quoted here, there remained no explicit sequence with density 0 which was known to be unavoidable with high average degree (or even unavoidable with high chromatic number, as defined analogously in Section~\ref{sec:odd}).

In this paper, we introduce new techniques for constructing even cycles while controlling their length. This allows us to find, in any graph $G$, a long interval of consecutive even numbers in $\mathcal{C}(G)$, as follows.

\begin{theorem}\label{thm:mindeg} There is $d_0>0$ such that the following holds. If $G$ is a graph with average degree $d\ge d_0$,  then, there is some $\ell\geq d/(10\log^{12}d)$ such that $\mathcal{C}(G)$ contains every even integer in $[\log^{8} \ell,\ell]$.
\end{theorem}

From the density of the even numbers in the interval $[\log^{8} \ell,\ell]$ as $\ell$ increases, we get immediately the following improvement of the result of Gy\'arf\'as, Koml\'os and Szemer\'edi~\cite{gyarfas1984distribution}, which confirms the asymptotically correct lower bound on the harmonic sum of $\cC(G)$, as conjectured by Erd\H{o}s~\cite{erdos1975-42}.

\begin{cor} If a graph $G$ has average degree $d$, then
\[
\sum_{\ell\in \cC(G)}\frac{1}{\ell}\geq \left(\frac{1}{2}-o_d(1)\right)\log d.
\]
\end{cor}

By Theorem~\ref{thm:mindeg}, any avoidable sequence of cycle lengths must continue to avoid some intervals $[\log^{8} \ell,\ell]$ as $\ell\to\infty$.
Thus, many increasing sequences of even integers are unavoidable with high average degree, as follows.

\begin{cor}\label{conj:unavoidable} There is some $d_0>0$ such that the following holds. Given any infinite sequence $\sigma_i$, $i\in \N$, of increasing even integers with $\sigma_{i+1}\leq \exp(\sigma_i^{1/10})$ for each $i\in \N$, any graph $G$ with average degree at least $\max\{d_0,\sigma^{2}_1\}$
has some $i\in\N$ with $\sigma_i\in \mathcal{C}(G)$.
\end{cor}

In particular, this answers the question of Erd\H{o}s~\cite{erdos1984-21} mentioned above by showing that the powers of 2 are unavoidable with high average degree, as, furthermore, is any exponentially bounded sequence of increasing even numbers. This latter implication answers a question of Sudakov and Verstra\"ete~\cite{SV08}; for further questions on specific sequences answered by Corollary~\ref{conj:unavoidable},
we refer the readers to~\cite{Erd95}.

On the other hand, the sequence defined by $\sigma_1=1$ and $\sigma_{i+1}=2(i+1)^{\sigma_{i}}$ for each $i\in \N$ is known to be avoidable (see~\cite{SV08}). Corollary~\ref{conj:unavoidable} is thus optimal up to the fraction $1/10$ in the exponent, though it may hold with $1/10$ replaced by $1-o_i(1)$. We have not optimised our methods to maximise the fraction $1/10$, but note that doing so could not increase it beyond $1/3$ (see Section~\ref{sec:limits}). 


\subsection{Chromatic number and odd cycle lengths}\label{sec:odd}
Consider the set of odd cycle lengths in $G$, by letting $\mathcal{C}_\text{odd}(G)=\{\ell\in \mathcal{C}(G):\ell\text{ is odd}\}$. As noted above, a natural condition to impose in search of odd cycles is one on the chromatic number rather than the average degree, and, for example, Gy\'arf\'as~\cite{gyarfas} proved in 1992 that a graph with chromatic number at least $2k+1$ must have $|\cC_\text{odd}(G)|\geq k$. In 1981, Erd\H{o}s and Hajnal~\cite{erdos1981-20} asked whether
\begin{equation}\label{eq:harmonic-odd}
\sum_{\ell\in \mathcal{C}_\text{odd}(G)}\frac{1}{\ell}\rightarrow \infty\quad \text{as} \quad \chi(G)\rightarrow\infty.
\end{equation}
As Erd\H{o}s noted~\cite{erdos1984-21}, this gives a `much deeper question' than the corresponding question for all cycle lengths and chromatic number that we considered in Section~\ref{sec:even}. Indeed, the only relevant result towards this has been by Sudakov and Verstra\"ete~\cite{SV11}, who showed that $\sum_{\ell\in \mathcal{C}_\text{odd}(G)}1/\ell\rightarrow \infty$ if the independence ratio of $G$ is not extremely small compared to its number of vertices. Here, the independence ratio of $G$ is $\sup_{X\subseteq V(G)}\frac{|X|}{\alpha(X)}$, where $\alpha(X)$ is the independence number of the subgraph of $G$ induced on $X$ and the supremum is taken over all non-empty vertex subsets.

Building on our methods for even cycles, we use the precision of these techniques to construct a large interval of cycle lengths using a high chromatic number, as follows.

\begin{theorem}\label{thm:chrom} For each $\eps>0$, there is some $k_0\in \N$ such that the following holds for each $k\geq k_0$. If $G$ is a graph with chromatic number $k$, then, for some $\ell\in \mathbb{N}$, $\mathcal{C}(G)$ contains every odd integer in $[\ell,\ell\cdot k^{1-\eps}]$.
\end{theorem}

As the harmonic sum of the odd integers in $[\ell,\ell\cdot k^{1-\eps}]$ diverges as $k\to\infty$ (for all values of $\ell$), this solves Erd\H{o}s and Hajnal's odd cycle problem by confirming  \eqref{eq:harmonic-odd}. Furthermore, in combination with Theorem~\ref{thm:mindeg}, we get the following immediate lower bound for the harmonic sum of the cycle lengths of any specific residue, which answers another question of Erd\H{o}s~\cite{erdos1984-21}.


\begin{cor}\label{thm:reciprocal}
Let $a,b\in \N$, and let $\mathcal{C}_{a,b}(G)=\{\ell\in \mathcal{C}(G):\ell\equiv a\!\!\mod b\}$. If $G$ has chromatic number $k$, then
\[
\sum_{\ell\in \mathcal{C}_{a,b}(G)}\frac{1}{\ell}\geq \left(\frac{1}{b}-o_k(1)\right)\log k.
\]
\end{cor}

We say an increasing sequence $\sigma_1,\sigma_2,\sigma_3,\ldots$ of integers is \emph{unavoidable with high chromatic number} if there is some $k$ such that any graph with chromatic number at least $k$ has a cycle with length in $\sigma_1,\sigma_2,\sigma_3,\ldots$. There have previously been no non-trivial sequences of increasing odd integers which were known to be unavoidable with high chromatic number, but, similarly to in Section~\ref{sec:even}, Theorem~\ref{thm:chrom} immediately implies that many such sequences are unavoidable with high chromatic number, as follows.

\begin{cor}\label{thm:unmissable} Given $C\in \N$, there exists $k_0\in\bN$ such that the following holds. Let $\sigma_1,\sigma_2,\ldots$ be an infinite increasing sequence of odd integers such that $\sigma_{i+1}\leq C\sigma_i$ for each $i\in \N$. Then, every graph $G$ with chromatic number at least $\max\{k_0,\sigma_1^2\}$ has some $i\in \N$ with $\sigma_i\in \mathcal{C}(G)$.
\end{cor}

We remark that our proof in fact shows that Theorem~\ref{thm:chrom} (and hence also Corollaries~\ref{thm:reciprocal} and~\ref{thm:unmissable}) holds with the weaker hypothesis that the graph $G$ is `not too close' to being bipartite. That is, as seen in Section~\ref{sec:min}, we use only the condition that, after removing the edge-set of any bipartite graph from $G$, there is still a subgraph with average degree $\Omega(k)$.

\subsection{Balanced subdivisions}\label{sec:sub}
A \emph{subdivision} of a graph $G$ is obtained by replacing each edge of $G$ by a path, such that the new paths are internally vertex disjoint. This notion has played a central role in topological graph theory since the seminal result of Kuratowski in 1930 that a graph is planar if and only if it does not contain a subdivision of the complete graph with five vertices or a subdivision of the complete bipartite graph with three vertices on each side~\cite{Kur30}.

In 1967, Mader~\cite{Mad67} proved that, for each $k\in \N$, there is some $d=d(k)$ such that every graph with average degree at least $d$ contains a subdivision of the complete graph $K_k$. After improved bounds on $d(k)$ by Mader~\cite{Mad72}, and Koml\'os and Szemer\'edi~\cite{K-Sz-1}, Bollob\'as and Thomason~\cite{BT98} proved that, optimally, we may take $d(k)=O(k^2)$. Koml\'os and Szemer\'edi~\cite{K-Sz-2} later improved their own methods to give an independent proof of this, and the graph expansion methods they introduced (see Section~\ref{sec:expand}) form the basis for many constructions in sparse graphs both here and elsewhere (see, for example,~\cite{KLSS17,liu2017proof}).

Our constructive approach to controlling the length of cycles also allows us to control the length of paths, and thus construct subdivisions in which each edge is replaced by a path of the same length. For integers $\ell,k\in\bN$, denote by $\tk ^{(\ell)}_k$ a subdivision of a complete graph $K_{k}$ in which each edge is replaced by a path with length $\ell$. We say that $\tk^{(\ell)}_k$ is a \emph{balanced subdivision} of $K_{k}$.

In 1984, Thomassen~\cite{Tho84} (see also~\cite{Tho85},~\cite{thomassen}) conjectured that, for each $k\in\bN$, high average degree in a graph is sufficient to guarantee a balanced subdivision of $K_k$. This was open even for $k=4$. We confirm Thomassen's conjecture, as follows.

\begin{theorem}\label{thm:balancedsub}
For each $k\in \N$, there exists $d$ such that every graph with average degree at least $d$ contains a $\tk^{(\ell)}_{k}$ for some $\ell\in \N$.
\end{theorem}

We note that it is conceivable that there is some $\eps>0$ such that any graph with average degree at least $d$ in fact contains a $\tk^{(\ell)}_{\eps\sqrt{d}}$, which would be optimal up to the value of $\eps$. Furthermore, this may be provable using appropriate extensions of our methods (in particular along the lines of the techniques in \cite{liu2017proof}). However, though the balanced subdivision problem was the original focus of our work, we do not push these techniques further at the expense of a clear presentation of the new cycle construction techniques.


\section{Structure and proof sketches}\label{sec-prelim}

After describing our notation in Section~\ref{sec:not}, in Section~\ref{sec:expand} we recall the graph expansion concepts introduced by Koml\'os and Szemer\'edi in~\cite{K-Sz-1,K-Sz-2}. This allows us to give our main theorem, Theorem~\ref{mainthm}, in Section~\ref{sec:mainthm} and derive Theorem~\ref{thm:mindeg}, while noting how we use similar methods to prove Theorem~\ref{thm:balancedsub}. We then sketch the proof of Theorem~\ref{mainthm} in Section~\ref{sec:mainthmsketch}. The proof of Theorem~\ref{thm:chrom} is sketched in Section~\ref{sec:chromsketch}. Finally, in Section~\ref{sec:techhigh}, we highlight one particular innovation for our constructions in expanders, which may prove useful elsewhere.

\subsection{Notation}\label{sec:not}

Let $G$ be a graph, let $v\in V(G)$ be a vertex and let $W\subseteq V(G)$ be a set of vertices. We write $|G|=|V(G)|$ for the order of the graph.
Let $\de(G), d(G)$ and $\De(G)$ be the minimum, average and maximum degree of $G$ respectively, and let $N_G(v)$ be the set of neighbours of $v$ in $G$.
Denote by $N_G(v,W)$ the set of neighbours of $v$ in $W$, and denote by $d_G(v,W)=|N_G(v,W)|$ the degree of $v$ into $W$ in $G$. Denote the (external) neighbourhood of $W$ by $N_G(W)=(\cup_{v\in W}N(v))\setminus W$.
Let $G[W]\subseteq G$ be the induced subgraph of $G$ with vertex set $W$.
Denote by $G-W$ the induced subgraph $G[V(G)\setminus W]$.
We write $N_G^0(W)=W$, and, for each integer $k\ge 1$, let
$N_G^{k}(W)=N_G(N_G^{k-1}(W))$ be the set of vertices a graph distance $k$ from $W$, and let $B_G^k(W)=\cup_{0\le j\le k}N_G^{j}(W)$ be the ball of radius $k$ around $W$ in $G$. We let $B(W)=B^1(W)$.

Given graphs $G$ and $H$, the graph $G\cup H$ has vertex set $V(G)\cup V(H)$ and edge set $E(G)\cup E(H)$. Denote by $G\setminus H$ the graph with vertex set $V(G)$ and edge set $E(G)\setminus E(H)$. For a collection $\cP$ of graphs, denote by $|\cP|$ the number of graphs in $\cP$ and write $V(\cP)=\cup_{G\in \cP}V(G)$.

For a path $P$, let its length be $\ell(P)$. Where we say $P$ is a path from a vertex set $A$ to a disjoint vertex set $B$, we mean that $P$ has one endvertex in each of $A$ and $B$, and no internal vertices in $A\cup B$. For each $\ell\in \N$ and $k>0$, $\tk ^{(\ell)}_k$ is a subdivision of a complete graph $K_{\lfloor k\rfloor}$ in which each edge is replaced by a path with length $\ell$.

Many of our results that build to the theorems stated in Section~\ref{sec:intro} state that for each $\eps_1,\eps_2>0$ (and perhaps each $k\in \N$), there is some $d_0(\eps_1,\eps_2)$  (or $d_0(\eps_1,\eps_2,k)$) such that some property holds for $n\geq d\geq d_0$. For brevity, we do not calculate the function $d_0(\eps_1,\eps_2)$  (or $d_0(\eps_1,\eps_2,k)$) and assume implicitly in our proofs that $n$ and $d$ are as large as needed, depending on $\eps_1$ and $\eps_2$ (and perhaps $k$) --- where it may help the reader we recall this at various points in the proofs. The results from now on in this paper are stated and proved in order, so that these functions could be chosen sequentially through the paper.

We omit the subscript $G$ when the underlying graph $G$ is clear.
When it is not essential, we omit the floors and ceilings. All logarithms are natural.

\subsection{Koml\'os-Szemer\'edi graph  expansion}\label{sec:expand}
An expansion property in a graph $G$ is typically one in which every set $X\subseteq V(G)$ satisfies $|N_G(X)|\geq \eps|X|$ for some function $\eps$ depending on the size of $X$.
Expansion is key to all of our constructions. Our expansion must therefore exist in some subgraph of any graph.
Following Koml\'os and Szemer\'edi~\cite{K-Sz-1,K-Sz-2}, effectively we use the strongest expansion that can be found in some subgraph of any graph, based on its average degree.

\begin{defn}
For each $\ep_1>0$ and $k>0$, a graph $G$ is an \emph{$(\eps_1,k)$-expander} if
$$|N(X)|\geq \eps(|X|,\eps_1,k)\cdot |X|$$
for all $X\subseteq V(G)$ with $k/2\leq |X|\leq |G|/2$, where
\begin{eqnarray}\label{epsilon}
\ep(x,\ep_1,k):=\left\{\begin{tabular}{ l l }
$0$ & $\mbox{ if } x<k/5$, \\
$\ep_1/\log^2(15x/k)$ & $\mbox{ if } x\ge k/5$. \\
\end{tabular}
\right.
\end{eqnarray}
Whenever the choices of $\eps_1, k$ are clear, we omit them and write $\eps(x)$ for $\eps(x,\eps_1,k)$.
\end{defn}

If an $n$-vertex graph $G$ is an $(\eps_1,k)$-expander and $X\subseteq V(G)$ has size at least $k/2$, then $B^i_G(X)$ increases as $i$ increases, until the set contains at least $n/2$ vertices. The rate of expansion, $|N_G(B^i_G(X))|/|B^i_G(X)|\geq \eps(|B^i_G(X)|,\eps_1,k)$ guaranteed by the expansion condition decreases as $i$ increases (see, for example, \cite{K-Sz-2}). That is, $\ep(x,\eps_1,k)$ decreases as $x\geq k/2$ increases. However, $\ep(x,\eps_1,k)\cdot x$ increases as $x$ does, so the lower bound from expansion we have for $|N_G(B_G^i(X))|$ increases as $i$ increases.

As Koml\'os and Szemer\'edi~\cite{K-Sz-2} showed, every graph $G$ contains an expander with comparable average degree to $G$, as follows.

\begin{theorem}[\cite{K-Sz-2}]\label{thm-expander}
There exists some $\eps_1>0$ such that the following holds for every $k>0$. Every graph $G$ has an $(\eps_1,k)$-expander subgraph $H$ with $d(H)\geq d(G)/2$ and $\delta(H)\geq d(H)/2$.
\end{theorem}

Note that, in Theorem~\ref{thm-expander},  the expander subgraph $H$ can be much smaller than the original graph $G$ in size. Indeed, $G$ could be the disjoint union of many copies of such a graph $H$.

We use expansion to expand and connect vertex sets, creating paths to construct cycles of varying lengths. A typical use is the following result of Koml\'os and Szemer\'edi, though we use the comparable Lemma~\ref{new-connect}.

\begin{lemma}[\cite{K-Sz-2}]\label{lem-diameter} Let $\eps_1,k>0$.
	If $G$ is an $n$-vertex $(\ep_1,k)$-expander, then any two vertex sets, each of size at least
	$x\ge k$, are of distance at most $\frac{2}{\ep_1}\log^3(15n/k)$ apart. This remains true even after deleting $x\cdot \eps(x)/4$ arbitrary vertices from $G$.
\end{lemma}

It is convenient to work in a bipartite graph; to do this we use the following simple and well known result.

\begin{prop}\label{bipsub} Within any graph $G$ there is a bipartite subgraph $H$ with $d(H)\geq d(G)/2$.
\end{prop}

Combining this with Theorem~\ref{thm-expander},
we get the following immediate corollary.

\begin{cor}\label{cor-expander}
There exists some $\eps_1>0$ such that the following holds for every $\eps_2>0$ and $d\in \N$. Every graph $G$ with $d(G)\geq 8d$ has a bipartite $(\eps_1,\eps_2d)$-expander subgraph $H$ with $\delta(H)\geq d$.\hfill\qed
\end{cor}


\subsection{A stronger version of Theorem~\ref{thm:mindeg}}\label{sec:mainthm}

We will prove Theorem~\ref{thm:mindeg} in the slightly stronger form of Theorem~\ref{mainthm} below, which applies to an expander. Combining this with Corollary~\ref{cor-expander} easily gives Theorem~\ref{thm:mindeg}, as shown below. We use Theorem~\ref{mainthm} to prove Theorem~\ref{thm:chrom}, as outlined in Section~\ref{sec:chromsketch}. Theorem~\ref{thm:balancedsub} is proved using very similar methods to Theorem~\ref{mainthm}, and we comment on this below.

In order to state Theorem~\ref{mainthm}, we use the following definition, which records whether there will be even or odd length paths between two vertices $u$ and $v$ in a connected bipartite graph $H$.

\begin{defn} For any connected bipartite graph $H$ and $u,v\in V(H)$, let\label{pidefn}
$$
\pi(u,v,H)=\left\{\begin{array}{ll}
0 & \text{ if }u=v, \\
1 & \text{ if $u$ and $v$ are in different vertex classes in the (unique) bipartition of $H$},\\
2 & \text{ if $u$ and $v$ are in the same vertex class and $u\neq v$}.
\end{array}
\right.
$$
\end{defn}

\noindent Note that, for example by Lemma~\ref{lem-diameter}, any $(\eps_1,k)$-expander subgraph with minimum degree at least $k$ is connected, and this will allow us to use this definition for the bipartite expanders we use.

In common with many of our results in the rest of this paper, Theorem~\ref{mainthm} applies only to $\tk_{\ell}^{(2)}$-free graphs for some $\ell$ (often $\ell=d/2$). A subdivision of the complete graph on $\ell$ vertices, with each edge subdivided into a path of length 2, has many different even cycle lengths, and many different path lengths between pairs of vertices. As our aim in applying Theorem~\ref{mainthm} is to find only \emph{some} subgraph with this property (see, for example, the proof of Corollary~\ref{mainforchrom}), we need look no further than such a subdivision. Why our constructions require the graph to be $\tk_{\ell}^{(2)}$-free is commented on in Section~\ref{skewbi}.

\begin{theorem}\label{mainthm} There exists $\eps_1>0$, such that, for each $0<\eps_2<1/5$, there exists $d_0=d_0(\eps_1,\eps_2)$ such that the following holds for each $n\geq d\geq d_0$. Suppose that $H$ is a $\tk_{d/2}^{(2)}$-free bipartite $n$-vertex $(\eps_1,\eps_2 d)$-expander with $\de(H)\ge d$.
	Let $x,y\in V(H)$ be distinct,  and let
	$$\ell\in [\log^{7}n,n/\log^{12}n]$$
	satisfy $\pi(x,y,H)= \ell\mod 2$.

  Then, $H$ contains an $x,y$-path with length $\ell$.
\end{theorem}

 Theorem~\ref{thm:mindeg} follows from Theorem~\ref{mainthm} and Corollary~\ref{cor-expander}, as follows.

\begin{proof}[Proof of Theorem~\ref{thm:mindeg}] Let $\eps_1>0$ be such that the condition in Corollary~\ref{cor-expander} applies, and let $\eps_2=1/100$. Let $d_0$ be large (see Section~\ref{sec:not}), and in particular large enough that the property in Theorem~\ref{mainthm} holds for $\eps_1$ and $\eps_2'=8\eps_2$, for each $d\geq d_0/8$.

  Let $G$ be a graph with average degree $d\ge d_0$ and let $\bar{d}=d/8$. By the property from Corollary~\ref{cor-expander}, $G$ contains a bipartite $(\eps_1,\eps_2 d)$-expander subgraph $H$ with $\delta(H)\geq \bar{d}$. As $\bar{d}=d/8$, $H$ is an $(\eps_1,8\eps_2\bar{d})$-expander. If $H$ contains a $\tk_{\bar{d}/2}^{(2)}$, then $H$ contains every even cycle length between 6 and $\bar{d}$. As $\bar{d}=d/8$ is large, the property in the theorem holds with $\ell=\bar{d}$ as $\log^{8}\bar{d}\geq 6$.

Assume then that $H$ is $\tk_{\bar{d}/2}^{(2)}$-free.  As $\delta(H)\geq \bar{d}>0$, we can pick distinct vertices $x,y\in V(H)$ such that $xy\in E(H)$. Note that $\pi(x,y,H)=1$. Let $n=|H|$ and $\ell=n/\log^{12}n\geq \bar{d}/\log^{12}\bar{d}\geq d/(10\log^{12}d)$, as $\bar{d}=d/8$ is large.

 For every even $\ell'\in [\log^{8}\ell,\ell]$, $(\ell'-1)$ is an odd number in $[\log^{7}n,n/\log^{12}n]$.
  Then, by the property from Theorem~\ref{mainthm} applied with $x$ and $y$, there is an $x,y$-path with length $\ell'-1$ in $H$, and therefore a cycle with length $\ell'$ in $H$.
\end{proof}

For Theorem~\ref{thm:balancedsub}, essentially, we find a copy of $\tk_k^{(\ell)}$, for some $\ell\in \N$, by first taking an expander subgraph and $k$ distinct vertices within it to be the \emph{core vertices}. Core vertices in a $K_k$-subdivision are the vertices which are not interior vertices of a path which replaces an edge. Using the same construction as for Theorem~\ref{mainthm} multiple times, we then find internally vertex disjoint paths with the same length between each pair of core vertices. This is done in Section~\ref{sec:balancedsub}.

\subsection{Proof sketch for Theorem~\ref{mainthm}}\label{sec:mainthmsketch}
To discuss the proof of Theorem~\ref{mainthm}, let $H$ be a $\tk_{d/2}^{(2)}$-free bipartite $n$-vertex $(\eps_1,\eps_2d)$-expander, with $0<\eps_1,\eps_2<1$, such that $\de(H)\geq d$, and let
$x,y\in V(H)$ be distinct.

Our aim is to find a sequence of $x,y$-paths in $H$ whose lengths increase by 2 each time. In Section~\ref{sec:onecycle} we describe how one cycle can be used to find two paths with the same endvertices and lengths differing by 2, and how a connected sequence of cycles can create a longer sequence of paths with lengths increasing by 2 each time. In Section~\ref{sec:manycycles}, we discuss how a sequence of cycles can be connected, and introduce the concept of an \emph{adjuster}. In Section~\ref{sec:manypaths}, we describe how a polylogarithmic number of adjusters can be used to find long paths with the lengths required by Theorem~\ref{mainthm}. In Section~\ref{sec:limits}, we discuss the natural limitations of these methods.

\subsubsection{Creating a little adjustment with small cycles}\label{sec:onecycle}
Our proof of Theorem~\ref{mainthm} is based on the following simple idea. Suppose we can find in $H$ a short cycle $C$, with length $2\ell$ say, which does not contain $x$ or $y$. Take two vertices $v_1$ and $v_2$ a distance $\ell-1$ apart on $C$. Then, $C$ contains one $v_1,v_2$-path with length $\ell-1$ and another with length $\ell+1$. If we can find, using new internal vertices, disjoint paths from $x$ to $v_1$ and from $y$ to $v_2$, then we have two $x,y$-paths whose lengths differ by 2 (see Figure~\ref{fig1}(a)).

If we can chain together many cycles between $x$ and $y$ in this fashion, then, by choosing the length of the path we take around each cycle, we can vary the length of the path between $x$ and $y$ in increments of 2 (see Figure~\ref{fig1}(c)).


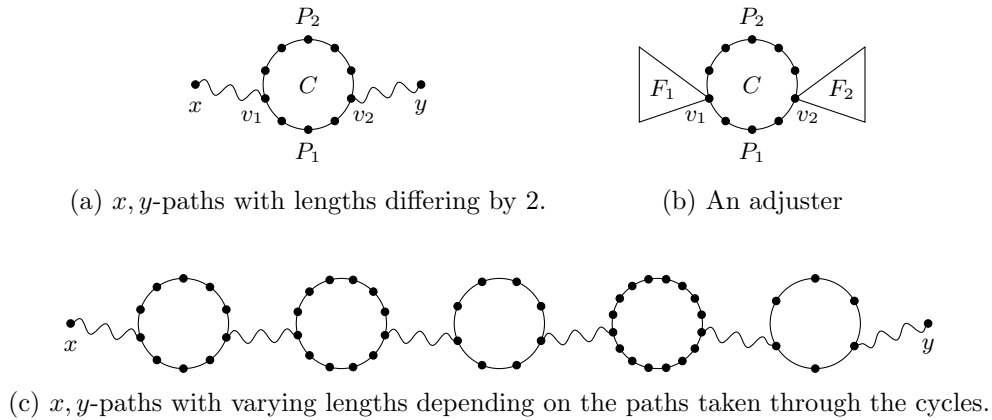
\begin{figure}[h]
  \centering
  \begin{subfigure}{0.4\textwidth}
    \centering
\begin{tikzpicture}

\def\radi{0.6}
\def\xyradi{1.5}
\def\adjj{0.3}

\draw (0,0) circle [radius=\radi];

\foreach \x in {1,...,10}
{
\coordinate (A\x) at ({36*\x+18}:\radi);
\draw [fill] (A\x) circle [radius=0.05cm];
}

\coordinate (X) at (180:\xyradi);
\coordinate (Y) at (0:\xyradi);
\draw [fill] (X) circle [radius=0.05cm];
\draw [fill] (Y) circle [radius=0.05cm];

\draw [snake it] (X) -- (A5);
\draw [snake it] (Y) -- (A9);

\draw (0,0) node {\footnotesize $C$};
\draw ($(A2)+(0,\adjj)$) node {\footnotesize $P_2$};
\draw ($(A7)-(0,\adjj)$) node {\footnotesize $P_1$};
\draw ($(A5)+\adjj*(235:1)$) node {\footnotesize $v_1$};
\draw ($(A9)+\adjj*(-55:1)$) node {\footnotesize $v_2$};
\draw ($(X)-(0,\adjj)$) node {\footnotesize $x$};
\draw ($(Y)-(0,\adjj)$) node {\footnotesize $y$};

\end{tikzpicture}
\caption{$x,y$-paths with lengths differing by 2.}
\end{subfigure}
  \begin{subfigure}{0.3\textwidth}
    \centering
\begin{tikzpicture}

\def\radi{0.6}
\def\xyradi{1.5}
\def\adjj{0.3}
\def\upp{0.5}

\draw (0,0) circle [radius=\radi];

\foreach \x in {1,...,10}
{
\coordinate (A\x) at ({36*\x+18}:\radi);
\draw [fill] (A\x) circle [radius=0.05cm];
}

\coordinate (X1) at ($(180:\xyradi)+(0,\upp)$);
\coordinate (X2) at ($(180:\xyradi)-(0,\upp)$);
\coordinate (Y1) at ($(0:\xyradi)+(0,\upp)$);
\coordinate (Y2) at ($(0:\xyradi)-(0,\upp)$);

\draw (A5) -- (X1) -- (X2) -- (A5);
\draw (A9) -- (Y1) -- (Y2) -- (A9);

\draw ($0.333*(A5)+0.333*(X1)+0.333*(X2)$) node {\footnotesize $F_1$};
\draw ($0.333*(A9)+0.333*(Y1)+0.333*(Y2)$) node {\footnotesize $F_2$};

\draw (0,0) node {\footnotesize $C$};
\draw ($(A2)+(0,\adjj)$) node {\footnotesize $P_2$};
\draw ($(A7)-(0,\adjj)$) node {\footnotesize $P_1$};
\draw ($(A5)+\adjj*(235:1)$) node {\footnotesize $v_1$};
\draw ($(A9)+\adjj*(-55:1)$) node {\footnotesize $v_2$};

\end{tikzpicture}
\caption{An adjuster}
\end{subfigure}

\vspace{0.7cm}

  \begin{subfigure}{\textwidth}
    \centering
\begin{tikzpicture}

\def\radi{0.6}
\def\xyradi{1.5}
\def\adjj{0.3}
\def\num{5}
\def\trans{2.1}

\foreach\y/\numm in {1/10,2/12,3/8,4/16,5/6}
{
\foreach \x in {1,...,\numm}
{
\coordinate (A\y-\x) at ($({360*\x/\numm+180/\numm}:\radi)+(\trans*\y-\trans,0)$);
\draw [fill] (A\y-\x) circle [radius=0.05cm];
}
\draw ($(\trans*\y-\trans,0)$) circle [radius=\radi];
}

\coordinate (X) at (180:\xyradi);
\coordinate (Y) at ($(\num*\trans-\trans,0)+(0:\xyradi)$);
\draw [fill] (X) circle [radius=0.05cm];
\draw [fill] (Y) circle [radius=0.05cm];

\draw [snake it] (X) -- (A1-5);
\draw [snake it] (A1-9) -- (A2-6);
\draw [snake it] (A2-11) -- (A3-4);
\draw [snake it] (A3-7) -- (A4-8);
\draw [snake it] (A4-15) -- (A5-3);
\draw [snake it] (Y) -- (A5-5);

\draw ($(X)-(0,\adjj)$) node {\footnotesize $x$};
\draw ($(Y)-(0,\adjj)$) node {\footnotesize $y$};

\end{tikzpicture}
\caption{$x,y$-paths with varying lengths depending on the paths taken through the cycles.}
\end{subfigure}
\caption{Creating $x,y$-paths of different lengths using cycles.}\label{fig1}
\end{figure}


\subsubsection{Connecting the cycles}\label{sec:manycycles}
In the situation above, if $C$ is a shortest cycle in $H$, then it is not too difficult to connect $x$ to $v_1$ and $y$ to $v_2$ (relabelling $v_1$ and $v_2$ if necessary). Indeed, for each $i\geq 0$ and $v\in V(H)$, any neighbours of the ball $B^{i}_{G-V(C)+v}(v)$ in $V(C)$ must be within distance $2i+2$ of each other on the cycle $C$, for otherwise $H$ contains a shorter cycle than $C$. Given our expansion conditions, this will mean that $B^{i}_{G-V(C)+v_1}(v_1)$ and $B^{i}_{G-V(C)}(x)$ will both expand as $i$ increases (see Section~\ref{sec:expandavoid}). Expanding these sets until they intersect allows us to find a short path from $v_1$ to $x$ while avoiding $V(C)\setminus\{v_1\}$. Roughly speaking, setting aside a shortest path $P$ from $\{v_1\}$ to $\{x,y\}$ which avoids $V(C)\setminus \{v_1\}$, we then expand around $v_2$ and the vertex left in $\{x,y\}\setminus V(P)$ while avoiding $(V(P)\cup V(C))\setminus \{v_2\}$ to find the second path described above.

Connecting multiple cycles is more difficult, simply as there are more vertices to avoid. Due to this, instead of cycles, we will find structures we call \emph{simple adjusters}. These correspond to the cycle $C$, vertices $v_1$ and $v_2$, as well as vertex disjoint subgraphs $F_{1},F_{2}\subseteq H$ such that $V(F_{1})\cap V(C)=\{v_1\}$ and $V(F_{2})\cap V(C)=\{v_2\}$ (see Figure~\ref{fig1}(b) and Definition~\ref{defn-adj}). The graph $F_{1}$ (and analogously $F_{2}$) has the property that every vertex is a distance $O_{\eps_1}(\log^3n)$ away from $v_1$ in $F_1$, and therefore a path leading into $F_{1}$ can be extended with a few additional vertices from $V(F_{1})$
to one ending with $v_1$. Furthermore, $|F_{1}|$ will be comfortably larger than $C$ so that $F_{1}$ can be expanded and connected while avoiding $V(C)\setminus \{v_1\}$. The distance $O_{\eps_1}(\log^3n)$ here comes from an application of Lemma~\ref{new-connect}.

The heart of our paper is the robust construction of these simple adjusters -- that is, having some set of vertices $U$ which is not too large, we find an adjuster in $H-U$. More discussion of this construction we defer till Section~\ref{sec:mainthmpf}, but we will highlight the key innovation that allows this in Section~\ref{sec:techhigh}.

\subsubsection{Increasing the size of the interval of path lengths}\label{sec:manypaths}
In an $n$-vertex $(\eps_1,\eps_2d)$-expander $H$, any two vertices $x,y$ are a distance at most $m=\Theta_{\eps_1}(\log^3n)$ apart (see Lemma~\ref{new-connect}). Using fairly straightforward methods (see those in Section~\ref{sec:lengthen}), we can use the expansion conditions to find, for any $20m\leq \ell\leq n/\log^{10} n$, a path with length between $\ell-20m$ and $\ell$ between $x$ and $y$ (see Corollary~\ref{longconnect4}). If we can find such a path which contains $10m$ simple adjusters, then, as long as $\ell=\pi(x,y,H)\mod 2$, we can use these adjusters to increase the path by length 2 until it has length exactly $\ell$ (see Section~\ref{sec:adjpath}).

\subsubsection{The limitations of our methods}\label{sec:limits}
Our methods are limited by the length of the paths we find between vertex pairs using the Koml\'os-Szemer\'edi graph expansion. Here, we can only guarantee a path with length at most $m=\Theta_{\eps_1}(\log^3n)$. Therefore the lower bound of the interval of path/cycle lengths in Theorem~\ref{mainthm} and Theorem~\ref{thm:mindeg} cannot be reduced below $\Omega(\log^3n)$ while using connection methods with this graph expansion. Correspondingly, optimising our methods cannot increase the fraction $1/10$ in Corollary~\ref{conj:unavoidable} beyond $1/3$.

For simplicity, we connect together around $m=\Theta_{\eps_1}(\log^3n)$ cycles which can each be used to adjust the length of the paths we find by 2. We note that, by using different cycles to adjust the length by different amounts, we could use perhaps $O(\log m)=O_{\eps_1}(\log\log n)$ cycles instead of around $m$ cycles. However, this saving does not reach a plausible optimal bound as we cannot guarantee our connecting paths are any shorter.


\subsection{Outline of the proof of Theorem~\ref{thm:chrom}}\label{sec:chromsketch}

Note that in a graph $H$ which is a copy of $\tk_{d}^{(2)}$, if $d\geq 3$, then every edge between two vertices $u,v$ can be replaced with a path with any odd length in $[5,2d-1]$. If $d$ is large, then this includes paths with any odd length in $[\log^{8}d,d]$. Using Corollary~\ref{cor-expander} and Theorem~\ref{mainthm}, then, in any graph with average degree at least $16d$, we can find a bipartite subgraph $H$ satisfying the following property.

\begin{enumerate}[label = \textbf{P}]
  \item There is some $\ell_H\geq d/\log^{12}d$ such that each edge $uv$ in $H$ can be replaced with a path with any odd length in $[\log^{8}\ell_H,\ell_H]$.\label{niceprop}
\end{enumerate}

Now, suppose $G$ has chromatic number at least $300d$ and find in $G$ a maximal collection $\cH$ of edge disjoint bipartite subgraphs $H$ satisfying \ref{niceprop} with some integer $\ell_H$. We can show that the chromatic number of the union of these graphs $\cup_{H\in \cH}H$ is at least 3, for otherwise $G\setminus (\cup_{H\in \cH}H)$ has high enough degree that some other bipartite subgraph satisfying \ref{niceprop} must exist, contradicting the maximality of $\cH$. Therefore, $\cup_{H\in \cH}H$ contains some odd cycle, $C$ say. Our aim is to use \ref{niceprop} to take many edges $e$ of $C$, and, where $H_e\in \cH$ is such that $e\in E(H_e)$, replace $e$ with odd paths in $H_e$ with many different lengths. This would allow us to transform $C$ into odd cycles of different lengths. The challenge is to do this so that all of the paths replacing these edges are vertex disjoint. This is possible by taking a certain minimal odd cycle $C$ and replacing edges $e$ corresponding to vertex disjoint graphs $H_e\in \cH$. This is discussed in more detail at the start of Section~\ref{sec:chrompf}, before Theorem~\ref{thm:chrom} is then proved.


\subsection{A new construction method: Robust construction of gadgets.}\label{sec:techhigh}
Here, we will highlight an innovation for constructions using Koml\'os-Szemer\'edi expansion. Roughly speaking, this technique overcomes the sublinear property of the expansion we work with. For more details, see Section~\ref{sec:techinno}. Recall from \eqref{epsilon} that in an expander $H$, any vertex set $A\subseteq V(G)$ of suitable size has neighbourhood with size $\Omega(|A|/\log^2|A|)$, expanding with at least a factor of $\eps(|A|)=\Theta(1/\log^2|A|)$.

Our aim is to find some special subgraph -- let us call it here a \emph{gadget} --  in an expander $H$ while avoiding a vertex set $W$. Roughly speaking, it is natural to pick a vertex $v$ in $H-W$ and try to construct a gadget locally in $H$ around $v$ using the expansion property. If this fails there will then be some vertex set $A_v$ containing $v$ which does not expand in $H-W$, not even by a reduced factor of $\eps(|A_v|)/10$.
If we fail repeatedly, then we find many disjoint sets $A_v$, say for the vertices $v\in V$, which do not expand in $H-W$ by a factor of  $\eps(|A_v|)/10$.
Eventually, $\cup_{v\in V}A_v$ will be much larger than $W$, so it  expands in $H-W$ by a factor of at least $\eps(|\cup_{v\in V}A_v|)/2$. However, as the expansion function in Section~\ref{sec:expand} is sublinear in $x$, it is not a contradiction that $\cup_{v\in V}A_v$ expands, yet no set $A_v$ does even if the expansion factor required is reduced by a factor of $5$.

Instead, we reach a contradiction by exploiting our particular circumstance. We will have that $W$ is polylogarithmic size (in $n$), that $|V|\geq n^{1/2}$, and the $\tk^{(2)}_{d/2}$-free condition will allow us to have that $|N_H(A_v,W)|\leq |A_v|^2$ for each $v\in V$. Note that, as, for each $v\in V$, $A_v$ does not expand in $H-W$ by a factor of $\eps(|A_v|)/10$, the set $A_v$ must have size at most $|W|\log^3n$.

 Now, firstly, if many of the sets $A_v$, $v\in V$, have size at least $(\log n)^{1/3}$, then we find a set $V'\subseteq V$ indexing some of these vertices so that $|W|\log^3n\leq |\cup_{v\in V'}A_v|\leq 2|W|\log^3n$. Then, $A:=\cup_{v\in V'}A_v$ is large enough to expand past $W$, and hence expand in $H-W$ with a factor of least $\eps(|A|)/2$.
 As, for each $v\in V'$, $\log |A_v|\approx \log\log n\approx \log|A|$, the expansion function $\eps$ for these sets is essentially the same -- near enough that we can get a contradiction from each set $A_v$, $v\in V'$, not expanding in $H-W$ by a factor of at least $\eps(|A_v|)/10$ yet $A=\cup_{v\in V'}A_v$ expanding in $H-W$ by a factor of at least $\eps(|A|)/2$.

 Thus, we can assume that most of the sets $A_v$, $v\in V$, have size at most $(\log n)^{1/3}$.
This gives us, for such $v\in V$, that $|N_H(A_v,W)|\leq |A_v|^2\leq (\log n)^{2/3}$.
As $W$ is polylogarithmic in size, we have at most $|W|^{(\log n)^{2/3}}= n^{o(1)}$ subsets of $W$ with size at most $(\log n)^{2/3}$. Thus, for some $r\leq (\log n)^{1/3}$, there must be at least $r^2$ vertices $v\in V$ -- say those in $V''$ -- such that $|A_v|=r$ for each $v\in V''$ and, over $v\in V''$, $N_H(A_v,W)$ is the same set, $W'\subseteq W$ say.  Let $A=\cup_{v\in V''}A_v$, so that $W'=\cup_{v\in V''}N_H(A_v,W)=N_H(A,W)$, and, for each $v\in V''$, $|A_v|=r$, and $|W'|\leq |A_v|^2=r^2$. As $|A|=r|V''|=r^3$, and $|N_H(A,W)|=|W'|\leq r^2$, $A$ does not expand well into $W$ in $H$, and hence must expand in $H-W$ by a factor of at least $\eps(|A|)/2$. Again, we have that $A_v$ does not expand in $H-W$ by a factor of at least $\eps(|A_v|)/10$ and $\log |A_v|\approx \log |A|$, for each $v\in V''$.
As before, this gives a contradiction.


\section{Preliminary expansion results}\label{sec:prelim}
In this section we cover the preliminary results on expansion that we use in Section~\ref{sec:mainthmpf}. The main results of this section are Lemmas~\ref{limit-contact},~\ref{new-connect},~\ref{lem-newbit},~\ref{lem-expansion},~\ref{cl-egg} and~\ref{lem-conexp} which we prove in Sections~\ref{sec:expandavoid},~\ref{sec:newpath},~\ref{sec:techinno},~\ref{sec:vxexpansion},~\ref{sec:largeexpansions} and~\ref{sec:lengthen} respectively.

\subsection{Expanding while avoiding sets}\label{sec:expandavoid}
In an expander $H$, we often want to expand a vertex set $A$, while avoiding another set $X$. We can do this if the set satisfies one of three conditions (matching the conditions \emph{\ref{sm1}}--\emph{\ref{sm3}} in Lemma~\ref{limit-contact}).
\begin{enumerate}
	\item Firstly, if $X$ is much smaller than $A$.\label{y1}

	\item Secondly, if $X$ is far enough (in graph distance) from $A$ in $H$ that $A$ expands to become much larger than $X$ before it encounters $X$.\label{y2}

	\item Finally, if $X$ does not intersect too much with each sphere around $A$ during the expansion. That is, if $N_H(B^{i}_{H-X}(A))$ does not intersect too much with $X$ for each $i\geq 0$.\label{y3}
\end{enumerate}

  For the last condition above, if $B^i_{H-X}(A)$ grows with $i$, then we can permit the intersection of $N_H(B^i_{H-X}(A))$ with $X$ to increase as $i$ increases. This condition is formally captured in the following notion, before we give our main lemma expanding a vertex set $A$ while avoiding some other vertex sets.

\begin{defn}\label{defn-limit-contact}
	A vertex set $A$ has \emph{$k$-limited contact} with a vertex set $X$ in a graph $H$ if, for each $i\in \N$,
	$$
	|N_H(B_{H-X}^{i-1}(A))\cap X|\leq k i.
	$$
\end{defn}

\begin{lemma} \label{limit-contact}
	Let $0<\ep_1, \ep_2<1$ and $k\in \N$. There is some $d_0=d_0(\eps_1,\eps_2,k)$ for which the following holds for each $n\geq d\geq d_0$.

  Suppose $H$ is an $n$-vertex $(\ep_1,\ep_2 d)$-expander.
	Let $m=\frac{16}{\eps_1}\log^3n$ and $\ell_0=(\log\log n)^5$. Let $A\subseteq V(H)$ with $|A|\ge \ep_2d/2$ and let $X,Y,Z\subseteq V(H)\setminus A$ be such that the following hold.
\stepcounter{propcounter}
\begin{enumerate}[label = {\bfseries \emph{\Alph{propcounter}\arabic{enumi}}}]
    \item $|X|\leq |A|\eps(|A|)/4$.\label{sm1}
    \item $B^{\ell_0}_{H-X-Z}(A)\cap Y=\varnothing$ and $|Y|\leq m^{300k}$.\label{sm2}
    \item $A$ has $k$-limited contact with $Z$ in $H$.\label{sm3}
\end{enumerate}
Then,
\begin{enumerate}[label = \emph{(\roman{enumi})}]
\item  $|B^{\ell_0}_{H-X-Y-Z}(A)|>m^{400k }$, and\label{exp1}
\item  $|B^{m}_{H-X-Y-Z}(A)|>n/2$.\label{exp2}
\end{enumerate}
\end{lemma}

\begin{proof}  We first prove \emph{\ref{exp1}}. Note that, by \emph{\ref{sm2}}, we have $B^{\ell_0}_{H-X-Y-Z}(A)=B^{\ell_0}_{H-X-Z}(A)$. Let $F=H-X-Z$, and suppose that $|B^{\ell_0}_{F}(A)|\le m^{400k }$, for otherwise \emph{\ref{exp1}} holds. We will show the following claim.

  \begin{claim} For each $0\le r\le \ell_0-1$,\label{sizeclaim}
	\begin{equation}\label{equa1}
	|N_{F}(B^r_{F}(A))|\ge \frac{1}{4}|B_F^r(A)|\cdot \ep(|B_F^r(A)|).
	\end{equation}
  \end{claim}

	Given Claim~\ref{sizeclaim}, we reach a contradiction as follows, and thus \emph{\ref{exp1}} must hold. For each $0\leq r< \ell_0$, we have $|B^r_F(A)|\leq |B^{\ell_0}_F(A)|\leq m^{400k }$, and hence, as $n\geq d\geq d_0(\ep_1,\eps_2,k)$ is large,
  \[
  \ep(|B_F^r(A)|)\geq \eps(m^{400k})=\frac{\eps_1}{\log^2(15m^{400k}/\eps_2d)}\geq\frac{4}{(\log\log n)^3}.
 \]
  Therefore, for each $0\leq r<\ell_0$, by Claim~\ref{sizeclaim},
  \[
  |B^{r+1}_{F}(A)|=|B^r_{F}(A)|+|N_{F}(B^r_{F}(A))|\ge \left(1+\frac{\ep(|B_F^r(A)|)}{4}\right)|B^r_{F}(A)|\ge \left(1+\frac{1}{(\log\log n)^3}\right)|B_F^r(A)|.
  \]
Therefore,
	\begin{equation}\label{eq-growgrowgrow}
		|B^{\ell_0}_{F}(A)|\ge \left(1+\frac{1}{(\log\log n)^3}\right)^{\ell_0}|A|\geq \exp\left(\frac{\ell_0}{2(\log\log n)^3}\right)=\exp\left(\frac{(\log\log n)^2}{2}\right)=\omega(m^{400k}),
	\end{equation}
    which contradicts our assumption that \emph{\ref{exp1}} does not hold as $n\geq d_0(\eps_1,\eps_2,k)$ is large. For \emph{\ref{exp1}}, it is left then to prove Claim~\ref{sizeclaim}.

\begin{poc} We prove this by induction on $r$.
    For the base case $r=0$, as $x\cdot \eps(x)$ increases in $x\geq \eps_2d/2$, we have
\[
|A|\eps (|A|)\geq \frac{\eps_2d}{2}\cdot \eps(\eps_2d/2)= \frac{\eps_1\eps_2d}{2\log^2(15/2)}\geq 4k,
\]
where we have used that $d\geq d_0(\eps_1,\eps_2,k)$ is large.
In combination with \emph{\ref{sm3}}, we have $|N_H(A)\cap Z|\le k \le |A|\ep(|A|)/4$. Thus, from the expansion of $H$ and \emph{\ref{sm1}}, we have
    $$|N_F(A)|\ge |N_H(A)|-|N_H(A)\cap X|-|N_H(A)\cap Z|\ge \left(1-\frac{1}{4}-\frac{1}{4}\right)|A|\ep(|A|)
    \ge\frac{1}{4}|A|\ep(|A|).$$

    Suppose then that $r\ge 1$ and that \eqref{equa1} holds with $r$ replaced by $r'$ for each $0\leq r'<r$. Then, similarly to~\eqref{eq-growgrowgrow},
    \begin{equation}|B^{r}_{F}(A)|\ge |A|\left(1+\frac{\ep(|B^{r}_{F}(A)|)}{4}\right)^{r}.\label{parteq}\end{equation}
    Now, let $\alpha$ be defined by $|B^r_{F}(A)|=\al \ep_2 d/15$. As $|B^r_F(A)|\ge|A|\ge\ep_2d/2$, $\alpha\geq 15/2$, and thus $\ep(|B^r_{F}(A)|)=\ep_1/\log^2\al\leq 1/2$.
    Therefore, from \eqref{parteq}, we have
      $$\al\ge \frac{|B_F^r(A)|}{|A|} \ge \left(1+\frac{\ep_1}{4\log^2\al}\right)^{r}\geq \exp\left(\frac{\ep_1 r}{8\log^2\al}\right),$$
which gives $r\le 8\log^3\al/\eps_1$. As $\alpha \geq 15/2$, we have $\log^5\al/\al\leq 100$, so that
\begin{equation}\label{lastmin}
r+1\leq 2r\leq \frac{1600\alpha}{\eps_1\log^2\alpha}=\frac{1600\alpha\cdot  \eps(|B^r_F(A)|)}{\eps_1^2}=\frac{1600\cdot 15}{\ep_2d\cdot \eps_1^2}\cdot|B^r_F(A)|  \eps(|B^r_F(A)|).
\end{equation}
As $d\geq d_0(\ep_1,\eps_2,k)$ is large, by \emph{\ref{sm3}}, we have
 \begin{equation}\label{newneweq}
 |N_H(B^r_{H-Z}(A))\cap Z)|\leq k(r+1)\overset{\eqref{lastmin}}{\leq} k \cdot \frac{1600\cdot 15}{\ep_2d\cdot \eps_1^2}\cdot|B^r_F(A)|  \eps(|B^r_F(A)|)\leq \frac{1}{4}|B_F^r(A)|\ep(|B_F^r(A)|).
 \end{equation}
As $x\cdot \eps(x)$ increases with $x\geq \ep_2d/2$, by \emph{\ref{sm1}} we have $|X|\le |A|\ep(|A|)/4\leq |B_F^r(A)|\ep(|B_F^r(A)|)/4$. Therefore, using the expansion of $H$, we have
    \begin{eqnarray*}
    	|N_{F}(B^r_{F}(A))|&\ge & |N_H(B^r_F(A))|-|N_H(B^r_F(A))\cap X)|-|N_H(B^r_{H-Z}(A))\cap Z)|\\
    	&\geq & |N_H(B^r_F(A))|-|X|-|N_H(B^r_{H-Z}(A))\cap Z)|\overset{\eqref{newneweq}}{\ge} \frac{1}{4}|B_F^r(A)|\cdot \ep(|B_F^r(A)|),
    \end{eqnarray*}
and hence~\eqref{equa1} holds for $r$.
\end{poc}

    We now prove \emph{\ref{exp2}}. Suppose, for contradiction, that $|N_{H-X-Y-Z}^m(A)|\leq n/2$. Let $F'=H-X-Y-Z$, and for each $\ell_0\leq r\leq  m-1$, let $A_r=B^r_{F'}(A)$, so that, by \emph{\ref{exp1}}, we have $|A_r|\geq m^{400k}$.
Now, as $n\geq d_0(\ep_1,\eps_2,k)$ is large,
\begin{equation}\label{epseqn}
\eps(|A_r|)\geq \eps(n)\geq \eps_1/\log^2 n\geq 1/m,
\end{equation}
and thus $|A_r|\eps(|A_r|)/4\geq m^{400k-1}/4\geq k m^{300k}$. Therefore, by \emph{\ref{sm2}}, $|A_r|\eps(|A_r|)/4\geq |Y|$ and,
by \emph{\ref{sm3}},
\begin{equation*}
|N_H(A_r)\cap Z|\leq |N_H(B_{H-Z}^r(A))\cap Z|\leq k (r+1)\leq k m\leq |A_r|\eps(|A_r|)/4.\label{Zeq}
\end{equation*}
As $|A_r|\geq |A|$ and $x\cdot \eps(x)$ increases with $x\geq \eps_2d/2$, we have, by \emph{\ref{sm1}}, that $ |A_r|\eps(|A_r|)/4\geq |X|$. In total, then,
    \begin{equation}
 |X|+|Y|+|N_H(A_r)\cap Z|\leq \frac{3}{4}|A_r|\eps(|A_r|).\label{XYeq}
\end{equation}
Thus, using the expansion of $H$, we have
    \begin{eqnarray*}
      |N_{F'}(A_r)|\ge |N_H(A_r)|-|X\cup Y|-|N_H(A_r)\cap Z)| \overset{\eqref{XYeq}}{\ge}\frac{1}{4}|A_r|\cdot \ep(|A_r|)\overset{\eqref{epseqn}}{\geq} \frac{\eps_1|A_r|}{4\log^2 n}.
    \end{eqnarray*}
As this holds for all $r$ with $\ell_0\leq r\leq m-1$, we have
\begin{align*}
|B_{F'}^m(A)|&
\geq \left(1+\frac{\eps_1}{4\log^2 n}\right)^{m-\ell_0}|A_{\ell_0}|\geq \left(1+\frac{\eps_1}{4\log^2 n}\right)^{m/2}
\\
&\ge \exp\left(\frac{\eps_1m}{16\log^2n}\right)=\exp(\log n)>n/2,
\end{align*}
a contradiction.
\end{proof}


\subsection{Connecting sets with paths}\label{sec:newpath}
Lemma~\ref{limit-contact} allows us to find paths between sets $A$ and $B$ in an expander, as follows.

\begin{lemma}\label{new-connect} For each $0<\eps_1,\eps_2<1$, there exists $d_0=d_0(\eps_1,\eps_2)$ such that the following holds for each $n\geq d\geq d_0$ and $x\geq 1$. Let $G$ be an $n$-vertex $(\ep_1,\ep_2d)$-expander with $\delta(G)\geq d-1$.

Let $A,B\subseteq V(G)$ with $|A|,|B|\geq x$, and  let $W\subseteq V(G)\setminus(A\cup B)$ satisfy $|W|\log^3n\leq 10x$. Then, there is a path from $A$ to $B$ in $G-W$ with length at most $\frac{40}{\eps_1}\log^3n$.
\end{lemma}
\begin{proof}
If $x\geq \eps_2d/2$, then, as $n\geq d\geq d_0(\eps_1,\eps_2)$ is large, we have
\begin{equation}\label{neq1}
x\cdot \frac{\eps(x)}{4}= \frac{\eps_1x}{4\log^2(15x/\eps_2d)}\geq \frac{\eps_1x}{4\log^2(15n)}\geq \frac{\eps_1x}{8\log^2n}\geq \frac{10x}{\log^3n}\geq |W|.
\end{equation}
Therefore, letting $m=\frac{16}{\eps_1}\log^3n$, by Lemma~\ref{limit-contact} applied with $X=W$ and $Y=Z=\varnothing$, we have $|B^m_{G-W}(A)|,|B^m_{G-W}(B)|>n/2$. Thus, there is a path from $A$ to $B$ in $G-W$ with length at most $2m\leq \frac{40}{\eps_1}\log^3n$.

Suppose then that $x<\eps_2d/2\leq d/2$. Let $x'=\min\{|A\cup N_{G-W}(A)|,|B\cup N_{G-W}(B)|\}$. As $x< d/2$ and $A,B\neq \varnothing$, we have $x'\geq \delta(G)-|W|\geq d-1-10d/\log^3n\geq d/2\geq \eps_2d/2$, as $n\geq d\geq d_0(\eps_1,\eps_2)$ is large.
Therefore, as above, there is a path from $A\cup N_{G-W}(A)$ to $B\cup N_{G-W}(B)$ in $G-W$ with length at most $2m$, and hence a path with length at most $2m+2\leq \frac{40}{\eps_1}\log^3n$ between $A$ and $B$ in $G-W$, as required.
\end{proof}


\subsection{Expansion of sets of lower degree vertices}\label{sec:techinno}
The following lemma is the key new technicality that allows our constructions, as discussed in Section~\ref{sec:techhigh}. We then develop it for convenience of use to get Lemma~\ref{lem-newbit-other}.

\begin{lemma}\label{lem-newbit}
  For any $0<\eps_1,\eps_2<1$, there exists $d_0=d_0(\eps_1,\eps_2)$ such that the following holds for each $n\geq d\geq d_0$. Suppose that $G$ is an  $n$-vertex bipartite  $(\eps_1,\eps_2 d)$-expander with $\de(G)\ge d$.

  Let $U\subseteq V(G)$ satisfy $|U|\leq \exp((\log\log n)^2)$, and let $K=G-U$. Let $I$ be any set and $V_i\subseteq V(K)$, $i\in I$, be pairwise disjoint sets such that, for each $i\in I$,
  \stepcounter{propcounter}
  \begin{enumerate}[label = \emph{\bfseries \Alph{propcounter}\arabic{enumi}}]
  	\item $\eps_2 d\le |V_i|\le \exp((\log\log n)^2)$,\label{z1}
  	  	\item $|N_{K}(V_i)|\leq \frac{5|V_i|}{\log^{10} |V_i|}$, and\label{z2}
  	  	  	\item $d_G(v,U)\leq d/2$ for each $v\in V_i$.\label{z3}
  \end{enumerate}

Then, $|\cup_{i\in I}V_i|<n^{1/8}$.
\end{lemma}
\begin{proof}
Suppose to the contrary that $|\cup_{i\in I}V_i|\geq n^{1/8}$ and let $D=\exp((\log\log n)^2)$. Let $I_1=\{i\in I:|V_i|\geq (\log n)^{1/10}\}$ and $I_2=I\setminus I_1$.

First, suppose that $|\cup_{i\in I_1}V_i|\geq n^{1/8}/2$. Then, as $|V_i|\le D$ for each $i\in I$, $|I_1|\ge n^{1/9}\geq D^2$. Let $I_0\subseteq I_1$ be a subset of $I_1$ of size $D^2$, and let $W=\cup_{i\in I_0}V_i$. Now, for each $i\in I_0$, $\log |V_i|\geq (\log\log n)/10$, while $|W|\le |I_0|D\le D^3$, so that $\log|W|\le 3(\log\log n)^2$. Thus, we have
\begin{align*}
|N_G(W)|&\leq |U| +|N_K(\cup_{i\in I_0}V_i)|\leq D+\sum_{i\in I_0}|N_K(V_i)|\overset{\emph{\ref{z2}}}{\leq} D +\sum_{i\in I_0}\frac{5|V_i|}{\log^{10}|V_i|} \\
&\leq D +\sum_{i\in I_0}\frac{5(10)^{10}|V_i|}{(\log\log n)^{10}}= D +\frac{5(10)^{10}|W|}{(\log\log n)^{10}}< \frac{\eps_1|W|}{\log^3|W|}\leq \eps(|W|)|W|,
\end{align*}
as $|W|\geq D^2$, $\log |W|\le 3(\log\log n)^2$ and $n\geq d_0(\eps_1,\eps_2)$ is large. As $|W|\leq D^3$, this contradicts the fact that $G$ is an $(\eps_1,\eps_2 d)$-expander. Thus, we have $|\cup_{i\in I_1}V_i|\leq n^{1/8}/2$.

Therefore, we have $|\cup_{i\in I_2}V_i|\geq n^{1/8}/2$. Thus, $I_2\neq \varnothing$, and hence by definition, taking any $i\in I_2$, we have $\eps_2 d\leq |V_i|\leq (\log n)^{1/10}$ so that $d\leq (\log n)^{1/10}/\eps_2$. Furthermore, by the pigeonhole principle there must be some $r\in\bN$ with $\ep_2d\leq r\leq (\log n)^{1/10}$ for which there are at least $| I_2|/(\log n)^{1/10}\ge |\cup_{i\in I_2}V_i|/((\log n)^{1/10})^2\geq n^{1/9}$ indices $i\in I_2$ with $|V_i|=r$. Taking such an $r$, let $I_3=\{i\in I_2:|V_i|=r\}$, so that $|I_3|\ge n^{1/9}$.

Now, for each $i\in I_3$, as $d_G(v,U)\leq d\leq r/\ep_2$ for each $v\in V_i$, we have $|N_G(V_i)\cap U|\leq r^2/\eps_2\leq (\log n)^{1/4}$, as $n\geq d_0(\ep_1,\eps_2)$ is large. As $|U|\leq D$, the number of sets of size at most $(\log n)^{1/4}$ in $U$ is at most
$$
\sum_{i=0}^{(\log n)^{1/4}}\binom{D}{i}\leq (\log n)^{1/4}D^{(\log n)^{1/4}}\le \exp((\log n)^{1/3}).
$$
Therefore, there must be at least $n^{1/9}/\exp((\log n)^{1/3})\ge n^{1/10}$ indices $i\in I_3$ for which $N_{G}(V_{i})\cap U$ is the same set, $Z$ say. Taking any $i\in I_3$, note that, by \emph{\ref{z3}}, $|Z|=|N_G(V_i)\cap U|\leq dr\leq r^2/\eps_2$. Let $I_4$ be a set of $r^2\leq (\log n)^{1/5}$ indices $i\in I_3$ for which $N_G(V_i)\cap U=Z$.

Let $Y=\cup_{i\in I_4}V_i$, so that $N_G(Y)\cap U=Z$ and $|Y|=r|I_4|=r^3$. Then, as $d\geq d_0(\eps_1,\eps_2)$ is large and $r\geq \eps_2d$, we have
$$
|N_G(Y)|\leq |Z|+\sum_{i\in I_4}|N_K(V_i)|\overset{\emph{\ref{z2}}}{\leq} \frac{r^2}{\ep_2}+\frac{5r^3}{\log^{10}r}\le \frac{\eps_1r^3}{\log^3(r^3)}=\frac{\eps_1|Y|}{\log^3|Y|}< \eps(|Y|)|Y|,
$$
contradicting that $G$ is an $(\eps_1,\eps_2d)$-expander.
\end{proof}

Lemma~\ref{lem-newbit} is used three times, each in a similar situation, so for its application, we prove Lemma~\ref{lem-newbit-other} below. In this lemma, with $r=n^{1/8}$, we have sets $A_i$, $i\in [r]$, and wish to find some set $A_j$ which expands while avoiding some set $B_j\cup C_j$ which depends on $j$, as well as avoiding some large common set $U$. To find such an $A_j$, we assume for contradiction that no such set $A_j$ exists, before recording as $V_i$ the first ball around $A_i$ which does not expand nicely. Applying Lemma~\ref{lem-newbit} to the sets $V_i$, $i\in [r]$, will then reach a contradiction.

To prove the lemma, we will also use the following very simple proposition.
\begin{prop}\label{propeasy}
For each $i,\ell\geq 1$, we have $\ell^{2^{-i}}-(\ell-1)^{2^{-i}}\leq \ell^{-1+2^{-i}}$.
\end{prop}
\begin{proof} Not that this is true for any $\ell$ with equality if $i=0$. Assume then for induction that $i>0$ and it is true for $\ell$ with $i$ replaced with $i-1$. We have
\[
\ell^{-1+2^{-(i-1)}}\geq \ell^{2^{-(i-1)}}-(\ell-1)^{2^{-(i-1)}}=
(\ell^{2^{-i}}+(\ell-1)^{2^{-i}})(\ell^{2^{-i}}-(\ell-1)^{2^{-i}})\geq \ell^{2^{-i}}
(\ell^{2^{-i}}-(\ell-1)^{2^{-i}}),
\]
and therefore $\ell^{2^{-i}}-(\ell-1)^{2^{-i}}\leq \ell^{-1+2^{-(i-1)}-2^{-i}}=\ell^{-1+2^{-i}}$, as required.
\end{proof}

\begin{lemma}\label{lem-newbit-other}
  For each $0<\eps_1<1$, $0<\eps_2<1/5$ and $k\in \N$, there exists $d_0=d_0(\eps_1,\eps_2,k)$ such that the following holds for each $n\geq d\geq d_0$.
   Suppose that $G$ is an $n$-vertex bipartite  $(\eps_1,\eps_2 d)$-expander with $\de(G)\ge d$.
  Let $U\subseteq V(G)$ satisfy $|U|\leq \exp((\log\log n)^2)$. Let $r=n^{1/8}$ and $\ell_0=(\log\log n)^{20}$. Suppose $(A_i,B_i,C_i)$, $i\in [r]$, are such that the following hold for each $i\in [r]$.
\stepcounter{propcounter}
\begin{enumerate}[label = \emph{\bfseries \Alph{propcounter}\arabic{enumi}}]
\item $|A_i|\geq d_0$.\label{mouse0}
\item $B_i\cup C_i$ and $A_i$ are disjoint sets in $V(G)\setminus U$, with $|B_i|\leq |A_i|/\log^{10}|A_i|$.\label{mouse1}
\item $A_i$ has $4$-limited contact with $C_i$ in $G-U-B_i$.\label{mouse2}
\item Each vertex in $B_{G-U-B_i-C_i}^{\ell_0}(A_i)$ has at most $d/2$ neighbours in $U$.\label{mouse3}
\item For each $j\in [r]\setminus\{i\}$, $A_i$ and $A_j$ are at least a distance $2\ell_0$ apart in $G-U-B_i-C_i-B_j-C_j$.\label{mouse4}
\end{enumerate}

Then, for some $i\in [r]$, $|B^{\ell_0}_{G-U-B_i-C_i}(A_i)|\geq \log^kn$.
\end{lemma}
\begin{proof} Suppose, for contradiction, that $|B^{\ell_0}_{G-U-B_i-C_i}(A_i)|< \log^kn$ for each $i\in [r]$.
Let $\alpha=1/16$, and note that $\exp(\ell_0^\alpha)= \exp((\log\log n)^{1.25})>\log^kn$, as $n\geq d_0(\ep_1,\ep_2,k)$ is large, and hence $|B^{\ell_0}_{G-U-B_i-C_i}(A_i)|< \exp(\ell_0^\alpha)$ for each $i\in [r]$. Therefore, for each $i\in [r]$, we can let $\ell_i$ be the smallest $\ell\in [\ell_0]$ such that
\begin{equation}\label{eq-VHP-star}
|B^{\ell}_{G-U-B_i-C_i}(A_i)|\leq \exp(\ell^{\alpha}).
\end{equation}
For each $i\in [r]$, let $V_i=B^{\ell_i-1}_{G-U-B_i-C_i}(A_i)$.
By the definition of $\ell_i$, we have that
\begin{equation}\label{eq-VHP-heart}
	|V_i|\geq \exp((\ell_i-1)^{\alpha}) \quad \mbox{ and } \quad |B_{G-U-B_i-C_i}(V_i)|\leq \exp(\ell_{i}^{\alpha}).
\end{equation}

\begin{claim}\label{cl-crackletot} For each $i\in [r]$, we have
	$|N_{G-U}(V_i)|\leq \frac{5|V_i|}{\log^{10}|V_{i}|}$.
\end{claim}
\begin{poc}
	 Fix $i\in [r]$.  By Proposition~\ref{propeasy}, we have $\ell_i^\alpha-(\ell_i-1)^\alpha\leq \ell_i^{-1+\alpha}\leq (\ell_i^\alpha)^{-10}$. Note that, as $d\geq d_0(\ep_1,\eps_2,k)$ is large, we have that $\ell_i\geq \log d_0$ is large by
   \emph{\ref{mouse0}} and~\eqref{eq-VHP-star}.  As $\exp(1/x)-1\le 2/x$ for large $x>0$, we thus have
   \begin{equation}\label{newnneq}
\exp(\ell_i^\alpha-(\ell_i-1)^{\alpha})-1\leq \exp((\ell_i^\alpha)^{-10})-1\leq {2}/(\ell_i^\alpha)^{10}.
   \end{equation}
Then,
   	\begin{align}
   	|N_{G-U-B_i-C_i}(V_i)|&\le |B_{G-U-B_i-C_i}(V_i)|-|V_i|
     =\left(\frac{|B_{G-U-B_i-C_i}(V_i)|}{|V_i|}-1\right)|V_i|\nonumber\\
   	 &\stackrel{\eqref{eq-VHP-heart}}{\le} \left(\frac{\exp(\ell_i^\alpha)}{\exp((\ell_i-1)^{\alpha})}-1\right)|V_i|\stackrel{\eqref{newnneq}}{\leq}\frac{2|V_i|}{(\ell_i^{\alpha})^{10}}\stackrel{\eqref{eq-VHP-heart}}{\le} \frac{2|V_i|}{\log^{10}|V_i|}.\label{crackle}
   	\end{align}

	Now, due to~\emref{mouse2}, we have
	\begin{eqnarray}\label{crackleon}
	|N_{G-U-B_i}(V_i)\cap C_i|\leq 4\ell_i\stackrel{\eqref{eq-VHP-heart}}{\le} 4(\log^{16}|V_i|+1)\le \frac{|V_i|}{\log^{10}|V_i|},
	\end{eqnarray}
  as, by \emref{mouse0}, $|V_i|\geq |A_i|\geq d_0(\eps_1,\eps_2,k)$ is large.
Therefore, as
  $$|N_{G-U}(V_i)|\leq |N_{G-U-B_i-C_i}(V_i)|+|B_i|+|N_{G-U-B_i}(V_i)\cap C_i|,$$
  the claim follows from \eqref{crackle},  \eqref{crackleon}, $|A_i|\leq |V_i|$ and \emref{mouse1}.
\end{poc}

We now check the conditions to appy Lemma~\ref{lem-newbit} to the sets $V_i$, $i\in [r]$. By \emref{mouse4}, the sets $V_i$, $i\in [r]$, are pairwise disjoint. Note that $|V_i|\ge |B_{G-U-B_i-C_i}(A_i)|\ge\eps_2 d$. Indeed, if $|A_i|\ge \eps_2d$, this holds clearly; if $|A_i|<\eps_2d$, then by \emref{mouse1}, \emref{mouse2} and \emref{mouse3}, $|B_{G-U-B_i-C_i}(A_i)|\ge \delta(G)-|B_i|-4-d/2\ge\eps_2d$. Thus, for each $i\in [r]$, $\eps_2d\leq |V_i|<\log^kn\leq \exp((\log\log n)^2)$. By Claim~\ref{cl-crackletot}, if $K=G-U$, then $|N_K(V_i)|\leq 5|V_i|/\log^{10}|V_i|$ for each $i\in [r]$.
By \emref{mouse3}, for each $i\in [r]$ and $v\in V_i$, we have $d_G(v,U)\leq d/2$.
Therefore, by Lemma~\ref{lem-newbit}, we have $r\leq |\cup_{i\in [r]}V_i|<n^{1/8}$, a contradiction.
\end{proof}

\subsection{Disjoint vertex expansions}\label{sec:vxexpansion}
In order to connect structures together in an expander, we typically find the structures we want with an extra subgraph attached to the vertex, $v$ say, we wish to make connections from. This extra subgraph, $F$ say, should have enough vertices that either Lemma~\ref{limit-contact} or Lemma~\ref{new-connect} can be used to connect $V(F)$ to another vertex set while avoiding some structures that we have already found. The graph $F$ should also have a short path from $v$ to any other vertex in $F$, so that a path to $V(F)$ can be extended to one to $v$ without many additional vertices. This motivates the following definition.

\begin{defn}
Given a vertex $v$ in a graph $F$, $F$ is a \emph{$(D,m)$-expansion of $v$} if $|F|=D$ and $v$ is a distance at most $m$ in $F$ from any other vertex of $F$.
\end{defn}

Before we find vertex expansions, we first prove the following simple proposition which finds a smaller expansion within any expansion.

\begin{prop}\label{prop-trimming} Let $D,m\in \N$ and $1\leq D'\leq D$. Then, any graph $F$ which is a $(D,m)$-expansion of $v$ contains a subgraph which is a $(D',m)$-expansion of $v$.
\end{prop}
\begin{proof}
We prove this for each $1\leq D'\leq D$ by induction on $D'$ for $D'=D,D-1,\ldots,1$. Note that $F$ demonstrates this is true for $D'=D$. Suppose then it is true for $D'\geq 2$, and let $F'$ with $|F'|=D'$ be a $(D',m)$-expansion of $v$. Let $w\in V(F')$ maximise the graph distance from $v$ to $w$ in $F'$. As $D'\geq 2$, $v\neq w$. Noting that $F'-w$ is a $(D'-1,m)$-expansion of $v$ completes the proof of the inductive step, and hence the proposition.
\end{proof}

We now give our lemma which finds vertex expansions. Its proof is different according to whether there are many vertices in the graph of high degree (Case I) or not (Case II). We will construct structures using short cycles later, so for the application of this lemma we need to find vertex expansions while avoiding a short cycle as much as possible.

\begin{lemma}\label{lem-expansion}
	For each $k\in \N$ and any $0<\eps_1, \eps_2<1$, there exists $d_0=d_0(\ep_1,\ep_2,k)$ such that the following holds for each $n\geq d\geq d_0$.

  Suppose that $G$ is an $n$-vertex bipartite $(\eps_1,\eps_2 d)$-expander with $\delta(G)\geq d-1$.
	Let $m=\frac{40}{\ep_1}\log^3 n$. Let $C$ be a shortest cycle in $G$, and let $x_1,\ldots,x_k$ be distinct vertices in $G$. For each $i,j\in [k]$, let $D_{i,j}\in [1,\log^{5k}n]$.

    Then, there are graphs $F_{i,j}\subseteq G$, $i,j\in [k]$, such that the following hold.
  \begin{itemize}
    \item For each $i,j\in [k]$, $F_{i,j}$ is a $(D_{i,j},5m)$-expansion around $x_i$ which contains no vertices other than $x_i$ in $V(C)\cup \{x_1,\ldots,x_k\}$.
    \item The sets  $V(F_{i,j})\setminus\{x_i\}$, $i,j\in [k]$, are pairwise disjoint.
\end{itemize}
\end{lemma}
\begin{proof} Note first that, as $\de(G)\ge d-1$, $|C|\le 2\log n/\log (d-1)$.  Let $D= \log^{5k}n$. By Proposition~\ref{prop-trimming}, we can assume that $D_{i,j}=D$ for all $i,j\in [k]$. Let $r=k^2$ and let $L$ be the set of vertices in $G$ with degree at least $\Delta=D^2$. We will split into two cases depending on whether $|L\setminus V(C)|\geq 2r$ or not.

\medskip

\noindent\textbf{Case I:}  First suppose that there are $2r$ vertices $v_1,\ldots,v_{2r}\in L\setminus V(C)$. Assume, by relabelling if necessary, that $V:=\{v_1,\ldots,v_r\}$ is disjoint from $X:=\{x_1,\ldots,x_k\}$.

Let $\cP$ be a maximal collection of paths in $G$ from $X$ to $V$ such that
\begin{itemize}
	\item each path in $\cP$ has length  at most $3m$ and internal vertices in $V(G)\setminus (V(C)\cup X\cup V)$,
	\item the paths in $\cP$ are vertex disjoint outside of $X$, and
	\item there are at most $k$ paths containing each vertex in $X$.
\end{itemize}
Subject to $|\cP|$ being maximal, suppose that $\sum_{P\in \cP}\ell(P)$ is minimised.
Note that, as there are at most $k$ paths containing each vertex in $X$, $|\cP|\le k^2$.

Now, suppose for contradiction that there is some $x\in X$ which is in fewer than $k$ paths in $\cP$, and so $|\cP|< k^2$. Let $U=(V\cup X\cup V(\cP)\cup V(C))\setminus \{x\}$.  For each path $P\in \cP$ and $\ell\in \N$, at most $\ell+1$ vertices in $N_G(B^{\ell-1}_{G-U}(x))$ can lie on $P$, otherwise we can find a shorter path than $P$ from $x$ to $V(P)\cap V$ in $G-(U\setminus V(P))$. Swapping $P$ for this shorter path contradicts the minimality of $\sum_{P'\in \cP}\ell(P')$.
Therefore, for each $\ell\in \N$,
\begin{equation}\label{buzz}
|N_G(B^{\ell-1}_{G-U}(x))\cap V(\cP)|\leq \sum_{P\in \cP}|N_G(B^{\ell-1}_{G-U}(x))\cap V(P)|\leq (\ell+1)|\cP|
\leq (\ell+1)k^2.
\end{equation}
Furthermore, for any vertex $v$ and any integer $\ell\in \N$, at most $2\ell+1$ vertices in $B^{\ell}_G(v)$ can lie on $C$, as $C$ is a shortest cycle in $G$.
Therefore,
\begin{align*}
|N_G(B^{\ell-1}_{G-U}(x))\cap U|&\leq |V\cup X|+|N_G(B^{\ell-1}_{G-U}(x))\cap V(\cP)|+|N_G(B^{\ell-1}_{G-V(C)+x}(x))\cap V(C)|\\
&\overset{\eqref{buzz}}{\leq} 2k^2+(\ell+1)k^2+|B^{\ell}_G(x)\cap V(C)|\leq (\ell+3)k^2+(2\ell+1)\leq 10\ell k^2.
\end{align*}
Thus, $\{x\}$ has $10k^2$-limited contact with $U$ in $G$, so that $B_{G-U}(x)$ has $20k^2$-limited contact with $U$ in $G$.
We also have $|B_{G-U}(x)|\ge \de(G)-10k^2\ge d/2\ge \ep_2d/2$. Therefore, applying Lemma~\ref{limit-contact} with $(A,X,Y,Z,k)_{\ref{limit-contact}}=(B_{G-U}(x),\varnothing,\varnothing, U,20k^2)$, we get that $|B^{m+1}_{G-U}(x)|=|B^{m}_{G-U}(B_{G-U}(x))|>n/2$.

Note that, by the choice of $\cP$, each $P\in \cP$ has a distinct vertex in $V$. As $|\cP|<k^2$, we can choose a vertex $v\in V\setminus V(\cP)$. Note that $|U|\le 2k^2+k^2\cdot (3m+1)+2\log n/\log (d-1)\le \log^4n$.
As $d_G(v)\geq \Delta=\log^{10k}n\ge |U|\log^3n/10$ and $|B^{m+1}_{G-U}(x)|>n/2\ge |U|\log^3n/10$, using Lemma~\ref{new-connect}, we can connect $B^{m+1}_{G-U}(x)$ and $N_G(v)$ with a path of length at most $m$ in $G-U$, which extends in $B^{m+1}_{G-U}(x)\cup \{v\}$ to an $x,v$-path in $G-U$ with length at most $3m$, contradicting the maximality of~$\cP$.

Therefore, each vertex in $X$ is in exactly $k$ paths in $\cP$.
Label the paths in $\cP$ as $P_{i,j}$, $i,j\in [k]$, so that $x_i$ is an endvertex of $P_{i,j}$, and let $v_{i,j}$ be the endvertex of $P_{i,j}$ in $V$. Greedily, using that $|N_{G}(v_{i,j})\setminus(V\cup X\cup V(C)\cup V(\cP))|\geq \Delta-\log^4n\geq k^2D$ for each $i,j\in [k]$, pick disjoint sets $A_{i,j}\subseteq N_{G}(v_{i,j})\setminus(V\cup X\cup V(C)\cup V(\cP))$, $i,j\in [k]$, with size $D-|P_{i,j}|$. Then $F_{i,j}=G[A_{i,j}]\cup P_{i,j}$, $i,j\in[k]$, are easily seen to be the $(D,5m)$-expansions we require.

\medskip

\noindent\textbf{Case II:} Suppose then that $|L\setminus V(C)|< 2r$. Relabelling if necessary, let $0\le k'\leq k$ be such that $\{x_1,\ldots,x_k\}\setminus L=\{x_1,\ldots,x_{k'}\}$. Let $G'=G-L$,  $X=\{x_1,\ldots,x_{k'}\}$, $r'=k'k$ and $\ell_0=2(\log\log n)^5$.

Let $s\leq r'$ be the largest integer for which there are vertices $w_1,\ldots,w_s\in V(G')$ such that
\begin{itemize}
\item the sets $B^{5\ell_0}_{G'}(w_i)$, $i\in [s]$, $X$ and $V(C)\setminus L$ are all pairwise disjoint.
\end{itemize}

Suppose $s<r'$. Then, we must have $V(G')=B^{10\ell_0}_{G'}((\{w_1,\ldots,w_s\}\cup X\cup V(C))\setminus L)$. However, as $\De(G')\le \De=D^2$ and $|C|\leq 2\log n/\log(d-1)\leq 2\log n$,
$$|G'|=|B_{G'}^{10\ell_0}((\{w_1,\ldots,w_{s}\}\cup X\cup V(C))\setminus L)|\le 2\cdot (r'+k'+2\log n)\cdot \De^{10\ell_0}\le\exp((\log\log n)^7)<n/2,$$
contradicting $|G'|\geq n-|L\setminus V(C)|-|C|\geq n-2r-2\log n\geq n/2$. Therefore, $s=r'$.

Now, fixing an arbitrary $i\in [r']$,  similarly to before, for each $\ell\in \N$ at most $2\ell+1$ vertices in $B^{\ell}_G(w_i)$ can lie on $C$, otherwise  there is a shorter cycle in $G$ than $C$. Thus,
\begin{equation}\label{litttle}
|B_{G'- V(C)}(w_i)|\ge \de(G)-|L\setminus V(C)|-3\ge \de(G)-2r-3\ge d/2,
\end{equation}
and, for each $\ell\in \N$,
$$|N_G(B_{G-V(C)}^{\ell-1}(B_{G'- V(C)}(w_i)))\cap V(C)|\le |B^{\ell+1}_G(w_i)\cap V(C)|\le 2\ell+3\le 5\ell.$$
Therefore, $B_{G'-V(C)}(w_i)$ has $5$-limited contact with $V(C)$ in $G$.
Let $z=|B_{G'- V(C)}(w_i)|$, and note that, as $|L\setminus V(C)|\leq 2r=2k^2$ and, by \eqref{litttle}, $z\geq d/2\geq d_0(\ep_1,\eps_2,k)/2$ is large, we have that  $|L\setminus V(C)|\leq \ep(z)z/4$. Thus, by  Lemma~\ref{limit-contact} with $(A,X,Y,Z,k)_{\ref{limit-contact}}=(B_{G'-V(C)}(w_i), L\setminus V(C),\varnothing, V(C), k+5)$ we get that $|B^{\ell_0+1}_{G'-V(C)}(w_i)|\geq D^2=\De$. Hence, by Proposition~\ref{prop-trimming} we can pick a subgraph $F_i\subseteq G'$ induced on a subset of $B^{\ell_0}_{G'-V(C)}(w_i)$ which is a $(\De,2\ell_0)$-expansion of $w_i$.

Now, let $\cP$ be a maximal collection of paths in $G'$ from $X$ to $V:=\cup_{i\in [r']}V(F_i)$ such that
\begin{itemize}
	\item each path in $\cP$ has length  at most $3m$ and internal vertices in $V(G')\setminus (V(C)\cup X\cup V)$,
	\item the paths in $\cP$ are vertex disjoint outside of $X$,
  \item at most one path in $\cP$ has a vertex in $V(F_i)$, for each $i\in [r']$, and
	\item there are at most $k$ paths containing each vertex in $X$.
\end{itemize}
Subject to $|\cP|$ being maximal, suppose that $\sum_{P\in \cP}\ell(P)$ is minimised.

Suppose again there is some $x\in X$ in fewer than $k$ paths in $\cP$ and let $U=(L\cup X\cup V(\cP)\cup V(C))\setminus \{x\}$. As in Case I, by the minimality of $\sum_{P'\in\cP}\ell(P')$, for each path $P\in \cP$ and $\ell\in \N$, at most $\ell+1$ vertices in $N_G(B^{\ell-1}_{G-U}(x))$ can lie on $P$.
Therefore, for each $\ell\in \N$,
\begin{equation*}
|N_G(B^{\ell-1}_{G-U}(x))\cap V(\cP)|\leq \sum_{P\in \cP}|N_G(B^{\ell-1}_{G-U}(x))\cap V(P)|\leq (\ell+1)|\cP|
\leq (\ell+1)k^2.
\end{equation*}
Again, for any integer $\ell\in \N$, at most $2\ell+1$ vertices in $B^{\ell}_G(x)$ can lie on $C$, as $C$ is a shortest cycle in $G$. Thus,
\begin{align*}
|N_G(B^{\ell-1}_{G-U}(x))\cap U|&\leq |X\cup (L\setminus V(C))|+|N_G(B^{\ell-1}_{G-U}(x))\cap V(\cP)|+|B^{\ell}_G(x)\cap V(C)|\\
&\leq k+2r+(\ell+1)k^2+(2\ell+1)\leq 10\ell k^2.
\end{align*}
Therefore, $\{x\}$ has  $10k^2$-limited contact with $U$ in $G$, and so $B_{G-U}(x)$ has $20k^2$-limited contact with $U$ in $G$. Furthermore, by the choice of the $w_i$, $B^{\ell_0}_{G-U}(x)\cap V=\varnothing$, thus
$$|B_{G-U-V}(x)|=|B_{G-U}(x)|\ge d-10k^2\geq d/2.$$
 We also have that $|V|\leq k^2\Delta= k^2\log^{10k}n$.
Therefore,  by Lemma~\ref{limit-contact} with $(A,X,Y,Z,k)_{\ref{limit-contact}}=(B_{G-U}(x), \varnothing, V,U,20k^2)$, $|B^{m+1}_{G-U-V}(x)|>n/2$.

By the choice of $\cP$, each $P\in \cP$ has exactly one vertex in $V(F_i)$ for some $i\in[r']$. As $|\cP|<k'k=r'$, there is some $j\in[r']$ such that $V(\cP)$ has no vertex in $V(F_j)$. Now, the vertices $w_i$, $i\in[r']$, are pairwise at least $10\ell_0$-far in $G'$, and $F_i$ is a $(\De,\ell_0)$-expansion around $w_i$. Therefore, the subgraphs $F_i$, $i\in[r']$, are pairwise at least $8\ell_0$-far from each other in $G'$, so that $F_j$ is at least $8\ell_0$-far from $V\setminus V(F_j)$.

Now, as $|U|\leq |X|+|L\setminus V(C)|+|C|+|V(\cP)|\leq k+2r+2\log n+(3m+1)r\leq \log^4n$, we have $|F_j|=\De\geq m|U|$. As $L\neq V(G)$, we have $\De> \delta(G)\ge \ep_2d$, so that $|F_j|\geq \eps_2d$. Therefore, by Lemma~\ref{limit-contact} with $(A,X,Y,Z,k)_{\ref{limit-contact}}=(V(F_j),U, V\setminus V(F_j),\varnothing,k)$, we have
$$|B^{m}_{G-U-(V\setminus V(F_j))}(V(F_j))|>n/2.$$
Thus, as $|B^{m+1}_{G-U-V}(x)|>n/2$, there is a path from $\{x\}$ to $V(F_j)$ in $G'$ with length at most $3m$ which is internally disjoint from $V(C)\cup X\cup V$. This contradicts the maximality of $\cP$.

Therefore, each vertex in $X$ is in exactly $k$ paths in $\cP$.
Label the paths in $\cP$ with $P_{i,j}$, $i\in[k']$ and $j\in[k]$, and graphs $F_{i'}$, $i'\in [r']$, as $F'_{i,j}$, $i\in[k']$ and $j\in[k]$, so that, for each $i\in[k']$ and $j\in[k]$, $P_{i,j}$ is a path from  $x_i$ to $F'_{i,j}$. Recall that, for a path $P$, we write $\ell(P)$ for its length. Note that $P_{i,j}\cup F'_{i,j}$ is a $(|P_{i,j}\cup F'_{i,j}|,\ell(P_{i,j})+2\ell_0)$-expansion of $x_i$, for each $i\in [k']$ and $j\in [k]$.
For each $i\in[k']$ and $j\in[k]$, apply Proposition~\ref{prop-trimming} to obtain a $(D,5m)$-expansion $F_{i,j}\subseteq P_{i, j}\cup F'_{i,j}$ around $x_i$.

Lastly, for each $k'+1\le i\le k$ and $j\in[k]$, as $x_i\in L$ we can greedily pick pairwise disjoint $(D,1)$-expansions $F_{i,j}$ induced on a subset of $N_{G}(x_{i})\setminus(V(C)\cup (\cup _{i'\in[k'], j'\in[k]}V(F_{i',j'})))$.
\end{proof}

\subsection{Enlarging vertex expansions}\label{sec:largeexpansions}
In this section, we take up to 4 disjoint vertex expansions, and expand them disjointly to get larger vertex expansions around the same vertices (see Lemma~\ref{lem-large-exp}). This enlargement allows us to later connect vertex expansions with very long paths, as our path lengths need to be smaller than the expansions (see Section~\ref{sec:lengthen}).

We first show for Lemma~\ref{cl-egg} that we can always find a linear size vertex set with polylogarithmic diameter in $G$ while avoiding an arbitrary set of up to $\Theta(n/
\log^2n)$ vertices. This is used for Lemma~\ref{lem-large-exp} and Lemma~\ref{lem-conexp}, as well as later in Section~\ref{sec:robadj}

\begin{lemma}\label{cl-egg} For any $0<\eps_1, \eps_2<1$, there exists $d_0=d_0(\eps_1,\eps_2)$ such that the following holds for each $n\geq d\geq d_0$. Suppose that $G$ is an $n$-vertex bipartite $(\eps_1,\eps_2 d)$-expander with $\de(G)\ge d$ and let $m=\frac{50}{\eps_1}\log^3n$.

For any set $W\subseteq V(G)$ with $|W|\leq \eps_1 n/100\log^{2}n$, there is a set $B\subseteq G-W$ with size at least $n/25$ and diameter at most $2m$, and such that $G[B]$ is a $(D,m)$-expansion around some vertex $v\in B$ for $D=|B|$.
\end{lemma}
\begin{proof} Let $\ell_0=\frac{50}{\eps_1}\log^2n$.  Suppose that $G$ and $W\subseteq V(G)$ satisfy the conditions in the lemma and let $G'=G-W$. Take the largest integer $r\leq \log n$ such that there is a set of at most $1+n/(10\cdot 4^r)$ vertices $V\subseteq V(G')$ with $|B^{\ell_0 r}_{G'}(V)|\geq n/25$. Note that such a set of vertices exists for $r=0$, as $|G'|\geq n-\eps_1 n/100\log^2n\geq n/25$ and $n\geq d_0(\ep_1,\ep_2)$ is large. Suppose, for contradiction, that $|V|>1$, and hence, that $r\le \log  n-1$.

Let $A=B^{\ell_0 r}_{G'}(V)$. As $|W|\leq \eps_1 n/100\log^{2}n\leq \ep_1|A|/4\log^2n$,  for each $\ell$ with $|B^{\ell}_{G'}(A)|< n/2$, we have, by the expansion property of $G$, and as $\eps(|B^{\ell}_{G'}(A)|)\geq \eps(n)\ge \eps_1/\log^2n$,
$$
|N_{G'}(B_{G'}^{\ell}(A))|\ge |N_{G}(B_{G'}^{\ell}(A))|-|W|\geq \frac{\eps_1}{\log^2n}\cdot |B_{G'}^{\ell}(A)|-\frac{\ep_1|A|}{4\log^2n}\geq \frac{\eps_1}{2\log^2n}\cdot |B_{G'}^\ell(A)|,
$$
so that $|B^{\ell+1}_{G'}(A)|\geq (1+\eps_1/2\log^2n)|B^\ell_{G'}(A)|$. If $|B^{\ell_0}_{G'}(A)|< n/2$, then
$$
|B^{\ell_0}_{G'}(A)|\geq \left(1+\frac{\eps_1}{2\log^2n}\right)^{\ell_0}|A|\geq \frac{\eps_1\ell_0}{2\log^2n}\cdot\frac{n}{25}\geq n/2,
$$
a contradiction. Therefore,  we have
  $$
  |B^{\ell_0(r+1)}_{G'}(V)|=|B^{\ell_0}_{G'}(A)|\ge n/2.
  $$
  Consequently, by averaging, there exists a set of at most $\lceil|V|/12\rceil$
vertices $V'\subseteq V$ such that $|B^{\ell_0(r+1)}_{G'}(V')|\ge  |B^{\ell_0(r+1)}_{G'}(V)|/12\ge n/25$. Noting that, as $|V|\geq 2$,
$$
\lceil|V|/12\rceil \leq 1+(|V|-1)/12\leq 1+n/(10\cdot 4^{r+1}),
$$
this contradicts the maximality of $r$.

  Therefore, we have $|V|=1$. That is, there is some vertex $v\in V(G')$ with $|B_{G'}^{\ell_0r}(v)|\geq n/25$. Letting $B=B_{G'}^{\ell_0 r}(v)$, we have that $|B|\geq n/25$ and $B$ has diameter at most $2\ell_0r\leq \frac{100}{\ep_1}\log^3 n=2m$, as required, noting that $G[B]$ is a $(D,m)$-expansion around $v$ for $D=|B|$.
\end{proof}

The large vertex sets with small diameter found in Lemma~\ref{cl-egg} are large vertex expansions around some vertex. To find large vertex expansions around particular vertices, which themselves sit in smaller vertex expansions, we take disjointly many large vertex expansions, then expand and connect the smaller vertex expansions to the larger ones, for the following lemma.

\begin{lemma}\label{lem-large-exp}
For any $0<\eps_1, \eps_2<1$, there exists $d_0=d_0(\eps_1,\eps_2)$ such that the following holds for each $n\geq d\geq d_0$. Suppose that $G$ is an $n$-vertex bipartite $(\eps_1,\eps_2 d)$-expander with $\de(G)\ge d$.

Let $\log^{10}n\leq D\leq n/\log^{10}n$ and $m=\frac{100}{\eps_1}\log^3n$.
Let $A\subseteq V(G)$ satisfy $|A|\leq D/\log^3n$.
Let $F_1,\ldots,F_4\subseteq G-A$ be vertex disjoint subgraphs and $v_1,\ldots,v_4$ be vertices such that, for each $i\in [4]$,  $F_i$ is a $(D,m)$-expansion of $v_i$.
Then, $G-A$ contains vertex disjoint subgraphs $F'_1,\ldots,F'_4$ such that, for each $i\in [4]$, $F'_i$ is an $(n/m^2,3m)$-expansion of $v_i$.
\end{lemma}
\begin{proof}
Let $W=V(F_1)\cup \ldots\cup V(F_4)\cup A$ so that $|W|\le 5D\leq 5n/\log^{10}n$. Applying Lemma~\ref{cl-egg} iteratively $32m$ times, we can find disjoint sets $B_1,\ldots,B_{32m}$ in $G-W$ such that each $B_i$ has size $n/m^2$ and diameter at most $m$. Note that this is possible, as after we have found $B_i$, where $i\in [32m]$,
we have $W\cup(\cup_{j\in [i]}B_j)\leq 5n/\log^{10}n+32m\cdot n/m^2\leq \eps_1n/100\log^2n$, as $n\geq d_0(\eps_1,\eps_2)$ is large. Lemma~\ref{cl-egg} and Proposition~\ref{prop-trimming} then shows that $G-W\cup(\cup_{j\in [i]}B_j)$ contains a set with $n/m^2$ vertices and diameter at most $m$.

Next, take a maximal set $I\subseteq [4]$ such that there are paths $P_{i,j}$, with $i\in I$ and $j\in[8m]$, and distinct $k_{i,j}\in [32m]$, satisfying the following.
\stepcounter{propcounter}
\begin{enumerate}[label = {\bfseries \Alph{propcounter}\arabic{enumi}}]
	\item $P_{i,j}$ is a path from $v_i$ to $B_{k_{i,j}}$ of length at most $2m$.\label{seal1}

	\item The sets $V(P_{i,j})\setminus V(F_i)$ are vertex disjoint across $i\in I$ and $j\in [8m]$.\label{seal2}

	\item There is an ordering $\sigma$ on $I$ such that $V(P_{i,j})$ is disjoint from $\cup_{i'\in [4]\setminus I'} V(F_{i'})$, where $I'=\{i'\in I:~\sigma(i')\le\sigma(i)\}$.\label{seal3}
\end{enumerate}

Suppose, for contradiction, that $J:=[4]\setminus I\neq \varnothing$, and let $\sigma:I\to [|I|]$ be an ordering for which \ref{seal3} holds.  Let $W'=A\cup (\cup_{i\in I,j\in [8m]}V(P_{i,j}))$. Take a maximal set $K\subseteq [32m]\setminus \{k_{i,j}:i\in I,j\in [8m]\}$ for which there are paths $P_{k}$, $k\in K$, with length at most $m$ from $\cup_{i\in J}V(F_i)$ to $B_{k}$, which avoid $W'$ and which are vertex disjoint.  Suppose, to contradict the maximality of $K$, that there is some $j'\in [32m]\setminus \{k_{i,j}:i\in I,j\in [8m]\}$. Note that $|\cup_{i\in J}V(F_i)|\geq D$ and $|B_{j'}|=n/m^2\geq D$.
Noting that $|W'\cup (\cup_{k\in K}V(P_k))|\leq D/\log^3n+32m(2m+1)\leq 2D/\log^3n$, by Lemma~\ref{new-connect}, there is a path between $\cup_{i\in J}V(F_i)$ and $B_{j'}$ with length at most $m$ which avoids $W'\cup (\cup_{k\in K}V(P_k))$, contradicting the maximality of $K$.

Therefore, we have $K=[32m]\setminus \{k_{i,j}:i\in I,j\in [8m]\}$, and hence $|K|=32m-8m|I|=8m|J|$. Consequently, for some $i'\in J$, there are at least $8m$ values of $k\in K$ for which $P_{k}$ has a vertex in $V(F_{i'})$. Taking $k_{i',1},\ldots,k_{i',8m}$ to be distinct such values of $k$, for each $j\in [8m]$, as $F_{i'}$ is a $(D,m)$-expansion of $v_{i'}$, we can find a path $P_{i',j}\subseteq P_{k_{i',j}}\cup F_{i'}$ from $v_{i'}$ to $B_{k_{i',j}}$ with length at most $2m$. The paths $P_{i,j}$, $i\in I\cup\{i'\}$ and $j\in [8m]$, satisfy \ref{seal1}--\ref{seal3} (the last with the ordering $\sigma$ extended by setting $\sigma(i')=|I|+1$), contradicting the maximality of $I$.

Thus, we have $I=[4]$. By relabelling if necessary, assume from \ref{seal3} that $V(P_{i,j})$ is disjoint from $V(F_{i'})$ for each $i'>i$ and $j\in [8m]$.
Now, for each $1\leq i\leq 4$, greedily select $r_i\in [8m]$ in turn such that $V(P_{i,r_i})\cup B_{k_{i,r_i}}$ has no vertices in $\cup_{i'<i}P_{i',r_{i'}}$. Note that this is possible as, for each $i\in [4]$, $\cup_{i'<i}V(P_{i',r_{i'}})$ contains at most $3(2m+1)$ vertices, none of which are in $V(F_i)$, and the sets $V(P_{i,j})\cup B_{k_{i,j}}$, $j\in [8m]$ are disjoint outside of $V(F_i)$.

For each $i\in [4]$, let $F''_i=P_{i,r_i}\cup G[B_{k_{i,r_i}}]$. Note that the subgraphs $F''_i$, $i\in [4]$, are vertex disjoint, and, as each set $B_i$, $i\in[32m]$, has diameter at most $m$, for each $i\in [4]$, $F''_i$ is a $(D_i,3m)$-expansion of $v_i$ for some $D_i\geq n/m^2$. For each $i\in [4]$, using Proposition~\ref{prop-trimming}, let $F'_i\subseteq F''_i$ be an $(n/m^2,3m)$-expansion of $v_i$, completing the proof.
\end{proof}

\subsection{Long paths between vertex expansions}\label{sec:lengthen}

Our goal now is, given some desired path length, to connect two vertices with an initial path with close to this desired length. Section~\ref{sec:mainthmpf} then makes the fine adjustment to this path so that it has exactly the desired length. Given vertex expansions around the vertices to be connected, we find such an initial path, as follows.

\begin{lemma}\label{lem-conexp}
	For any $0<\eps_1, \eps_2<1$, there exists $d_0=d_0(\eps_1,\eps_2)$ such that the following holds for each $n\geq d\geq d_0$. Suppose that $G$ is an $n$-vertex bipartite $(\eps_1,\eps_2 d)$-expander with $\de(G)\ge d$.

	Let $\log^3n \leq D\leq n/\log^{4}n$ and $\frac{300}{\ep_1}\log^3n\le m\le 3\log^4n$. Suppose $F_1,F_2$ are vertex disjoint $(D,m)$-expansions of vertices $v_1,v_2\in V(G)$ respectively. Suppose $W\subseteq V(G)\setminus (V(F_1)\cup V(F_2))$ satisfies $|W|\leq D/\log^3n$. Then, for any $\ell\leq D/\log^3n$, there is a $v_1,v_2$-path in $G-W$ with length between $\ell$ and $\ell+5m$.
\end{lemma}
\begin{proof}
Let $(P_1,v_3,F_3,P_2,v_4,F_4)$ be such that $\ell(P_1)+\ell(P_2)$ is maximised subject to the following properties.
\stepcounter{propcounter}
\begin{enumerate}[label = {\bfseries \Alph{propcounter}\arabic{enumi}}]
\item For each $i\in [2]$, $P_i$ is a $v_i,v_{i+2}$-path in $G-W$.\label{ch1}
\item $\ell(P_1)+\ell(P_2)\leq \ell+2m$.\label{ch2}
\item For each $i\in \{3,4\}$, $F_i$ is a $(D,m)$-expansion of $v_i$ in $G-W$ with $V(F_i)\cap V(P_{i-2})=\{v_i\}$.\label{ch3}
\item $V(P_1\cup F_3)$ and $V(P_2\cup F_4)$ are vertex disjoint.\label{ch4}
\end{enumerate}
Note that $P_1=G[\{v_3\}]$, $v_3=v_1$, $F_3=F_1$, $P_2=G[\{v_4\}]$, $v_4=v_2$, $F_4=F_2$  satisfy \ref{ch1}--\ref{ch4}, and therefore such a sextuple $(P_1,v_3,F_3,P_2,v_4,F_4)$ exists.

We claim that $\ell(P_1)+\ell(P_2)\geq \ell$. Suppose for contradiction that $\ell(P_1)+\ell(P_2)< \ell$. Note that $|W\cup V(F_3\cup F_4\cup P_1\cup P_2)|\leq 3D+\ell\leq n/\log^3n$. By Lemma~\ref{cl-egg}, then, $G-W-V(F_3\cup F_4\cup P_1\cup P_2)$ contains a set, $B$ say, with size at least $D$ and diameter at most $m$. Note furthermore that $|W\cup V(P_1)\cup V(P_2)|\leq D/\log^3 n+\ell+2\leq 10D/\log^3n$. By Lemma~\ref{new-connect}, there is a path, $Q'$ say, from $B$ to $V(F_3)\cup V(F_4)$ which avoids $(W\cup V(P_1)\cup V(P_2))\setminus\{v_3,v_4\}$ and has length at most $m$. Say without loss of generality that $Q'$ has endvertices $v_3''\in V(F_3)$ and $v_3'\in B$. By~\ref{ch3} and~\ref{ch4}, we can extend $Q'$ in $F_3$ to a $v_3,v_3'$-path $Q$ with length at most $2m$ which is  vertex disjoint from $P_2$ and $P_1-v_3$.

Using Proposition~\ref{prop-trimming}, let $F_3'\subseteq G[B]$ be a $(D,m)$-expansion around $v_3'$. Let $P_1'=P_1\cup Q$ and note that, as $v_3'\in B$ and $\ell(Q)\leq 2m$, this is a $v_1,v_3'$-path with length at least $\ell(P_1)+1$ and at most $\ell(P_1)+2m$. Then, $P_1',v_3',F_3',P_2,v_4,F_4$ satisfy \ref{ch1}--\ref{ch4} with $P_1',v_3',F_3'$ in place of $P_1,v_3,F_3$, and $\ell(P_1')+\ell(P_2)>\ell(P_1)+\ell(P_2)$, a contradiction. Therefore, $\ell(P_1)+\ell(P_2)\geq \ell$.

Now, as $|W\cup V(P_1)\cup V(P_2)|\leq 10D/\log^3 n$, by Lemma~\ref{new-connect} there is a path, $R$ say, from some $r_1\in V(F_3)$ to some $r_2\in V(F_4)$ avoiding $(W\cup V(P_1)\cup V(P_2))\setminus\{v_3,v_4\}$ with length at most $m$. For each $i\in [2]$, let $Q_i$ be a path from $v_{i+2}$ to $r_i$ in $F_{i+2}$ with length at most $m$.
 Then, $P_1\cup Q_1\cup R\cup Q_2\cup P_2$ is a $v_1,v_2$-path in $G-W$ with length at least $\ell(P_1)+\ell(P_2)\geq \ell$ and at most, by \ref{ch2}, $\ell+2m+3m\leq \ell+5m$.
\end{proof}

Combining Lemma~\ref{lem-conexp} with our result on extending vertex expansions, we now convert Lemma~\ref{lem-conexp} into the precise form we apply later. The following corollary finds not one but two paths, whose combined length is close to some desired length. We later apply it so that these two paths connect two vertices with a string of simple adjusters in the middle.

\begin{cor}\label{longconnect4}
For any $0<\eps_1, \eps_2<1$, there exists $d_0=d_0(\eps_1,\eps_2)$ such that the following holds for each $n\geq d\geq d_0$. Suppose that $G$ is an $n$-vertex bipartite $(\eps_1,\eps_2 d)$-expander with $\de(G)\ge d$.

Let $\log^{10}n\leq D\leq n/\log^{10}n$, $\frac{100}{\ep_1}\log^3n\le m\le \log^4n$ and $\ell\leq n/\log^{12}n$. Let $A\subseteq V(G)$ satisfy $|A|\leq D/\log^3n$. Let $F_1,\ldots,F_4\subseteq G-A$ be vertex disjoint subgraphs and $v_1,\ldots,v_4$ be vertices such that, for each $i\in [4]$, $F_i$ is a $(D,m)$-expansion of $v_i$.

Then, $G-A$ contains vertex disjoint paths $P$ and $Q$ with $\ell \leq \ell(P)+\ell(Q)\leq \ell+22m$ such that both $P$ and $Q$ connect $\{v_1,v_2\}$ to $\{v_3,v_4\}$.
\end{cor}
\begin{proof}
By Lemma~\ref{lem-large-exp}, $G-A$ contains disjoint subgraphs $F'_1,\ldots,F'_4$ such that, for each $i\in [4]$, $F'_i$ is an $(n/m^2,3m)$-expansion of $v_i$. By Lemma~\ref{new-connect}, there is a path $P'\subseteq G-A$ from $V(F'_1)\cup V(F'_2)$ to $V(F'_3)\cup V(F'_4)$ with length at most $m$. Note that we can assume, without loss of generality, that $P'$ goes from $V(F'_1)$ to $V(F'_3)$. Using that $F'_1$ and $F'_3$ are $(n/m^2,3m)$-vertex expansions of $v_1$ and $v_3$, respectively, let $P$ be a $v_1,v_3$-path with length at most $7m$ in $F'_1\cup P'\cup F'_3$.

Let $W=A\cup V(P)$, noting that $|W|\leq D/\log^3n+7m+1\leq n/m^2\log^3n$. Note that $0\leq \ell-\ell(P)+7m\leq 2n/\log^{12}n\leq n/m^2\log^3n$. Therefore, by Lemma~\ref{lem-conexp} with $(F_1,F_2,D,m,W,\ell)_{\ref{lem-conexp}}=(F_2',F_4',n/m^2,3m,W,\ell-\ell(P)+7m)$, there is a path $Q$ in $G-W$ from $v_2$ to $v_4$ with length between $\ell-\ell(P)+7m$ and $\ell-\ell(P)+22m$. As $\ell\leq \ell(P)+\ell(Q)\leq \ell+22m$, the paths $P$ and $Q$ satisfy the property in the corollary.
\end{proof}

\subsection{Subdivisions in skewed bipartite graphs}\label{skewbi}
As commented on before Theorem~\ref{mainthm}, we often work in graphs without a subdivision of a certain size clique with each edge divided once, that is, in a $\tk_\ell^{(2)}$-free graph for some $\ell$. This is because we often construct structures in a graph $G$ while avoiding a vertex set, say $W$, for which we need many edges in $G-W$. For the sizes of $W$ and values of $\ell$ that we use, if the graph $G$ is $\tk_\ell^{(2)}$-free, then the following simple proposition shows that there cannot be too many edges between $W$ and $V(G)\setminus W$. In our implementation this will imply that $G-W$ contains many edges.

\begin{prop}\label{prop-1subdivision} Let $d\in \N$ and let $G$ be a graph containing disjoint vertex sets $U$ and $W$ such that $|U|\geq |W|^2$ and every vertex in $U$ has at least $d$ neighbours in $W$.
Then, $G$ contains a $\tk_{d}^{(2)}$.
\end{prop}
\begin{proof}
	Take a maximal set $I\subseteq W^{(2)}$ for which there is a set of distinct vertices $v_{\{x,y\}}$, $\{x,y\}\in I$, in $U$ such that $x,y\in N(v_{\{x,y\}})$ for each $\{x,y\}\in I$. Now, as $|U|\geq |W|^2>|W^{(2)}|$, there is some $u\in U\setminus \{v_{\{x,y\}}:\{x,y\}\in I\}$.
	Let $A=N(u,W)$, so that $|A|\geq d$. In the choice of $I$, the vertex $u$ is a good candidate for $v_{\{x,y\}}$ for each $\{x,y\}\in A^{(2)}$. Thus, by the maximality of $I$, we have $A^{(2)}\subseteq I$.

Taking the $\tk_d^{(2)}$ with vertex set $A\cup \{v_{\{x,y\}}:\{x,y\}\in A^{(2)}\}$  and edge set $\{xv_{\{x,y\}},yv_{\{x,y\}}:\{x,y\}\in A^{(2)}\}$, then gives a $\tk_d^{(2)}$ in $G$, as required.
\end{proof}

\section{Proof of Theorem~\ref{mainthm}}\label{sec:mainthmpf}

In this section, we prove Theorem~\ref{mainthm}. As discussed in Section~\ref{sec:mainthmsketch}, the basic mechanism we use to adjust the length of a path is an adjuster, which we formally define as follows (see Figure~\ref{fig1}(b) for an illustration).

\begin{defn}\label{defn-adj}
	A \emph{$(D,m,k)$-adjuster} $\cA=(v_1,F_1,v_2,F_2,A)$ in a graph $G$ consists of vertices $v_1,v_2\in V(G)$, graphs $F_1,F_2\subseteq G$ and a vertex set $A\subseteq V(G)$ such that the following hold for some $\ell\in\N$.
  \stepcounter{propcounter}
  \begin{enumerate}[label = {\bfseries \Alph{propcounter}\arabic{enumi}}]
  \item $A$,  $V(F_1)$ and $V(F_2)$ are pairwise disjoint.\label{d-a-1}
  \item For each $i\in [2]$, $F_i$  is a $(D,m)$-expansion around $v_i$.\label{d-a-2}
  \item $|A|\leq 10mk$.\label{d-a-3}
  \item For each $i\in \{0,1,\ldots,k\}$, there is a $v_1,v_2$-path in $G[A\cup \{v_1,v_2\}]$ with length $\ell+2i$.\label{d-a-4}
\end{enumerate}
We call the smallest such $\ell$ for which these properties hold the \emph{length of the adjuster} and denote it $\ell(\cA)$. Note that it immediately follows that $\ell(\cA)\leq |A|+1\leq 10mk+1$. We call a $(D,m,1)$-adjuster a \emph{{simple adjuster}}. We refer to the subgraphs $F_1$ and $F_2$ of an adjuster $\cA=(v_1,F_1,v_2,F_2,A)$ as the \emph{ends} of the adjuster, and let $V(\cA)=V(F_1)\cup V(F_2)\cup A$.
\end{defn}

In this section, we start by finding one simple adjuster in an expander for Lemma~\ref{lem-twin-path} in Section~\ref{sec:oneadj}. We then find such an adjuster despite the removal of any medium-sized vertex set from the expander, giving Lemma~\ref{lem-robust-adj} in Section~\ref{sec:robadj}.
In Section~\ref{sec:adjpath}, we chain simple adjusters together for Lemma~\ref{lem-adj-to-adj-path}, before using this to join vertex expansions by paths with precise lengths for Lemma~\ref{lem-finalconnect}. Finally, we prove Theorem~\ref{thm:balancedsub} in Section~\ref{sec:balancedsub} and Theorem~\ref{mainthm} in Section~\ref{sec:mainthmfinal}.

\subsection{Finding one simple adjuster}\label{sec:oneadj}
Here, we find one simple adjuster, proving Lemma~\ref{lem-twin-path}. The adjuster $(v_1,F_1,v_2,F_2,A)$ is found for prespecified vertices $v_1$ and $v_2$, as required by one of the applications of Lemma~\ref{lem-twin-path}.

\begin{lemma}\label{lem-twin-path}
	For any $0<\eps_1<1$, $0<\eps_2<1/5$ and $k\in \N$, there exists $d_0=d_0(\eps_1,\eps_2,k)$ such that the following is true for each $n\geq d\geq d_0$. Suppose that $G$ is an $n$-vertex bipartite $(\eps_1,\eps_2 d)$-expander with $\de(G)\ge d-1$.

  Let $C$ be a shortest cycle in $G$ and let $x_1,x_2$ be distinct vertices in $V(G)\setminus V(C)$. Let $m=\frac{200}{\eps_1}\log^3n$ and $D\leq \log^{5k}n$.

  Then, $G$ contains a $(D,m,1)$-adjuster $(v_1,F_1,v_2,F_2,A)$ with $v_1=x_1$, $v_2=x_2$ and $V(C)\subseteq A$.
\end{lemma}
\begin{proof}
Noting that, as $G$ is bipartite, $C$ has even length, let $\ell_0$ be such that $2\ell_0$ is the length of $C$. Since $\de(G)\ge d-1$, we must have $\ell_0\leq \log n/\log (d-1)\leq m$, as $n\geq d_0(\eps_1,\eps_2,k)$ is large. Pick vertices $x_3,x_4\in V(C)$  which are distance $\ell_0-1$ apart on $C$ and let the paths separating them in $C$ be $R_1$ and $R_2$, where $R_1$ is the shorter path.

Let $D_{1,1}=D_{2,1}=D$,  $D_{1,2}=D_{3,1}=m^3D$,  $D_{2,2}=D_{4,1}=m^2D$, and note that $m^3D\leq \log^{15k}n$.
Using any arbitrary vertices $x_5, x_6,\ldots,x_{15k}$ and $D_{i,j}=D$ for any $i,j\in [15k]$ not already chosen, apply Lemma~\ref{lem-expansion} to $x_1, x_2,\ldots,x_{15k}$ and $C$ with $(k,m)_{\ref{lem-expansion}}=(15k,m/5)$ to get graphs $F_{i,j}$, $i,j\in [2]$ and $F_{3,1}, F_{4,1}$, for which the following hold.
\begin{itemize}
  \item For each $i,j\in [2]$ or $i\in\{3,4\}$ and $j=1$, $F_{i,j}$ is a $(D_{i,j},m)$-expansion around $x_i$ in $G$ which contains no vertices other than $x_i$ in $\{x_1,\ldots,x_4\} \cup V(C)$.
  \item The sets $V(F_{i,j})\setminus\{x_i\}$, $i,j\in [2]$ or $i\in\{3,4\}$ and $j=1$, are pairwise disjoint.
\end{itemize}

Now, we have $|V(C)\cup V(F_{1,1}\cup F_{2,1}\cup F_{2,2}\cup F_{4,1})|\leq m+2D+2m^2D\leq 10m^3D/\log^3n$. Therefore, as $|F_{1,2}|=|F_{3,1}|=m^3D$, by Lemma~\ref{new-connect}, we can find a path $P'$ with length at most $m$ from $V(F_{1,2})$ to $V(F_{3,1})$ with no vertices in $(V(C)\cup V(F_{1,1}\cup F_{2,1}\cup F_{2,2}\cup F_{4,1}))\setminus \{x_1,x_3\}$. As, $F_{1,2}$ is a $(D_{1,2},m)$-expansion of $x_1$, and $F_{3,1}$ is a $(D_{3,1},m)$-expansion of $x_3$, we can extend $P'$ using vertices from $V(F_{1,2}\cup F_{1,3})$ to get an $x_1,x_3$-path, say $P$, with length at most $3m$, which has no vertices in $(V(C)\cup V(F_{1,1}\cup F_{2,1}\cup F_{2,2}\cup F_{4,1}))\setminus \{x_1,x_3\}$.

 Next, observe that $|V(C)\cup V(P)\cup V(F_{1,1})\cup V(F_{2,1})|\leq m+3m+1+2D\leq 10m^2D/\log^3n$. Therefore, as $|F_{2,2}|=|F_{4,1}|=m^2D$, by Lemma~\ref{new-connect}, we can find a path $Q'$ with length at most $m$ from $V(F_{2,2})$ and $V(F_{4,1})$ which has no vertices in $(V(C)\cup V(P)\cup V(F_{1,1})\cup V(F_{2,1}))\setminus \{x_2,x_4\}$.
As, $F_{2,2}$ is a $(D_{2,2},m)$-expansion of $x_2$, and $F_{4,1}$ is a $(D_{4,1},m)$-expansion of $x_4$, we can extend $Q'$  using vertices from $V(F_{2,2}\cup F_{4,1})$ to get an $x_2,x_4$-path, say $Q$, with length at most $3m$ and no vertices in $(V(C)\cup V(P)\cup V(F_{1,1})\cup V(F_{2,1}))\setminus \{x_2,x_4\}$.

Set now $v_1=x_1$, $v_2=x_2$, $F_1=F_{1,1}$, $F_2=F_{2,1}$ and $A=V(P\cup Q\cup R_1\cup R_2)\setminus \{v_1,v_2\}$.
Then, $|A|\leq 2(3m+1)+2\ell_0\leq 10m$ and note that $A$ is disjoint from $V(F_1)\cup V(F_2)$. Letting $\ell=\ell(P\cup R_1\cup Q)$, note that $P\cup R_1\cup Q$ and $P\cup R_2\cup Q$ are $v_1,v_2$-paths in $G[A\cup\{v_1,v_2\}]$ with length $\ell$ and $\ell+2$ respectively. Thus, $(v_1,F_1,v_2,F_2,A)$ is a $(D, m, 1)$-adjuster, as desired.
\end{proof}

\subsection{Finding simple adjusters robustly}\label{sec:robadj}
In this section, we prove Lemma~\ref{lem-robust-adj}, a key component of our proof. This finds a simple adjuster robustly in an expander $G$ -- that is, given any subset $U\subseteq V(G)$ with moderate size, we construct an adjuster in $G-U$. 
\begin{lemma}\label{lem-robust-adj}
There exists some $\eps_1>0$ such that, for every $0<\eps_2<1$ and $k\in \N$, there exists $d_0=d_0(\eps_1,\eps_2,k)$ such that the following is true for each $n\geq d\geq d_0$. Suppose that $G$ is a $\tk_{d/2}^{(2)}$-free $n$-vertex bipartite $(\eps_1,\eps_2 d)$-expander with $\de(G)\ge d$.
Let $m=\frac{200}{\eps_1}\log^3n$ and $D\leq \log^k n$.
Let $U\subseteq V(G)$ satisfy $|U|\leq 10D$.

Then, $G-U$ contains a $(D,2m,1)$-adjuster.
\end{lemma}

The following proof sketch is illustrated in Figure~\ref{fig:rob}. Essentially, we find an expander subgraph in $H\subseteq G-U$ and apply Lemma~\ref{lem-twin-path} to find a simple adjuster in $H$. However, $H$ may be much smaller than $G$, so this simple adjuster may be far too small to satisfy Lemma~\ref{lem-robust-adj}. We thus find many of these simple adjusters and use Lemma~\ref{lem-newbit-other} to expand the ends of one of them to make them large enough to satisfy Lemma~\ref{lem-robust-adj}.

More precisely, to prove Lemma~\ref{lem-robust-adj}, we assume no such adjuster exists, before collecting the high degree vertices in a set $L$.
We take a maximal set of adjusters $\mathbf{A}_0$ in $G-U$ so that their ends (the sets $V(F_1)$ and $V(F_2)$ in an adjuster $(v_1,F_1,v_2,F_2,A)$) are in $G-L$, and furthermore the ends of different adjusters in $\mathbf{A}_0$ are far apart in $G-L$. The adjusters in $\mathbf{A}_0$ will each not have large enough ends to be a $(D,m,1)$-adjuster, so we wish to expand the ends of some adjuster to make them larger. The challenge is to do this while avoiding $U$.

We first show that $\mathbf{A}_0$ contains many adjusters (see Claim~\ref{ping}). If this is not the case, then, collecting together $U$ with the adjusters in $\mathbf{A}_0$ and any vertices near their ends in $G-L$, we remove them and show that there must be an expander subgraph $H$ in what remains. Applying Lemma~\ref{lem-twin-path} to (essentially) $H$, we get an adjuster that either satisfies the lemma (a contradiction) or should have belonged in $\mathbf{A}_0$ (another contradiction).

If many adjusters in $\mathbf{A}_0$ have an end with a short path to $L\setminus U$, then applying Lemma~\ref{lem-newbit-other} shows that in one of these adjusters the other end must expand while avoiding the short path to $L\setminus U$. Larger ends can then be chosen for this adjuster respectively from the expansion and from the short path to $L\setminus U$  and the neighbourhood of its endvertex in $L\setminus U$. This gives an adjuster satisfying the lemma (a contradiction). Thus, many adjusters in $\mathbf{A}_0$ have no short path to $L\setminus U$ (see Claim~\ref{ping2}) -- we collect such adjusters in $\mathbf{A}_1\subseteq \mathbf{A}_0$.

We then find a large set $Z$ in $G-L$ with small diameter using Lemma~\ref{cl-egg} -- destined to provide a large expansion for the end of an adjuster.
If many adjusters in $\mathbf{A}_1$ have an end with a short path to $Z$, then applying Lemma~\ref{lem-newbit-other} shows that, for one of these adjusters, the other end must expand while avoiding the short path to $Z$. Larger ends can then be chosen for this adjuster respectively from the expansion and from the short path to $Z$  and $Z$ itself. This gives an adjuster satisfying the lemma (a contradiction). Thus, many adjusters in $\mathbf{A}_1$ have no short path to $Z$ (see Claim~\ref{ping3}) -- we collect such adjusters in $\mathbf{A}_2\subseteq \mathbf{A}_1$.

However, by Lemma~\ref{lem-newbit-other}, the ends of the adjusters in $\mathbf{A}_2$ must expand while avoiding $U$. Therefore, by Lemma~\ref{new-connect}, one of them must connect to $Z$, giving the final contradiction  which completes the proof.


\begin{figure}[h]
\centering

\begin{tikzpicture}
\draw (0,0) -- (0,4) -- (8,4) -- (8,0) -- (0,0);

\def\Zh{3}
\def\Zw{1.5}
\draw (\Zw,\Zh) circle [radius=0.75cm];
\draw (\Zw,\Zh) node {$Z$};
\coordinate (Z2) at ($(\Zw,\Zh)+(0:0.75)$);
\coordinate (Z1) at ($(\Zw,\Zh)+(-60:0.75)$);

\def\Lw{4}
\def\Lh{1}
\def\Lsw{3.5}
\def\Lsh{2.75}
\def\Uw{1.25}
\def\Uh{3}
\def\Usw{6.5}
\def\Ush{0.5}
\draw [rounded corners] ($(\Lsw,\Lsh)+0.5*(\Lw,0)$) -- ($(\Lsw,\Lsh)+(\Lw,0)$) -- ($(\Lsw,\Lsh)+(\Lw,\Lh)$)-- ($(\Lsw,\Lsh)+(0,\Lh)$) -- ($(\Lsw,\Lsh)$) --  ($(\Lsw,\Lsh)+0.5*(\Lw,0)$);
\draw [rounded corners]  ($(\Usw,\Ush)+0.5*(\Uw,0)$) -- ($(\Usw,\Ush)+(\Uw,0)$) -- ($(\Usw,\Ush)+(\Uw,\Uh)$)-- ($(\Usw,\Ush)+(0,\Uh)$) -- ($(\Usw,\Ush)$) -- ($(\Usw,\Ush)+0.5*(\Uw,0)$);

\def\sep{0.7}

\foreach \x in {1,2,3}
{
\coordinate (L\x) at ($(\Lsw,\Lsh)+(0.2+\x*\sep,0)$);
}

\draw ($(\Lsw,\Lsh)+0.5*(\Lw,\Lh)$) node {$L$};
\draw ($(\Usw,\Ush)+0.5*(\Uw,\Uh)$) node {$U$};

\draw [gray] (3,1.2) ellipse (2.85cm and 1cm);
\draw [gray] (2.15,1.2) ellipse (2cm and 0.8cm);
\draw [gray] (1.35,1.2) ellipse (1.2cm and 0.6cm);


\def\radi{0.6}
\def\xyradi{1.5}
\def\adjj{0.3}
\def\upp{0.5}
\def\scaler{0.2}
\def\lwidth{0.02cm}

\draw (4.6,0.75) node {$\mathbf{A}_0$};
\draw (3.2,0.85) node {$\mathbf{A}_1$};
\draw (1.8,0.95) node {$\mathbf{A}_2$};

\foreach \y/\z/\k in {1.8/1.4/15,0.6/1.3/5,1.1/0.9/-20}
{
\begin{scope}[xshift=\y cm,yshift=\z cm,rotate=\k]
\draw (0,0) [line width=\lwidth] circle [radius={\radi*\scaler}];
\foreach \x in {1,...,10}
{
\coordinate (A\x) at ({36*\x+18}:{\radi*\scaler});
}
\coordinate (X1) at ($(180:{\xyradi*\scaler})+(0,{\upp*\scaler})$);
\coordinate (X2) at ($(180:{\xyradi*\scaler})-(0,{\upp*\scaler})$);
\coordinate (Y1) at ($(0:{\xyradi*\scaler})+(0,{\upp*\scaler})$);
\coordinate (Y2) at ($(0:{\xyradi*\scaler})-(0,{\upp*\scaler})$);
\draw [line width=\lwidth] (A5) -- (X1) -- (X2) -- (A5);
\draw [line width=\lwidth] (A9) -- (Y1) -- (Y2) -- (A9);
\end{scope}
}

\foreach \y/\z/\k/\xx in {2.9/1.5/-45/1,3.75/1.25/-15/2}
{
\begin{scope}[xshift=\y cm,yshift=\z cm,rotate=\k]
\draw (0,0) [line width=\lwidth] circle [radius={\radi*\scaler}];
\foreach \x in {1,...,10}
{
\coordinate (A\x) at ({36*\x+18}:{\radi*\scaler});
}
\coordinate (X1) at ($(180:{\xyradi*\scaler})+(0,{\upp*\scaler})$);
\coordinate (X2) at ($(180:{\xyradi*\scaler})-(0,{\upp*\scaler})$);
\coordinate (Y1) at ($(0:{\xyradi*\scaler})+(0,{\upp*\scaler})$);
\coordinate (Y2) at ($(0:{\xyradi*\scaler})-(0,{\upp*\scaler})$);
\draw [line width=\lwidth] (A5) -- (X1) -- (X2) -- (A5);
\draw [line width=\lwidth] (A9) -- (Y1) -- (Y2) -- (A9);
\draw [snake it] (X1) -- (Z\xx);
\end{scope}
}

\foreach \y/\z/\k/\xx in {4.5/1.75/10/1, 5.35/1/0/3, 5/1.4/-15/2}
{
\begin{scope}[xshift=\y cm,yshift=\z cm,rotate=\k]
\draw (0,0) [line width=\lwidth] circle [radius={\radi*\scaler}];
\foreach \x in {1,...,10}
{
\coordinate (A\x) at ({36*\x+18}:{\radi*\scaler});
}
\coordinate (X1) at ($(180:{\xyradi*\scaler})+(0,{\upp*\scaler})$);
\coordinate (X2) at ($(180:{\xyradi*\scaler})-(0,{\upp*\scaler})$);
\coordinate (Y1) at ($(0:{\xyradi*\scaler})+(0,{\upp*\scaler})$);
\coordinate (Y2) at ($(0:{\xyradi*\scaler})-(0,{\upp*\scaler})$);
\draw [line width=\lwidth] (A5) -- (X1) -- (X2) -- (A5);
\draw [line width=\lwidth] (A9) -- (Y1) -- (Y2) -- (A9);
\draw [snake it] (Y1) -- (L\xx);
\end{scope}
}

\end{tikzpicture}
\caption{An illustration of the proof of Lemma~\ref{lem-robust-adj}. We find $\mathbf{A}_0$, a large set of adjusters in $G-U$, before discarding those with a short path to $L\setminus U$, and then those with a short path to $Z$. Showing that many adjusters still remain, in the set $\mathbf{A}_2$, leads to a contradiction.\label{fig:rob}}
\end{figure}
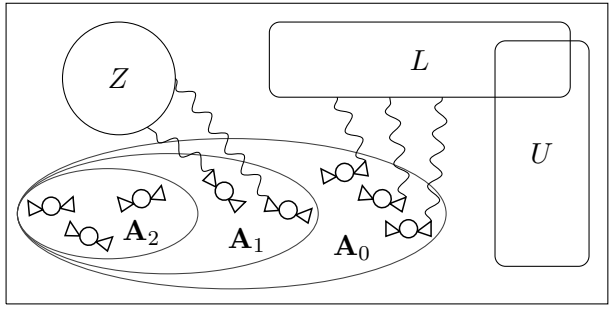


\begin{proof}[Proof of Lemma~\ref{lem-robust-adj}] Let $0<\eps_1<1$ be small enough that the property in Corollary~\ref{cor-expander} holds. Suppose, for contradiction, that $G-U$ contains no $(D,2m,1)$-adjuster. Let $\Delta=200mD$, $L=\{v\in V(G):d_G(v)\geq \Delta\}$ and $G'=G-L$, so that $\Delta(G')\leq \Delta$.

Set $\ell_0=(\log\log n)^{20}$. Let $U_0=\{v\in V(G)\setminus U:d_G(v,U)\geq d/2\}$.
Note that if $|U_0|\geq 100D^2\ge |U|^2$, then by Proposition~\ref{prop-1subdivision} with $(U,W)_{\ref{prop-1subdivision}}=(U_0,U)$, $G$ contains a $\tk_{d/2}^{(2)}$, a contradiction. Therefore, we can assume that $|U_0|\leq 100D^2$, and hence, as $\delta(G)\geq d$ and $n\geq d_0(\eps_1,\eps_2,k)$ is large, $G-U$ contains at least $(n-|U|-|U_0|)\cdot (d/2)/2\geq nd/8$ edges. Let $U_1=U\cup U_0$, so that $|U_1|\le 200D^2\leq 200\log^{2k}n$.

Take a maximal collection $\mathbf{A}_0$ of adjusters in $G-U$, such that the following hold.
\stepcounter{propcounter}
\begin{enumerate}[label = {\bfseries \Alph{propcounter}\arabic{enumi}}]
\item The sets $V(F_{1}\cup F_2)$, $(v_1,F_1,v_2,F_2,A)\in \mathbf{A}_0$, are subsets of $V(G')$ and are all at least a distance $10\ell_0$ apart from each other and from $U_1\setminus L$ in $G'$.\label{hot1}
\item For each $\cA\in \mathbf{A}_0$, for some $m_{\cA}$ with $\log^3d_0\leq m_{\cA}\leq m$, $\cA$ is an $(m_\cA^{2},m_\cA,1)$-adjuster.\label{hot2}
\end{enumerate}

  \begin{claim}\label{ping}
$|\mathbf{A}_0|\geq n^{1/4}$.
  \end{claim}
  \begin{poc} Suppose, for contradiction, that $|\mathbf{A}_0|< n^{1/4}$. Let $W=(U_1\cup (\cup_{\cA\in \mathbf{A}_0}V(\cA))\setminus L$. For each $\cA=(v_1,F_1,v_2,F_2,A)\in \mathbf{A}_0$, $|V(\cA)|=|F_1|+|F_2|+|A|\leq 2m_\cA^2+10m_\cA\leq 3m^2$, and therefore $|W|\leq n^{1/4}\cdot 3m^3+200\log^{2k}n\leq n^{1/3}$. Let $W'=B_{G'}^{10\ell_0}(W)$, so, as $\Delta(G')\le \Delta$, we have that $|W'|\leq 2|W|\cdot\Delta^{10\ell_0}\leq n^{1/2}$.

Now, there are at most $|W'|\Delta\leq \Delta n^{1/2}\leq nd/16$ edges in $G$ with some vertex in $W'$. Let $\bar{d}=d/64$. As  $G-U$ contains at least $nd/8$ edges, $G-U-W'$ contains at least $nd/16$ edges, so that $d(G-U-W')\geq d/8=8\bar{d}$. Then, by Corollary~\ref{cor-expander}, $G-U-W'$ contains an $(\eps_1,\eps_2\bar{d})$-expander $H$ with $\delta(H)\geq \bar{d}$.
Let $C$ be a shortest cycle in $H$. We will consider two cases, depending on how many vertices of $L$ there are in $V(H)\setminus V(C)$.

\medskip

\noindent\textbf{Case I: $|(V(H)\setminus V(C))\cap L|\leq 1$.} Let $H'=H-(V(H)\setminus V(C))\cap L$, so that $\delta(H')\geq \bar{d}-1$. Note that, for each $X\subseteq V(H')$ with $\eps_2\bar{d}/2\leq |X|\leq |H'|/2<|H|/2$, we have
\begin{align*}
|N_{H'}(X)|&\geq |N_H(X)|-1\geq |X|\cdot \eps(|X|,\eps_1,\eps_2\bar{d})-1 \\
&\geq \frac{1}{2}|X|\cdot \eps(|X|,\eps_1,\eps_2\bar{d})+\frac{\eps_2\bar{d}}{4}\cdot \eps(\eps_2\bar{d}/2,\eps_1,\eps_2\bar{d})-1
\\
&\geq |X|\cdot \eps(|X|,\eps_1/2,\eps_2\bar{d})+\frac{\eps_2\bar{d}}{4}\cdot \frac{\eps_1}{\log^2(15/2)}-1\geq |X|\cdot \eps(|X|,\eps_1/2,\eps_2\bar{d}),
\end{align*}
where the last inequality follows as $\bar{d}\geq d_0(\eps_1,\eps_2,k)/64$ is large. Therefore, $H'$ is a $(\eps_1/2,\eps_2\bar{d})$-expander with $\delta(H')\geq \bar{d}-1$. Note that $C$ is a shortest cycle in $H'$.

Let $m_{H'}=200\log^3|H'|/\eps_1\le m$, and note that, as $|H'|\geq \delta(H')+1\geq \bar{d}\geq d_0/64$, and $d_0=d_0(\ep_1,\ep_2,k)$ is large, $m_{H'}\geq \log^3d_0$. Picking arbitrary vertices $x_1,x_2\in V(H')\setminus V(C)$ and noting that $\bar{d}\geq d_0(\eps_1,\eps_2,k)/64$ is large, by Lemma~\ref{lem-twin-path} with $(k,D)_{\ref{lem-twin-path}}=(10,m_{H'}^2)$,
 $H'$ contains an $(m_{H'}^{2},m_{H'},1)$-adjuster  $(v_1,F_1,v_2,F_2,A)$ with $V(C)\subseteq A$. As $A$ is disjoint from $V(F_1\cup F_2)$, $V(C)\subseteq A$ and $(V(H')\setminus V(C))\cap L=\varnothing$, we have that $V(F_1\cup F_2)$ is disjoint from $L$, and hence lies in $V(G')$.
 Together with $V(F_1\cup F_2)\subseteq V(H')$ being disjoint from $W'$ and so $10\ell_0$-far in $G'$ from the ends of the adjusters in $\mathbf{A}_0$ and from $U_1\setminus L$, this violates the maximality of $\mathbf{A}_0$, a contradiction.

\medskip

\noindent\textbf{Case II: $|(V(H)\setminus V(C))\cap L|\geq 2$.} Let $x_1,x_2\in (V(H)\setminus V(C))\cap L$ be distinct and let $m_{H'}=200\log^3|H'|/\eps_1\le m$. By Lemma~\ref{lem-twin-path} with $(k,D)_{\ref{lem-twin-path}}=(1,1)$, $H$ contains a $(1,m_{H'},1)$-adjuster $(v_1,F_1,v_2,F_2,A)$ with $v_1=x_1$ and $v_2=x_2$. Using that $|A|\leq 10m_{H'}\leq 10m$, $|U|\leq 10D$, and $d_G(x_1),d_G(x_2)\geq \Delta=200mD$,
pick disjointly sets $X_1\subseteq N_G(x_1)\setminus (U\cup A\cup \{x_2\})$ and $X_2\subseteq N_G(x_2)\setminus (U\cup A\cup \{x_1\})$ with $|X_1|=|X_2|=D-1$. Letting $F_i'=G[\{x_i\}\cup X_i]$ for each $i\in [2]$, and noting $|A|\leq 20m$, we have that $(x_1,F'_1,x_2,F'_2,A)$ is a $(D,2m,1)$-adjuster in $G-U$, a contradiction.
  \end{poc}

Now, let $\mathbf{A}_1\subseteq \mathbf{A}_0$ be the set of adjusters $(v_1,F_1,v_2,F_2,A)\in \mathbf{A}_0$ for which there is no path with length at most $\ell_0$ from $V(F_1)\cup V(F_2)$ to $L\setminus U$ in $G-U-A$.

  \begin{claim}\label{ping2}
$|\mathbf{A}_1|\geq n^{1/4}/2$.
  \end{claim}
  \begin{poc} Let $r= n^{1/8}$.
Suppose, for contradiction, that we can label distinct $\cA_1,\ldots,\cA_r\in \mathbf{A}_0\setminus \mathbf{A}_1$. Say, for each $i\in [r]$, that $\cA_i=(v_{i,1},F_{i,1},v_{i,2},F_{i,2},\bar{A}_i)$ and let $P_i'$ be a shortest path with length at most $\ell_0$ from $V(F_{i,1})\cup V(F_{i,2})$ to $L\setminus U$ in $G-U-\bar{A}_i$. Relabelling, if necessary, for each $i\in [r]$ suppose the endvertex of $P_i'$ in $V(F_{i,1}\cup F_{i,2})$ is in $V(F_{i,1})$, and let $Q_i$ be a path from this endvertex of $P_i'$ to $v_{i,1}$ in $F_{i,1}$ with length at most $m_{\cA_i}$.

For each $i\in [r]$, let $x_i$ be the endpoint of $P_i'$ in $L\setminus U$, and let $P_i=P_i'-x_i$. We shall apply Lemma~\ref{lem-newbit-other} by setting, for each $i\in [r]$,  $A_i=V(F_{i,2})$, $B_i=\bar{A}_i\cup V(Q_i)\cup \{x_i\}$ and $C_i=V(P_i)$. Firstly, as $|A_i|=m_{\cA_i}^2\geq \log^6d_0$ by \ref{hot2} and $d_0=d_0(\ep_1,\ep_2,k)$ is large, we have that $|A_i|\geq d_0^{\ref{lem-newbit-other}}$, where $d_0^{\ref{lem-newbit-other}}$ is the function in Lemma~\ref{lem-newbit-other}, so that \emref{mouse0} holds.

 As $V(F_{i,2})\subseteq V(G')=V(G)\setminus L$ by \ref{hot1}, and $V(F_{i,2})$ is disjoint from $V(F_{i,1})$ and $\bar{A}_i$ by \ref{d-a-1}, we have that $A_i$ and $B_i\cup C_i$ are disjoint. Furthermore, $|B_i|\leq |\bar{A}_i|+|Q_i|+1\leq 20m_{\cA_i}\leq m_{\cA_i}^2/\log^{10}(m_{\cA_i}^2)$ as $m_{\cA_i}\geq \log^3(d_0(\ep_1,\ep_2,k))$ is large, and thus \emref{mouse1} holds.

Now, as $P_i'$ is a shortest path from $V(F_{i,1})\cup V(F_{i,2})$ to $L\setminus U$ in $G-U-\bar{A}_i$, which has an endvertex in $V(F_{i,1})$, and $A_i=V(F_{i,2})$, we have, for each $\ell\in \N$, that $B^\ell_{G-U-\bar{A_i}}(A_i)$
has at most $\ell+1$ vertices in $P_i'$, and hence $P_i$. Therefore, $A_i$ has 4-limited contact with $C_i$ in $G-U-\bar{A_i}$, and hence in $G-U-B_i$, and thus \emref{mouse2} holds.

Suppose there is a path, $R_i$ say, with length at most $10\ell_0$ from $A_i$ to $L\setminus (U\cup\{x_i\})$ in $G-U-B_i-C_i$. Then, there is a path $R'_i\subseteq R_i\cup F_{i,2}$ from $v_{i,2}$ to some vertex $y_i\in L\setminus (U\cup\{x_i\})$ with length at most $10\ell_0+m_{\cA_i}\leq 2m-1$, and the path $Q_i\cup P'_i$ is a path from $v_{i,1}$ to $x_i$ with length at most $m_{\cA_i}+\ell_0\leq 2m-1$ in $G-U-\bar{A}_i$ with vertices in $B_i\cup C_i$.
Then, as $|U\cup A_i\cup V(R_i')\cup V(Q_i\cup P_i')|\leq 10D+10m_{\cA_i}+4m\leq 10D+15m$, as $x_i,y_i\in L$ both have degree at least $\Delta=200mD$, we can comfortably choose $X_i\subseteq N_G(x_i)$ and $Y_i\subseteq N_G(y_i)$ which are disjoint from each other and from $U\cup A_i\cup V(R_i')\cup V(Q_i\cup P_i')$ and have size $D-|P_i'\cup Q_i|$ and $D-|R'_i|$ respectively. Then, $(v_{i,1}, G[X_i\cup V(P_i')\cup V(Q_i)],v_{i,2},G[Y_i\cup V(R_i')],A_i)$
is a $(D,2m,1)$-adjuster in $G-U$, a contradiction.
Therefore, there is no such path $R_i$. Consequently, recalling that $A_i=V(F_{i,2})$, we have
\[
B_{G-U-B_i-C_i}^{\ell_0}(A_i)=B_{G'-U-B_i-C_i}^{\ell_0}(A_i),
\]
which, by~\ref{hot1}, is disjoint from $U_1$. By the choice of $U_0\subseteq U_1$, we have that \emref{mouse3} holds.

Now, similarly, for any $j\in [r]\setminus\{i\}$, we have that $B_{G-U-B_j-C_j}^{\ell_0}(A_j)=B_{G'-U-B_j-C_j}^{\ell_0}(A_j)$, so that, by \ref{hot1}, $B_{G-U-B_j-C_j}^{\ell_0}(A_j)$ and $B_{G-U-B_i-C_i}^{\ell_0}(A_i)$ are disjoint. In particular, $A_i$ and $A_j$ are a distance at least $2\ell_0$ apart in $G-U-B_i-C_i-B_j-C_j$, and therefore $\emref{mouse4}$ holds.

Thus, by Lemma~\ref{lem-newbit-other}, there is some $j\in [r]$ for which $|B^{\ell_0}_{G-U-B_j-C_j}(A_j)|\geq \log^kn\geq D$. As $F_{j,2}$ is an $(m_{\cA_j}^{2},m_{\cA_j})$-expansion of $v_{j,2}$  in $G'-U-B_j-C_j$, $m_{\cA_j}\leq m$ and $A_j=V(F_{j,2})$, we have that $|B^{2m}_{G-U-B_j-C_j}(v_{j,2})|\geq D$ as $\ell_0\ll m$. Therefore, by Proposition~\ref{prop-trimming}, we can pick a $(D,2m)$-expansion, $F'_{j,2}$ say, of $v_{j,2}$ in $G-U-B_i-C_j$.

As $x_j\in L$, we can then pick a set $U'$ of neighbours of $x_{j}$ disjoint from $U\cup V(F'_{j,2})\cup \bar{A}_{j}\cup  V(Q_j)\cup V(P_j')$ with $|U'|=D-|V(P_j'\cup Q_j)|$. Let $F'_{j,1}=G[U'\cup V(P_j')\cup V(Q_j)]$. Note that $F'_{j,1}$ is then a $(D,2m)$-expansion of $v_{j,1}$ as $Q_j\cup P'_j$ is a $v_{j,1},x_j$-path with length at most $m_{\cA_j}+\ell_0\leq 2m-1$. Finally, note that $(v_{j,1},F'_{j,1},v_{j,2},F'_{j,2},\bar{A}_j)$ is a $(D,2m,1)$-adjuster in $G-U$, a contradiction. Therefore, $ |\mathbf{A}_0 \setminus \mathbf{A}_1| <  r=n^{1/8}$, and so by Claim~\ref{ping}, we have $|\mathbf{A}_1|> n^{1/4}-r\geq n^{1/4}/2$.
\end{poc}

Let $\mathbf{A}_1'\subseteq \mathbf{A}_1$ satisfy $|\mathbf{A}_1'|=n^{1/4}/2$. Then, $|\cup_{\cA\in \mathbf{A}_1'}V(\cA)|\leq n^{1/4}\cdot 3m^2\leq n^{1/3}$ by \ref{hot2}. Therefore,
\[
|U\cup B_{G'}^{\ell_0}(\cup_{\cA\in \mathbf{A}_1'}(V(\cA)\setminus L))|\leq 10D+n^{1/3}\cdot 2\Delta^{\ell_0}\leq n^{1/2}.
\]
Thus, by Lemma~\ref{cl-egg}, there is a set $Z\subseteq V(G)\setminus U$ which has diameter at most $m/2$ and size $10m^2D$, and is a distance at least $\ell_0$ in $G'$ from $V(\cA)\setminus L$ for each $\cA\in \mathbf{A}_1'$.

Let $\mathbf{A}_2\subseteq \mathbf{A}_1'$ be the set of adjusters $(v_1,F_1,v_2,F_2,A)\in \mathbf{A}_1'$ for which there is no path with length at most $m/2$ from $V(F_1)\cup V(F_2)$ to $Z$ in $G-U-A$.

  \begin{claim}\label{ping3}
$|\mathbf{A}_2|\geq n^{1/4}/4$.
  \end{claim}
  \begin{poc} Let $r=n^{1/8}$.
Suppose, for contradiction, we can label distinct $\cA_1,\ldots,\cA_r\in \mathbf{A}_1'\setminus \mathbf{A}_2$. Say, for each $i\in [r]$, that $\cA_i=(v_{i,1},F_{i,1},v_{i,2},F_{i,2},\bar{A}_i)$ and let $P_i$ be a shortest path with length at most $m/2$ from $V(F_{i,1})\cup V(F_{i,2})$ to $Z$ in $G-U-\bar{A}_i$. Relabelling, if necessary, for each $i\in [r]$ suppose the endvertex of $P_i$ in $V(F_{i,1}\cup F_{i,2})$ is in $V(F_{i,1})$, and let $Q_i$ be a path from this endvertex of $V(P_i)$ to $v_{i,1}$ in $F_{i,1}$ with length at most $m_{\cA_i}$.

We will apply Lemma~\ref{lem-newbit-other} to $A_i=V(F_{i,2})$, $B_i=\bar{A}_i\cup V(Q_i)$ and $C_i=V(P_i)$, for each $i\in [r]$. For each $i\in [r]$, similarly to the proof of Claim~\ref{ping2}, we have that \emref{mouse0}--\emref{mouse2} hold.
By the choice of $\mathbf{A}_1$, for each $i\in[r]$, there is no path of length at most $\ell_0$ from $A_i$ to $L\setminus U$ in $G-U-B_i-C_i$.
Therefore, the sets $B_{G-U-B_i-C_i}^{\ell_0}(A_i)$ and $B_{G'-U-B_i-C_i}^{\ell_0}(A_i)$ are the same set, and thus, by \ref{hot1}, this set is disjoint from $U_1$. Thus, \emref{mouse3} holds by the definition of $U_1$.
It similarly follows that $B_{G-U-B_i-C_i}^{\ell_0}(A_i)$ and $B_{G-U-B_j-C_j}^{\ell_0}(A_j)$ are vertex disjoint for each $j\in [r]\setminus \{i\}$, and thus \emref{mouse4} holds.

Thus, by Lemma~\ref{lem-newbit-other}, there is some $j\in [r]$ for which $|B^{\ell_0}_{G'-U-B_j-C_j}(A_j)|=|B^{\ell_0}_{G-U-B_j-C_j}(A_j)|\geq D$. Thus, as $F_{j,2}$ is an $(m_{\cA_j}^2,m_{\cA_j})$-expansion of $v_{j,2}$ in $G'-U-B_j-C_j$ by \ref{hot1} and \ref{hot2}, and $A_j=V(F_{j,2})$, by Proposition~\ref{prop-trimming}, there is a $(D,2m)$-expansion, $F'_{j,2}$ say, of $v_{j,2}$ in $B^{\ell_0}_{G'-U-B_j-C_j}(V(F_{j,2}))$. As $Z$ was chosen to have a distance at least $\ell_0$ in $G'$ from $V(\cA_j)\setminus L$, we have that $V(F'_{j,2})$ is disjoint from $Z$.

Now, as $Z$ has diameter at most $m/2$ in $G$, $Q_j\cup P_j\cup G[Z]$ is an expansion of $v_{j,1}$ with radius at most $\ell(Q_j)+\ell(P_j)+m/2\leq 2m$ and size at least $D$. Therefore, by Proposition~\ref{prop-trimming}, we can find within $Q_j\cup P_j\cup G[Z]$ a $(D,2m)$-expansion, $F'_{j,1}$ say, of $v_{j,1}$,
which then must be vertex-disjoint from $\bar{A}_j$ and from $V(F'_{j,2})\subseteq B^{\ell_0}_{G'-U-B_j-C_j}(V(F_{j,2}))$.
Thus,  we have that $(v_{j,1},F'_{j,1},v_{j,2},F'_{j,2},\bar{A}_j)$
is a $(D,2m,1)$-adjuster in $G-U$, a contradiction. Thus, $|\mathbf{A}_2|\geq |\mathbf{A}_1|-r\geq n^{1/4}/4$, by Claim~\ref{ping2}.
\end{poc}

Let $r=n^{1/8}$. Using Claim~\ref{ping3}, label distinct $\cA_1,\ldots,\cA_r\in \mathbf{A}_2$, and say, for each $i\in [r]$, that $\cA_i=(v_{i,1},F_{i,1},v_{i,2},F_{i,2},\bar{A}_i)$.
We shall apply Lemma~\ref{lem-newbit-other}
to $A_i=V(F_{i,1}\cup F_{i,2})$, $B_i=\bar{A}_i$ and $C_i=\varnothing$.
Similarly as in the proof of Claim~\ref{ping3}, the only difference being that \emref{mouse2} holds trivially as $C_i=\varnothing$ and $A_i$ is slightly larger, we have that~\emref{mouse0}--\emref{mouse4} hold.

Thus, by Lemma~\ref{lem-newbit-other}, there is some $j\in [r]$ with $|B^{\ell_0}_{G-U-B_j-C_j}(A_j)|=|B^{\ell_0}_{G-U-B_j}(A_j)|\geq 10m^2D\ge 10\log^3n |U\cup B_j|$. Therefore, by Lemma~\ref{new-connect}, there is a path in $G-U-B_j$
 from $B^{\ell_0}_{G-U-B_j}(A_j)$ to $Z$ with length at most $m/4$. Then, as $A_j=V(F_{j,1}\cup F_{j,2})$ and $B_j=\bar{A}_j$, there is a path in $G-U-\bar{A}_j$ from $V(F_{j,1}\cup F_{j,2})$ to $Z$ with length at most $m/2$, contradicting $\cA_j\in \mathbf{A}_2$, and completing the proof.
\end{proof}

\subsection{Connecting simple adjusters for paths with specific lengths}\label{sec:adjpath}
Using Lemma~\ref{lem-robust-adj}, we can find many vertex disjoint simple adjusters. We now connect them together into a larger adjuster, for Lemma~\ref{lem-adj-to-adj-path}, before using these to construct paths with specific lengths for Lemma~\ref{lem-finalconnect}.

\begin{lemma}\label{lem-adj-to-adj-path}
	There exists some $\eps_1>0$ such that, for any $0<\eps_2<1/5$ and $k\geq 10$, there exists $d_0=d_0(\eps_1,\eps_2,k)$ such that the following holds for each $n\geq d\geq d_0$. Suppose that $G$ is an $n$-vertex $\tk_{d/2}^{(2)}$-free bipartite $(\eps_1,\eps_2 d)$-expander with $\de(G)\ge d$.

	Let $m=\frac{800}{\ep_1}\log^3n$.
	Suppose $\log^{10} n\leq D\leq \log^kn$, $1\le r\le 30m$ and $U\subseteq V(G)$ with $|U|\leq D$.

Then, there is a $(D,m,r)$-adjuster in $G-U$.
\end{lemma}
\begin{proof} Let $\eps_1>0$ be such that the property in Lemma~\ref{lem-robust-adj} holds. By this property, as $d\geq d_0(\ep_1,\ep_2,k)$ is large, for every set $V\subseteq V(G)$ with $|V|\leq \log^{2k}n$, $G-V$ contains a $(D,m/2,1)$-adjuster. By Lemma~\ref{new-connect}, as $d\geq d_0(\ep_1,\ep_2,k)$ is large, for any sets $X$ and $Y$ with size at least $2D$, and any set $V\subseteq V(G)\setminus (X\cup Y)$ with size at most $20D/\log^3n$, there is a path from $X$ to $Y$ in $G-V$ with length at most $m$.

We now prove the property in the lemma by induction on $r$. Note that, we already have this property for $r=1$ as $|U|\leq D\leq \log^{2k}n$, and a $(D,m/2,1)$-adjuster is also a $(D,m,1)$-adjuster. Suppose then, for some $r$ with $1\leq r< 30m$, $G-U$ contains a $(D,m,r)$-adjuster, $(v_1,F_1,v_2,F_2,A_1)$ say. Let $U'=U\cup A_1\cup V(F_1)\cup V(F_2)$, so that $|U'|\leq 4D\leq \log^{2k}n$. Therefore, $G-U'$ contains a $(D,m/2,1)$-adjuster, $(v_3,F_3,v_4,F_4,A_2)$ say. As $|F_1\cup F_2|=|F_3\cup F_4|=2D$, and $|A_1\cup A_2|\leq 20rm\leq 600m^2\leq D/\log^3n$, there is a path, $P$ say, with length at most $m$, from $V(F_1)\cup V(F_2)$ to $V(F_3)\cup V(F_4)$ avoiding $A_1\cup A_2$.

Note that, without loss of generality, we can assume that $P$ is a path from $V(F_1)$ to $V(F_3)$. Using that $F_1$ and $F_3$ are $(D,m)$-expansions of $v_1$ and $v_3$ respectively, take a $v_1,v_3$-path $Q\subseteq F_1\cup P\cup F_3$ with length at most $5m$. Then, we claim $(v_2,F_2,v_4,F_4,A_1\cup A_2\cup V(Q))$ is a $(D,m,r+1)$-adjuster. Indeed, we easily have that \ref{d-a-1} and \ref{d-a-2} hold, and $|A_1\cup A_2\cup V(Q)|\leq 5m+10\cdot (m/2)+10mr=10(r+1)m$, so that \ref{d-a-3} holds.

Finally, let $\ell_1=\ell((v_1,F_1,v_2,F_2,A_1))$, $\ell_2=\ell((v_3,F_3,v_4,F_4,A_2))$ and $\ell=\ell_1+\ell_2+\ell(Q)$. If $i\in \{0,1,\ldots,r+1\}$, then there is some $i_1\in \{0,1,\ldots,r\}$ and $i_2\in \{0,1\}$ with $i=i_1+i_2$. Let $P_1$ be a $v_2,v_1$-path in $G[A_1\cup\{v_1,v_2\}]$ with length $\ell_1+2i_1$ and let $P_2$ be a $v_3,v_4$-path with length $\ell_2+2i_2$ in $G[A_2\cup\{v_3,v_4\}]$. Then, $P_1\cup Q\cup P_2$ is a $v_2,v_4$-path in $G[A_1\cup A_2\cup V(Q)]$ with length $\ell+2i$, and thus $\ell$ satisfies \ref{d-a-4}.
\end{proof}

Combining Lemma~\ref{lem-adj-to-adj-path} with Corollary~\ref{longconnect4}, we can finally find paths with exactly some desired length, as follows.

\begin{lemma}\label{lem-finalconnect}
There exists some $\eps_1>0$ such that, for any $0<\eps_2<1/5$ and $k\geq 10$, there exists $d_0=d_0(\eps_1,\eps_2,k)$ such that the following holds for each $n\geq d\geq d_0$. Suppose that $G$ is an $n$-vertex $\tk_{d/2}^{(2)}$-free bipartite $(\eps_1,\eps_2 d)$-expander with $\de(G)\ge d$.

Suppose $\log^{10} n\leq D\leq \log^kn$, and $U\subseteq V(G)$ with $|U|\leq D/2\log^3n$, and let $m=\frac{800}{\eps_1}\log^3n$. Suppose $F_1,F_2\subseteq G-U$ are vertex disjoint such that $F_i$ is a $(D,m)$-expansion of $v_i$, for each $i\in[2]$. Let $\log^{7}n\leq \ell\leq n/\log^{12}n$ be such that $\ell=\pi(v_1,v_2,G)\mod 2$.

Then, there is a $v_1,v_2$-path with length $\ell$ in $G-U$.
\end{lemma}
\begin{proof}
By Lemma~\ref{lem-adj-to-adj-path}, there is a $(D,m,22m)$-adjuster, $\cA=(v_3,F_3,v_4,F_4,A)$ say, in $G-U$ with length $\ell(\cA)\le |A|+1\leq 500m^2$. Let $\bar{\ell}=\ell-22m-\ell(\cA)$,
so that $0\leq \bar{\ell}\le n/\log^{12}n$. As $|A\cup U|\leq 500m^2+D/2\log^3n\leq D/\log^3n$, by Corollary~\ref{longconnect4}, there are paths $P$ and $Q$ in $G-U-A$ which are vertex disjoint, both connect $\{v_1,v_2\}$ to $\{v_3,v_4\}$ and so that $\bar{\ell}\leq \ell(P)+\ell(Q)\leq \bar{\ell}+22m$.
Note that we can assume, without loss of generality, that $P$ is a $v_1,v_3$-path and $Q$ is a $v_2,v_4$-path.

Now, $0\leq \ell-\ell(P)-\ell(Q)-\ell(\cA)\leq 22m$. As $\cA$ is a $(D,m,22m)$-adjuster there is a $v_3,v_4$-path in $G[A\cup \{v_3,v_4\}]$ with length $\ell(\cA)$, and therefore $\ell(\cA)= \pi(v_3,v_4,G)\mod 2$. Then, as $\ell(P)=\pi(v_1,v_3,G)\mod 2$, $\ell(Q)=\pi(v_2,v_4,G)\mod 2$, $\ell=\pi(v_1,v_2,G)\mod 2$ and $\pi(v_1,v_2,G)=\pi(v_1,v_3,G)+\pi(v_3,v_4,G)+\pi(v_4,v_2,G) \mod 2$, we have $\ell-\ell(P)-\ell(Q)-\ell(\cA)=0\mod 2$. That is, there is some $i\in \N$ with $2i=\ell-\ell(P)-\ell(Q)-\ell(\cA)$, where $i\leq 11m$.

Therefore, by the property of the adjuster, there is a $v_3,v_4$-path, $R$ say, with length $\ell(\cA)+2i=\ell-\ell(P)-\ell(Q)$ in $G[A\cup\{v_3,v_4\}]$. Then, $P\cup R\cup Q$ is a $v_1,v_2$-path with length $\ell$ in $G-U$.
\end{proof}

\subsection{Proof of Theorem~\ref{thm:balancedsub}}\label{sec:balancedsub}
We can now prove Theorem~\ref{thm:balancedsub}. We take some core vertices $v_1,\ldots,v_k$ in an expander, find expansions around them using Lemma~\ref{lem-expansion} and then connect each pair of core vertices using Lemma~\ref{lem-finalconnect}.

\begin{proof}[Proof of Theorem~\ref{thm:balancedsub}] Let $\eps_1>0$ be such that the properties in Corollary~\ref{cor-expander} and Lemma~\ref{lem-finalconnect} hold. Let $\eps_2=1/10$. Let $d_0=d_0(\ep_1,\ep_2,k)$ be large, and let $d=8d_0$. Let $G$ be a graph with $d(G)\geq d$.

By Corollary~\ref{cor-expander}, we can find a bipartite $(\eps_1,\eps_2d)$-expander $H\subseteq G$ with $\de(H)\ge d_0$.
Let $K=\binom{k}{2}$, $n=|H|\geq d_0$, $m=\frac{800}{\eps_1}\log^3n$ and $\ell=\log^7n$. Take $k$ distinct vertices in the same partition in $H$, say $v_1,\ldots,v_k$. As $d_0(\ep_1,\eps_2,k)$ is large and $m^{10K}\leq \log^{30k^2}n$, by Lemma~\ref{lem-expansion} with $k_{\ref{lem-expansion}}=30k^2$ (and $C_{\ref{lem-expansion}}$ an arbitrary shortest cycle in $
H$, which will not play a role here), we can find, for each $i,j\in [k]$ an $(m^{10K},m)$-expansion $F_{i,j}\subseteq H$ of $v_i$  so that the sets $V(F_{i,j})\setminus \{v_i\}$ are pairwise disjoint over $i,j\in [k]$.

Let $f:[K]\to [k]^{(2)}$ be a bijection and let $g,h:[K]\to [k]$ be such that $f(i)=\{g(i),h(i)\}$ for each $i\in [K]$.
For each $i\in [K]$, using Proposition~\ref{prop-trimming}, let $H_{i,1}\subseteq F_{g(i),h(i)}$ be such that  $H_{i,1}$ is an $(m^{10(K+1-i)},m)$-expansion of $v_{g(i)}$ and let $H_{i,2}\subseteq F_{h(i),g(i)}$ be such that $H_{i,2}$ is an $(m^{10(K+1-i)},m)$-expansion of $v_{h(i)}$.

We shall connect pairs of core vertices, in the order given by $f$. For each $i\in[K]$, the expansions $H_{i,1}$ and $H_{i,2}$ will be used to connect $v_{g(i)}$ and $v_{h(i)}$. We will make sure that the expansions that are not yet used are protected. More precisely, we will find paths $P_1,\ldots,P_K$, each with length $\ell$, so that the following hold.
\stepcounter{propcounter}
\begin{enumerate}[label = {\bfseries \Alph{propcounter}\arabic{enumi}}]
  \item For each $i\in [K]$, $P_i$ is a $v_{g(i)},v_{h(i)}$-path with length $\ell$.\label{this1}
  \item For each $i\in [K]$, $V(P_i)$ is disjoint from\label{this2}
  $$
  U_i:=(\{v_j:j\in [k]\}\cup (\cup_{j>i}(V(H_{j,1})\cup V(H_{j,2})))\cup (\cup_{j<i}V(P_j))) \setminus \{v_{g(i)},v_{h(i)}\}.
  $$
\end{enumerate}

This is sufficient to prove the theorem. Indeed, by \ref{this2}, for each $1\leq i< j\leq K$, as $V(P_i)\setminus (\{v_{g(i)},v_{h(i)}\}\cap \{v_{g(j)},v_{h(j)}\}) \subseteq U_j$, we have that $P_i$ and $P_j$ are internally disjoint. Therefore, $\cup_{i\in [K]}P_i$ is a copy of $\tk_{k}^{(\ell)}$.

Suppose then that $1\leq i\leq K$, and we have found paths $P_1,\ldots,P_{i-1}$ satisfying \ref{this1} and \ref{this2}. Note that
$$
|U_i|\leq k+\sum_{j>i}2m^{10(K+1-j)}+\sum_{j<i}\ell\leq k+4m^{10(K-i)}+k^2m^{3}\leq \frac{m^{10(K+1-i)}}{m^3}= \frac{|H_{i,1}|}{m^3}=\frac{|H_{i,2}|}{m^3}.
$$
Therefore, by Lemma~\ref{lem-finalconnect} with $(v_1,F_1,v_2,F_2,U)_{\ref{lem-finalconnect}}=(v_{g(i)},H_{i,1},v_{h(i)},H_{i,2}, U_i)$, there is a path $P_{i}$ with length $\ell$ between $v_{g(i)}$ and $v_{h(i)}$ which does not intersect $U_i$. This completes the proof.
\end{proof}

\subsection{Proof of Theorem~\ref{mainthm}}\label{sec:mainthmfinal}
Finally, we combine Lemmas~\ref{lem-expansion} and~\ref{lem-finalconnect} to prove Theorem~\ref{mainthm}.

\begin{proof}[Proof of Theorem~\ref{mainthm}]
Let $\eps_1>0$ be such that the property in Lemma~\ref{lem-finalconnect} holds. Let $k=10$, let $d_0=d_0(\ep_1,\ep_2)$ be large and let $n\geq d\geq d_0$. Suppose then that $H$ is a $\tk_{d/2}^{(2)}$-free bipartite $n$-vertex $(\eps_1,\eps_2d)$-expander
with $\delta(H)\geq d$ and let $x,y\in V(H)$ be distinct. Let $\ell\in [\log^7n,n/\log^{12}n]$ satisfy $\ell=\pi(x,y,H)\mod 2$. We will show that $H$ contains an $x,y$ path with length $\ell$.

Let $m=\frac{800}{\eps_1}\log^3n$ and $D=\log^{10}n$. Then, by Lemma~\ref{lem-expansion} (applied with $C$ taken to be an arbitrary shortest cycle in $H$), there are vertex disjoint graphs $F_x,F_y\subseteq H$ so that $F_x$ is a $(D,m)$-expansion of $x$ and $F_y$ is a $(D,m)$-expansion of $y$. Then, by Lemma~\ref{lem-finalconnect} with $U=\varnothing$, there is a $x,y$-path with length $\ell$ in $H$, as required.
\end{proof}

\section{Proof of Theorem~\ref{thm:chrom}}\label{sec:chrompf}
We will now prove Theorem~\ref{thm:chrom} using Theorem~\ref{mainthm}. For convenience, before we discuss the proof further, we will prove the following corollary of Theorem~\ref{mainthm}.

\begin{cor}\label{mainforchrom}
For each $\eps>0$, there is some $d_0$ such that the following holds for each $d\geq d_0$. If a graph $G$ has $d(G)\geq 8d$, then it contains a connected bipartite subgraph $H$ for which there is some positive integer $\ell$ such that the following holds.
\stepcounter{propcounter}
\begin{enumerate}[label = {\bfseries \Alph{propcounter}}]
  \item For any $u,v\in V(H)$ with $u\neq v$ and $t\in [\ell,\ell\cdot d^{1-\eps}]$ with $t=\pi(u,v,H)\mod 2$, there is a $u,v$-path in $H$ with length $t$.\label{callback}
\end{enumerate}
\end{cor}
\begin{proof} Let $\eps_1>0$ be such that the properties in Theorem~\ref{mainthm} and Corollary~\ref{cor-expander} hold, and let $\ep_2=1/10$. Let $d_0=d_0(\eps,\ep_1,\eps_2)$ be large.
Suppose then that the graph $G$ has $d(G)\geq 8d$.

If possible, let $H\subseteq G$ be a copy of $\tk^{(2)}_{d/2}$, and let $\ell=6$. Then, as $d\geq d_0(\eps,\eps_1,\eps_2)$ is large, for any two distinct vertices $u,v$ in $H$ and any integer  $t\in [\ell,\ell\cdot d^{1-\eps}]\subseteq [6, d-8]$ with $t=\pi(u,v,H) \mod 2$, it is easy to see that there is a $u,v$-path in $H$ with length $t$.

Assume then that $G$ is $\tk^{(2)}_{d/2}$-free. By Corollary~\ref{cor-expander}, $G$ contains a bipartite $(\eps_1,\eps_2d)$-expander $H$ with $\delta(H)\geq d(G)/8\geq d$. Let $\ell=\log^7 |H|$. Note that $|H|\ge \de(H)+1\ge d+1$, and $d\geq d_0=d_0(\eps,\eps_1,\eps_2)$ is large, so that $|H|/\log^{12} |H|\geq \ell\cdot d^{1-\eps}$.
As $n\geq d\geq d_0(\eps,\eps_1,\eps_2)$ is large, by Theorem~\ref{mainthm}, for any two distinct vertices $u,v$ in $H$ and any integer $t\in [\ell,\ell\cdot d^{1-\eps}]\subseteq [\log^7 |H|, |H|/\log^{12} |H|]$ with $t=\pi(u,v,H) \mod 2$, there is a $u,v$-path in $H$ with length $t$.
\end{proof}

The following sketch is illustrated in Figure~\ref{fig:chrom}. For Theorem~\ref{thm:chrom}, we have a graph $G$ with $\chi(G)= k$ and wish to find a long interval in the set of odd cycle lengths in $G$. Letting $d\approx k/30$, we first find a maximal collection $H_i$, $i\in [s]$, of edge disjoint bipartite graphs and corresponding integers $\ell_i$, $i\in [s]$, which satisfy \ref{callback}. As $\chi(G)=k\approx 30d$ is large, it will follow from the maximality of this collection and Corollary~\ref{mainforchrom} that $\cup_{i\in [s]}H_i$ has high enough chromatic number that it must contain some odd cycle.

Now, say each bipartite $H_i$ has vertex classes $A_i$ and $B_i$, and consider the auxilliary graph $K$ formed from $\cup_{i\in [s]}H_i$ by including any missing edges between $A_i$ and $B_i$ for each $i\in [s]$. As $\cup_{i\in [s]}H_i$ has an odd cycle, so does $K$. Consider a shortest odd cycle $C$ in $K$. Each edge in $C$, say the edge $e$ between $A_{i(e)}$ and $B_{i(e)}$, can be replaced with a path with any odd length in $[\ell_{i(e)},\ell_{i(e)}\cdot k^{1-o(1)}]$ by \ref{callback}. Roughly speaking, doing this for each edge in $C$ creates cycles with all possible odd lengths in $[\ell,\ell\cdot k^{1-o(1)}]$, with $\ell=\sum_{e\in E(C)}\ell_{i(e)}$.

\begin{figure}[h]
\centering



\begin{tikzpicture}

\def\vxrad{0.05cm}


\begin{scope}[xshift=0cm,yshift=0cm,rotate=90]
\def \eh{0.75}
\def \ew{0.25}
\def \spacer{\ew}

\draw [fill = gray] ($(0*\ew,-0.1*\eh)$) rectangle ($(0*\ew,\eh)+(\spacer+2*\ew,\eh-0.1*\eh)$);
\draw [fill=white] ($(0*\ew,\eh-0.1*\eh)$) ellipse (\ew cm and \eh cm);
\draw [fill=white] ($(0*\ew,\eh-0.1*\eh)+(\spacer+2*\ew,0)$) ellipse (\ew cm and \eh cm);

\draw [white] ($(1.5*\ew,1.5*\eh)$) node {\footnotesize $H_1$};

\coordinate (A) at ($(0*\ew,0)$);
\coordinate (B) at ($(0*\ew,2*\eh-0.1*\eh-0.1*\eh)$);
\coordinate (C) at ($(0*\ew+\spacer+2*\ew,0.1*\eh-0.1*\eh)$);
\coordinate (D) at ($(0*\ew+\spacer+2*\ew,2*\eh-0.1*\eh-0.1*\eh)$);
\draw [fill] (A) circle [radius=\vxrad];
\draw [fill] (C) circle [radius=\vxrad];

\end{scope}


\begin{scope}[xshift=0cm,yshift=0.75cm,rotate=45]
\def \eh{0.5}
\def \ew{0.25}
\def \spacer{\ew}

\draw [fill = gray] ($(0*\ew,-0.1*\eh)$) rectangle ($(0*\ew,\eh)+(\spacer+2*\ew,\eh-0.1*\eh)$);
\draw [fill=white] ($(0*\ew,\eh-0.1*\eh)$) ellipse (\ew cm and \eh cm);
\draw [fill=white] ($(0*\ew,\eh-0.1*\eh)+(\spacer+2*\ew,0)$) ellipse (\ew cm and \eh cm);

\draw [white] ($(1.5*\ew,1.5*\eh)$) node {\footnotesize $H_2$};

\coordinate (A) at ($(0*\ew,0)$);
\coordinate (B) at ($(0*\ew,2*\eh-0.1*\eh-0.1*\eh)$);
\coordinate (C) at ($(0*\ew+\spacer+2*\ew,0.1*\eh-0.1*\eh)$);
\coordinate (D) at ($(0*\ew+\spacer+2*\ew,2*\eh-0.1*\eh-0.1*\eh)$);
\draw [fill] (A) circle [radius=\vxrad];
\draw [fill] (C) circle [radius=\vxrad];

\end{scope}


\begin{scope}[xshift=0cm,yshift=0cm,rotate=270]
\def \eh{0.95}
\def \ew{0.25}
\def \spacer{\ew}

\draw [fill = gray] ($(0*\ew,-0.1*\eh)$) rectangle ($(0*\ew,\eh)+(\spacer+2*\ew,\eh-0.1*\eh)$);
\draw [fill=white] ($(0*\ew,\eh-0.1*\eh)$) ellipse (\ew cm and \eh cm);
\draw [fill=white] ($(0*\ew,\eh-0.1*\eh)+(\spacer+2*\ew,0)$) ellipse (\ew cm and \eh cm);

\coordinate (A) at ($(0*\ew,0)$);
\coordinate (B) at ($(0*\ew,2*\eh-0.1*\eh-0.1*\eh)$);
\coordinate (C) at ($(0*\ew+\spacer+2*\ew,0.1*\eh-0.1*\eh)$);
\coordinate (D) at ($(0*\ew+\spacer+2*\ew,2*\eh-0.1*\eh-0.1*\eh)$);
\draw [fill] (A) circle [radius=\vxrad];
\draw [fill] (B) circle [radius=\vxrad];

\draw [white] ($(1.5*\ew,0.2*\eh)$) node {\footnotesize $H_5$};

\end{scope}


\begin{scope}[xshift=1.7cm,yshift=0cm,rotate=0]
\def \eh{0.7}
\def \ew{0.25}
\def \spacer{\ew}

\draw [fill = gray] ($(0*\ew,-0.1*\eh)$) rectangle ($(0*\ew,\eh)+(\spacer+2*\ew,\eh-0.1*\eh)$);
\draw [fill=white] ($(0*\ew,\eh-0.1*\eh)$) ellipse (\ew cm and \eh cm);
\draw [fill=white] ($(0*\ew,\eh-0.1*\eh)+(\spacer+2*\ew,0)$) ellipse (\ew cm and \eh cm);

\coordinate (A) at ($(0*\ew,0)$);
\coordinate (B) at ($(0*\ew,2*\eh-0.1*\eh-0.1*\eh)$);
\coordinate (C) at ($(0*\ew+\spacer+2*\ew,0.1*\eh-0.1*\eh)$);
\coordinate (D) at ($(0*\ew+\spacer+2*\ew,2*\eh-0.1*\eh-0.1*\eh)$);
\draw [fill] (A) circle [radius=\vxrad];
\draw [fill] (B) circle [radius=\vxrad];

\draw [white] ($(1.5*\ew,0.2*\eh)$) node {\footnotesize $H_4$};

\end{scope}


\begin{scope}[xshift=0.525cm,yshift=1.28cm,rotate=0]
\def \eh{0.75}
\def \ew{0.395}
\def \spacer{\ew}

\draw [fill = gray] ($(0*\ew,-0.1*\eh)$) rectangle ($(0*\ew,\eh)+(\spacer+2*\ew,\eh-0.1*\eh)$);
\draw [fill=white] ($(0*\ew,\eh-0.1*\eh)$) ellipse (\ew cm and \eh cm);
\draw [fill=white] ($(0*\ew,\eh-0.1*\eh)+(\spacer+2*\ew,0)$) ellipse (\ew cm and \eh cm);

\coordinate (A) at ($(0*\ew,0)$);
\coordinate (B) at ($(0*\ew,2*\eh-0.1*\eh-0.1*\eh)$);
\coordinate (C) at ($(0*\ew+\spacer+2*\ew,0.1*\eh-0.1*\eh)$);
\coordinate (D) at ($(0*\ew+\spacer+2*\ew,2*\eh-0.1*\eh-0.1*\eh)$);
\draw [fill] (A) circle [radius=\vxrad];
\draw [fill] (C) circle [radius=\vxrad];

\draw [white] ($(1.5*\ew,1.5*\eh)$) node {\footnotesize $H_3$};

\end{scope}

\begin{scope}[xshift=0cm,yshift=0cm,rotate=90]
\def \eh{0.75}
\def \ew{0.25}
\def \spacer{\ew}

\draw  ($(0*\ew,\eh-0.1*\eh)$) ellipse (\ew cm and \eh cm);
\draw  ($(0*\ew,\eh-0.1*\eh)+(\spacer+2*\ew,0)$) ellipse (\ew cm and \eh cm);
\end{scope}


\begin{scope}[xshift=0cm,yshift=0.75cm,rotate=45]
\def \eh{0.5}
\def \ew{0.25}
\def \spacer{\ew}

\draw  ($(0*\ew,\eh-0.1*\eh)$) ellipse (\ew cm and \eh cm);
\draw  ($(0*\ew,\eh-0.1*\eh)+(\spacer+2*\ew,0)$) ellipse (\ew cm and \eh cm);
\end{scope}


\begin{scope}[xshift=0cm,yshift=0cm,rotate=270]
\def \eh{0.95}
\def \ew{0.25}
\def \spacer{\ew}

\draw  ($(0*\ew,\eh-0.1*\eh)$) ellipse (\ew cm and \eh cm);
\draw  ($(0*\ew,\eh-0.1*\eh)+(\spacer+2*\ew,0)$) ellipse (\ew cm and \eh cm);
\end{scope}


\begin{scope}[xshift=1.7cm,yshift=0cm,rotate=0]
\def \eh{0.7}
\def \ew{0.25}
\def \spacer{\ew}

\draw  ($(0*\ew,\eh-0.1*\eh)$) ellipse (\ew cm and \eh cm);
\draw  ($(0*\ew,\eh-0.1*\eh)+(\spacer+2*\ew,0)$) ellipse (\ew cm and \eh cm);
\end{scope}


\begin{scope}[xshift=0.525cm,yshift=1.28cm,rotate=0]
\def \eh{0.75}
\def \ew{0.395}
\def \spacer{\ew}

\draw  ($(0*\ew,\eh-0.1*\eh)$) ellipse (\ew cm and \eh cm);
\draw  ($(0*\ew,\eh-0.1*\eh)+(\spacer+2*\ew,0)$) ellipse (\ew cm and \eh cm);
\end{scope}

\end{tikzpicture}\hspace{1cm}\begin{tikzpicture}

\def\vxrad{0.05cm}
\def\ellcol{lightgray}


\begin{scope}[xshift=0cm,yshift=0cm,rotate=90]
\def \eh{0.75}
\def \ew{0.25}
\def \spacer{\ew}

\draw [lightgray, fill = lightgray] ($(0*\ew,-0.1*\eh)$) rectangle ($(0*\ew,\eh)+(\spacer+2*\ew,\eh-0.1*\eh)$);
\draw [lightgray, fill=white] ($(0*\ew,\eh-0.1*\eh)$) ellipse (\ew cm and \eh cm);
\draw [lightgray, fill=white] ($(0*\ew,\eh-0.1*\eh)+(\spacer+2*\ew,0)$) ellipse (\ew cm and \eh cm);

\coordinate (A) at ($(0*\ew,0)$);
\coordinate (B) at ($(0*\ew,2*\eh-0.1*\eh-0.1*\eh)$);
\coordinate (C) at ($(0*\ew+\spacer+2*\ew,0.1*\eh-0.1*\eh)$);
\coordinate (D) at ($(0*\ew+\spacer+2*\ew,2*\eh-0.1*\eh-0.1*\eh)$);
\draw [fill] (A) circle [radius=\vxrad];
\draw [fill] (C) circle [radius=\vxrad];

\end{scope}


\begin{scope}[xshift=0cm,yshift=0.75cm,rotate=45]
\def \eh{0.5}
\def \ew{0.25}
\def \spacer{\ew}

\draw [fill = lightgray,lightgray] ($(0*\ew,-0.1*\eh)$) rectangle ($(0*\ew,\eh)+(\spacer+2*\ew,\eh-0.1*\eh)$);
\draw [\ellcol,fill=white] ($(0*\ew,\eh-0.1*\eh)$) ellipse (\ew cm and \eh cm);
\draw [\ellcol,fill=white] ($(0*\ew,\eh-0.1*\eh)+(\spacer+2*\ew,0)$) ellipse (\ew cm and \eh cm);

\coordinate (A) at ($(0*\ew,0)$);
\coordinate (B) at ($(0*\ew,2*\eh-0.1*\eh-0.1*\eh)$);
\coordinate (C) at ($(0*\ew+\spacer+2*\ew,0.1*\eh-0.1*\eh)$);
\coordinate (D) at ($(0*\ew+\spacer+2*\ew,2*\eh-0.1*\eh-0.1*\eh)$);
\draw [fill] (A) circle [radius=\vxrad];
\draw [fill] (C) circle [radius=\vxrad];

\end{scope}


\begin{scope}[xshift=0cm,yshift=0cm,rotate=270]
\def \eh{0.95}
\def \ew{0.25}
\def \spacer{\ew}

\draw [lightgray,fill = lightgray] ($(0*\ew,-0.1*\eh)$) rectangle ($(0*\ew,\eh)+(\spacer+2*\ew,\eh-0.1*\eh)$);
\draw [\ellcol,fill=white] ($(0*\ew,\eh-0.1*\eh)$) ellipse (\ew cm and \eh cm);
\draw [\ellcol,fill=white] ($(0*\ew,\eh-0.1*\eh)+(\spacer+2*\ew,0)$) ellipse (\ew cm and \eh cm);

\coordinate (A) at ($(0*\ew,0)$);
\coordinate (B) at ($(0*\ew,2*\eh-0.1*\eh-0.1*\eh)$);
\coordinate (C) at ($(0*\ew+\spacer+2*\ew,0.1*\eh-0.1*\eh)$);
\coordinate (D) at ($(0*\ew+\spacer+2*\ew,2*\eh-0.1*\eh-0.1*\eh)$);
\draw [fill] (A) circle [radius=\vxrad];
\draw [fill] (B) circle [radius=\vxrad];

\end{scope}

,\ellcol


\begin{scope}[xshift=1.7cm,yshift=0cm,rotate=0]
\def \eh{0.7}
\def \ew{0.25}
\def \spacer{\ew}

\draw [lightgray,fill = lightgray] ($(0*\ew,-0.1*\eh)$) rectangle ($(0*\ew,\eh)+(\spacer+2*\ew,\eh-0.1*\eh)$);
\draw [\ellcol,fill=white] ($(0*\ew,\eh-0.1*\eh)$) ellipse (\ew cm and \eh cm);
\draw [\ellcol,fill=white] ($(0*\ew,\eh-0.1*\eh)+(\spacer+2*\ew,0)$) ellipse (\ew cm and \eh cm);

\coordinate (A) at ($(0*\ew,0)$);
\coordinate (B) at ($(0*\ew,2*\eh-0.1*\eh-0.1*\eh)$);
\coordinate (C) at ($(0*\ew+\spacer+2*\ew,0.1*\eh-0.1*\eh)$);
\coordinate (D) at ($(0*\ew+\spacer+2*\ew,2*\eh-0.1*\eh-0.1*\eh)$);
\draw [fill] (A) circle [radius=\vxrad];
\draw [fill] (B) circle [radius=\vxrad];

\end{scope}


\begin{scope}[xshift=0.525cm,yshift=1.28cm,rotate=0]
\def \eh{0.75}
\def \ew{0.395}
\def \spacer{\ew}

\draw [lightgray,fill = lightgray] ($(0*\ew,-0.1*\eh)$) rectangle ($(0*\ew,\eh-0.1*\eh)+(\spacer+2*\ew,\eh)$);
\draw [\ellcol,fill=white] ($(0*\ew,\eh-0.1*\eh)$) ellipse (\ew cm and \eh cm);
\draw [\ellcol,fill=white] ($(0*\ew,\eh-0.1*\eh)+(\spacer+2*\ew,0)$) ellipse (\ew cm and \eh cm);

\coordinate (A) at ($(0*\ew,0)$);
\coordinate (B) at ($(0*\ew,2*\eh-0.1*\eh-0.1*\eh)$);
\coordinate (C) at ($(0*\ew+\spacer+2*\ew,0.1*\eh-0.1*\eh)$);
\coordinate (D) at ($(0*\ew+\spacer+2*\ew,2*\eh-0.1*\eh-0.1*\eh)$);
\draw [fill] (A) circle [radius=\vxrad];
\draw [fill] (C) circle [radius=\vxrad];

\end{scope}

\begin{scope}[xshift=0cm,yshift=0cm,rotate=90]
\def \eh{0.75}
\def \ew{0.25}
\def \spacer{\ew}

\draw [\ellcol] ($(0*\ew,\eh-0.1*\eh)$) ellipse (\ew cm and \eh cm);
\draw [\ellcol] ($(0*\ew,\eh-0.1*\eh)+(\spacer+2*\ew,0)$) ellipse (\ew cm and \eh cm);
\end{scope}


\begin{scope}[xshift=0cm,yshift=0.75cm,rotate=45]
\def \eh{0.5}
\def \ew{0.25}
\def \spacer{\ew}

\draw [\ellcol] ($(0*\ew,\eh-0.1*\eh)$) ellipse (\ew cm and \eh cm);
\draw [\ellcol] ($(0*\ew,\eh-0.1*\eh)+(\spacer+2*\ew,0)$) ellipse (\ew cm and \eh cm);
\end{scope}


\begin{scope}[xshift=0cm,yshift=0cm,rotate=270]
\def \eh{0.95}
\def \ew{0.25}
\def \spacer{\ew}

\draw [\ellcol] ($(0*\ew,\eh-0.1*\eh)$) ellipse (\ew cm and \eh cm);
\draw [\ellcol] ($(0*\ew,\eh-0.1*\eh)+(\spacer+2*\ew,0)$) ellipse (\ew cm and \eh cm);
\end{scope}


\begin{scope}[xshift=1.7cm,yshift=0cm,rotate=0]
\def \eh{0.7}
\def \ew{0.25}
\def \spacer{\ew}

\draw [\ellcol] ($(0*\ew,\eh-0.1*\eh)$) ellipse (\ew cm and \eh cm);
\draw [\ellcol] ($(0*\ew,\eh-0.1*\eh)+(\spacer+2*\ew,0)$) ellipse (\ew cm and \eh cm);
\end{scope}


\begin{scope}[xshift=0.525cm,yshift=1.28cm,rotate=0]
\def \eh{0.75}
\def \ew{0.395}
\def \spacer{\ew}

\draw [\ellcol] ($(0*\ew,\eh-0.1*\eh)$) ellipse (\ew cm and \eh cm);
\draw [\ellcol] ($(0*\ew,\eh-0.1*\eh)+(\spacer+2*\ew,0)$) ellipse (\ew cm and \eh cm);
\end{scope}


\begin{scope}[xshift=0cm,yshift=0cm,rotate=90]
\def \eh{0.75}
\def \ew{0.25}
\def \spacer{\ew}

\coordinate (A) at ($(0*\ew,0)$);
\coordinate (B) at ($(0*\ew,2*\eh-0.1*\eh-0.1*\eh)$);
\coordinate (C) at ($(0*\ew+\spacer+2*\ew,0.1*\eh-0.1*\eh)$);
\coordinate (D) at ($(0*\ew+\spacer+2*\ew,2*\eh-0.1*\eh-0.1*\eh)$);

\def\num{2}
\foreach \x in {1,...,\num}
{
\draw ($(A)+(0,1*\eh*\x/\num)$) -- ($(C)+(0,1*\eh*\x/\num)$);
}
\foreach \x in {1,...,\num}
{
\draw ($(A)+(0,1*\eh*\x/\num-1*\eh/\num)$) -- ($(C)+(0,1*\eh*\x/\num)$);
}
\draw (C) -- ($(A)+(0,1*\eh)$);
\end{scope}


\begin{scope}[xshift=0cm,yshift=0.75cm,rotate=45]
\def \eh{0.5}
\def \ew{0.25}
\def \spacer{\ew}

\coordinate (A) at ($(0*\ew,0)$);
\coordinate (B) at ($(0*\ew,2*\eh-0.1*\eh-0.1*\eh)$);
\coordinate (C) at ($(0*\ew+\spacer+2*\ew,0.1*\eh-0.1*\eh)$);
\coordinate (D) at ($(0*\ew+\spacer+2*\ew,2*\eh-0.1*\eh-0.1*\eh)$);

\def\num{1}
\foreach \x in {1,...,\num}
{
\draw ($(A)+(0,0.8*\eh*\x/\num)$) -- ($(C)+(0,0.8*\eh*\x/\num)$);
}
\foreach \x in {1,...,\num}
{
\draw ($(A)+(0,0.8*\eh*\x/\num-0.8*\eh/\num)$) -- ($(C)+(0,0.8*\eh*\x/\num)$);
}
\draw ($(A)+(0,0.8*\eh)$) -- (C);
\end{scope}


\begin{scope}[xshift=0cm,yshift=0cm,rotate=270]
\def \eh{0.95}
\def \ew{0.25}
\def \spacer{\ew}

\coordinate (A) at ($(0*\ew,0)$);
\coordinate (B) at ($(0*\ew,2*\eh-0.1*\eh-0.1*\eh)$);
\coordinate (C) at ($(0*\ew+\spacer+2*\ew,0.1*\eh-0.1*\eh)$);
\coordinate (D) at ($(0*\ew+\spacer+2*\ew,2*\eh-0.1*\eh-0.1*\eh)$);


\def\num{10}
\foreach \x in {1,...,\num}
{
\draw ($(A)+(0,1.8*\eh*\x/\num)$) -- ($(C)+(0,1.8*\eh*\x/\num)$);
}
\foreach \x in {1,...,\num}
{
\draw ($(A)+(0,1.8*\eh*\x/\num-1.8*\eh/\num)$) -- ($(C)+(0,1.8*\eh*\x/\num)$);
}
\end{scope}


\begin{scope}[xshift=1.7cm,yshift=0cm,rotate=0]
\def \eh{0.7}
\def \ew{0.25}
\def \spacer{\ew}

\coordinate (A) at ($(0*\ew,0)$);
\coordinate (B) at ($(0*\ew,2*\eh-0.1*\eh-0.1*\eh)$);
\coordinate (C) at ($(0*\ew+\spacer+2*\ew,0.1*\eh-0.1*\eh)$);
\coordinate (D) at ($(0*\ew+\spacer+2*\ew,2*\eh-0.1*\eh-0.1*\eh)$);

\def\num{2}
\foreach \x in {1,...,\num}
{
\draw ($(A)+(0,1.8*\eh*\x/\num)$) -- ($(C)+(0,1.1*\eh*\x/\num)$);
}
\foreach \x in {1,...,\num}
{
\draw ($(A)+(0,1.8*\eh*\x/\num-1.8*\eh/\num)$) -- ($(C)+(0,1.1*\eh*\x/\num)$);
}
\end{scope}


\begin{scope}[xshift=0.525cm,yshift=1.28cm,rotate=0]
\def \eh{0.75}
\def \ew{0.395}
\def \spacer{\ew}

\coordinate (A) at ($(0*\ew,0)$);
\coordinate (B) at ($(0*\ew,2*\eh-0.1*\eh-0.1*\eh)$);
\coordinate (C) at ($(0*\ew+\spacer+2*\ew,0.1*\eh-0.1*\eh)$);
\coordinate (D) at ($(0*\ew+\spacer+2*\ew,2*\eh-0.1*\eh-0.1*\eh)$);

\def\num{10}
\foreach \x in {1,...,\num}
{
\draw ($(A)+(0,1.8*\eh*\x/\num)$) -- ($(C)+(0,1.8*\eh*\x/\num)$);
}
\foreach \x in {1,...,\num}
{
\draw ($(A)+(0,1.8*\eh*\x/\num-1.8*\eh/\num)$) -- ($(C)+(0,1.8*\eh*\x/\num)$);
}
\draw (B) -- (C);
\end{scope}

\end{tikzpicture}


\caption{An illustration of the proof of Theorem~\ref{thm:chrom}. As seen on the left, we find a collection of edge disjoint bipartite expander graphs, here $H_1,\ldots,H_5$, so that $H_i$ intersects with $H_{i-1}$ and $H_{i+1}$ on at least one vertex each (working mod $5$ in the indices), and any cycle around the `cycle of subgraphs' is odd.
\newline
We then form different length odd cycles by choosing short paths between the intersecting vertices in some of the expanders $H_i$, while varying the length of the paths between a vertex disjoint collection of the expanders (here, $H_3$ and $H_5$).\label{fig:chrom}}
\end{figure}
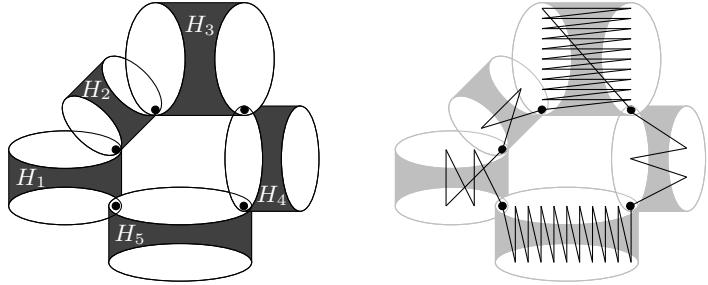

The main complication omitted in this sketch is that the paths replacing the edges in $C$ need to be internally vertex disjoint from each other and $V(C)$. To ensure this, we only replace some of the edges with paths of varying length, only doing so for a suitable large collection of the edges on $C$. If two edges on $C$ are far apart, yet their containing bipartite graphs $H_i$ intersect, then we can find a shorter odd cycle in $K$ than $C$, a contradiction. Therefore, we choose a subset $E\subseteq E(C)$ of edges which are pairwise nicely separated on $C$ (and which maximises $\sum_{e\in E}\ell_{i(e)}$ subject to this). We replace each edge in $E$ by an odd length path with length from $[\ell_{i(e)},\ell_{i(e)}k^{1-o(1)}]$, while we (essentially) replace each edge in $E(C)\setminus E$ with some minimal path between the same vertices in the corresponding graph $H_i$, in order to connect the paths corresponding to $E$ into a cycle.

\medskip

Recalling Definition~\ref{pidefn}, we use the following simple proposition.
\begin{prop}\label{simpleclass}
Given any connected bipartite graph $H$ containing three distinct vertices $a_1,a_2$ and $a_3$, we have
\begin{equation}\label{gen0}
\pi(a_1,a_3,H)+\pi(a_2,a_3,H)-\pi(a_1,a_2,H)\in \{0,2\}.
\end{equation}
Furthermore, for any (not necessarily distinct) vertices $a_1,a_2$ and $a_3$, we have
\begin{equation}\label{gen}
\pi(a_1,a_3,H)+\pi(a_2,a_3,H)=\pi(a_1,a_2,H)\mod 2
\end{equation}
\end{prop}
\begin{proof} First, suppose that $a_1,a_2$ and $a_3$ are distinct. Let the vertex classes of $H$ be $A$ and $B$, and assume without loss of generality that $a_3\in A$. If $a_1,a_2\in A$, then $\pi(a_1,a_3,H)+\pi(a_2,a_3,H)-\pi(a_1,a_2,H)=2+2-2=2$. If $a_1,a_2\in B$, then $\pi(a_1,a_3,H)+\pi(a_2,a_3,H)-\pi(a_1,a_2,H)=1+1-2=0$.
 We may then assume that $a_1\in A$ and $a_2\in B$, and so $\pi(a_1,a_3,H)+\pi(a_2,a_3,H)-\pi(a_1,a_2,H)=2+1-1=2$. Therefore, whenever $a_1,a_2$ and $a_3$ are distinct, \eqref{gen0} holds, and furthermore \eqref{gen} holds in this case as well.

  Now, note that if $a_1=a_2=a$, then \eqref{gen} holds as $2\pi(a,a_3,H)=0=\pi(a,a,H)\mod 2$, while if $a_1=a_3=a$, then $0+\pi(a_2,a,H)=\pi(a,a_2,H)\mod 2$, and, similarly, \eqref{gen} holds if $a_2=a_3$. This completes the remaining cases when $a_1,a_2$ and $a_3$ are not necessarily distinct.
\end{proof}

We are now ready to prove Theorem~\ref{thm:chrom}, which we do throughout Section~\ref{subsec:chrompf}.

\subsection{Proof of Theorem~\ref{thm:chrom}}\label{subsec:chrompf}
Let $\eps>0$. To prove Theorem~\ref{thm:chrom}, we will show that there is some $k_0\in \N$ such that the following holds for each $k\geq k_0$. If $G$ is a graph with chromatic number $k$, then, for some $\ell\in \mathbb{N}$, $\mathcal{C}(G)$ contains every odd integer in $[\ell,\ell\cdot k^{1-\eps}]$.

Let then $d_0$ be large enough that $d^{1-\eps/4}\geq (30(d+1))^{1-\eps/2}$ holds for each $d\geq d_0$ and the property in Corollary~\ref{mainforchrom} holds  for $d_0$ with $\eps_{\ref{mainforchrom}}=\eps/4$. Let $k_0=30 d_0$. Suppose $k\geq k_0$ and that the graph $G$ has $\chi(G)= k$.
Let $d=\lfloor k/30\rfloor$.

As outlined at the beginning of this section, using the high chromatic number of $G$, we first find a minimal `cycle of subgraphs' that could potentially offer many distinct odd cycle lengths (in Section~\ref{sec:min}). We then prove that non-consecutive subgraphs in this cycle are vertex disjoint (in Section~\ref{sec:nonadj}). We then choose some of these non-consecutive subgraphs in which to vary the path lengths (in Section~\ref{sec:tovary}). Finally, we take different path lengths in the chosen subgraphs and connect them up into a cycle, getting many odd cycles (in Section~\ref{sec:vary}).

\subsubsection{A minimal `cycle of subgraphs'}\label{sec:min}
Let $H_1,\ldots,H_s$ be a maximal collection of edge disjoint connected bipartite subgraphs of $G$ such that, for each $i\in [s]$, there is a positive integer $\ell_i$ for which the following holds.

\stepcounter{propcounter}
\begin{enumerate}[label = {\bfseries \Alph{propcounter}}]
  \item For any two distinct vertices $u,v$ in $H_i$ and any integer $t\in [\ell_i, \ell_i\cdot k^{1-\eps/2}]$ with $t=\pi(u,v,H_i) \mod 2$, there is a $u,v$-path in $H_i$ with length $t$.\label{bigprop}
\end{enumerate}
 Let $H=\cup_{i\in [s]}H_i$. Let $G'=G\setminus H$. As $d^{1-\eps/4}\geq (30(d+1))^{1-\eps/2}\geq k^{1-\eps/2}$, by the maximality of the collection $H_1,\ldots, H_s$ and Corollary~\ref{mainforchrom}, $G'$ has no subgraph with average degree at least $8d$. Therefore, every subgraph of $G'$ has a vertex with degree in that subgraph less than $8d$, and hence, as is well known, chromatic number less than $8d$.

 We will now show that $\chi(H)\geq 3$. Indeed, suppose to the contrary that $\chi(H)< 3$. Let $c_1:V(G')\to [8d]$ be a proper colouring of $G'$ and let $c_2:V(H)\to [2]$ be a proper colouring of $H$. Then, $c:V(G)\to [8d]\times [2]$ defined by $c(v)=(c_1(v),c_2(v))$ is easily seen to be a proper colouring of $G$. Therefore, $\chi(G)\leq 16d<k$, a contradiction. Thus, we have that $\chi(H)\geq 3$.

Therefore, we can choose an odd cycle $C$ in $H$. Say that $C$ has length $r'$ and vertices $a_1\ldots a_{r'+1}$ where $a_1=a_{r'+1}$. We now take a certain minimal sequence, where $C$ will demonstrate that such a sequence exists. That is, we take a sequence
$$
\mathcal{S}=b_1F_1 b_2F_2 b_3\ldots  b_r F_r b_{r+1}
$$
such that, setting $r(\mathcal{S})=r$, we have
\stepcounter{propcounter}
\begin{enumerate}[label = {\bfseries \Alph{propcounter}\arabic{enumi}}]
	\item $b_1,\ldots,b_r$ are distinct vertices and $b_1=b_{r+1}$,\label{S1}

	\item for each $i\in [r]$, $F_i\in \{H_1,\ldots,H_s\}$ and $b_{i},b_{i+1}\in V(F_{i})$,\label{S2}

	\item $\pi(\mathcal{S}):=\sum_{i=1}^{r}\pi(b_i,b_{i+1},F_i)$ is odd, and\label{S3}

    \item subject to~\ref{S1}--\ref{S3}, $\pi(\mathcal{S})+r(\mathcal{S})$ is minimised.\label{S4}
\end{enumerate}
   Indeed, such a sequence exists as the sequence $\cS'=a_1G_1a_2G_2a_3\ldots a_{r'} G_{r'}a_{r'+1}$ satisfies~\ref{S1}--\ref{S3}, where for each $i\in[r']$, $G_i=H_j$ for the $j\in [s]$  with $a_ia_{i+1}\in E(H_j)$, and we have that $\pi(\cS')=\sum_{i=1}^{r'}\pi(a_i,a_{i+1},G_i)=r'$ is odd.

Note that we must have $r\geq 2$. Indeed,
if $r=1$, then \ref{S1} implies that $\pi(b_1,b_2,F_1)=\pi(b_1,b_1,F_1)=0$, violating \ref{S3}.

\subsubsection{Non-consecutive subgraphs are vertex disjoint}\label{sec:nonadj}

We will now use the minimality of $\mathcal{S}$ (that is, \ref{S4}) to infer two key properties. These are (roughly) that non-consecutive graphs in $F_1,\ldots,F_{r}$ are vertex disjoint (Claim~\ref{oddtwo}) and that each graph is a different graph in $\{H_1,\ldots,H_s\}$ (Claim~\ref{norep}).

\begin{claim}\label{oddtwo}
For each $i,j\in [r]$ with $i\neq j$ and $i\neq j\pm 1\mod r$, $F_i$ and $F_j$ are vertex disjoint.
\end{claim}
\begin{poc}
Suppose, for contradiction, we have some distinct $i,j\in [r]$ with $i\neq j\pm 1\mod r$ (and thus $r\geq 4$) and that $F_i$ and $F_j$ are not vertex disjoint. We consider separately the case when $F_i$ and $F_j$  share some vertex not in $\{b_1,\ldots,b_r\}$ (Case I) and when they share some vertex in $\{b_1,\ldots,b_r\}$ (Case II).

\medskip

\noindent\textbf{Case I.} Suppose then that $F_i$ and $F_j$ share some vertex $a\notin\{b_1,\ldots,b_r\}$. Assume, without loss of generality, that $i<j$.
As depicted in Figure~\ref{fig:SSS}, consider the two sequences
$$
\mathcal{S}_1=b_1F_1b_2\ldots b_iF_iaF_{j}b_{j+1}\ldots b_rF_{r}b_{r+1}\quad\text{ and }\quad \mathcal{S}_2=b_{i+1}F_{i+1}b_{i+2}F_{i+2}\ldots b_jF_jaF_ib_{i+1}.
$$
We will show that one of these sequences satisfies \ref{S1}--\ref{S3} in place of $\mathcal{S}$ and contradicts the minimality of $\mathcal{S}$ in \ref{S4}. That each sequence $\mathcal{S}_1$ and $\mathcal{S}_2$ satisfies the corresponding version of \ref{S1} and \ref{S2} follows immediately from \ref{S1} and \ref{S2} and as $a\notin \{b_1,\ldots,b_r\}$ is in both $V(F_i)$ and $V(F_j)$.

\begin{figure}[b]
\centering


\begin{tikzpicture}[scale=0.7]

\def\radout{3}
\def\radin{2}
\def\rotang{45}
\def\offset{35}
\foreach \rotno in {0,1,...,10}
{
\begin{scope}[rotate={\rotang*\rotno}]

\draw (0:\radout) arc (0:\offset:\radout);
\draw (0:\radin) arc (0:\offset:\radin);

\draw [fill] (-5:2.5) circle [radius=0.05cm];
\coordinate (b\rotno) at (-5:2.5);
\coordinate (F\rotno) at (17.5:2.5);
\coordinate (A\rotno) at (10:\radin);
\coordinate (B\rotno) at (25:\radin);

\draw (0:\radin) arc (-180:0:0.5);
\draw (\offset:\radin) arc (215:\offset:0.5);

\end{scope}
}

\def\bratio{1.325}
\def\Fratio{1.35}
\def\FFratio{1}
\def\bsize{\footnotesize}
\def\Fsize{\footnotesize}
\draw ($\bratio*(b1)$) node {\bsize $b_j$};
\draw ($1.05*\bratio*(b3)$) node {\bsize $b_{i+1}$};
\draw ($\bratio*(b4)$) node {\bsize $b_i$};
\draw ($\bratio*(b5)$) node {\bsize $b_2$};
\draw ($\bratio*(b6)$) node {\bsize $b_1$};
\draw ($\bratio*(b7)$) node {\bsize $b_{r}$};
\draw ($1.1*\bratio*(b8)$) node {\bsize $b_{j+1}$};

\draw ($\FFratio*(F2)$) node {\Fsize $F_{i+1}$};
\draw ($\FFratio*(F3)$) node {\Fsize $F_i$};
\draw ($\FFratio*(F5)$) node {\Fsize $F_{1}$};
\draw ($\FFratio*(F6)$) node {\Fsize $F_r$};
\draw ($\Fratio*(F8)$) node {\Fsize $F_{j}$};

\draw [white,thick] (145:\radin) arc (145:160:\radin);


\coordinate (A) at ($0.5*(B8)+0.5*(A3)+0.5*(A8)-0.5*(B3)+0.1*(B8)-0.1*(A3)$);
\coordinate (B) at ($(A8)+0.1*(B8)-0.1*(A3)$);
\coordinate (C) at ($0.5*(A)+0.5*(B)-0.05*(B8)+0.05*(A3)$);
\coordinate (Clab) at ($0.5*(A)+0.5*(B)-0.175*(B8)+0.175*(A3)$);

\def\labelh{1.45}
\coordinate (SS) at ($\labelh*0.5*(B8)+\labelh*0.5*(A3)$);
\draw (SS) node {$\mathcal{S}_2$};

\def\labelh{-0.6}
\coordinate (S) at ($\labelh*0.5*(B8)+\labelh*0.5*(A3)$);
\draw (S) node {$\mathcal{S}_1$};

\draw (B) -- (B3);
\draw (A) -- (A3);

\draw (A)  to [out=-5,in=-15]  (B);
\draw [fill] (C) circle [radius=0.05cm];
\draw (Clab) node {\footnotesize $a$};


\end{tikzpicture}\hspace{1cm}\begin{tikzpicture}[scale=0.7]

\def\radout{3}
\def\radin{2}
\def\rotang{45}
\def\offset{35}
\foreach \rotno in {0,1,...,10}
{
\begin{scope}[rotate={\rotang*\rotno}]

\draw (0:\radout) arc (0:\offset:\radout);
\draw (0:\radin) arc (0:\offset:\radin);

\draw [fill] (-5:2.5) circle [radius=0.05cm];
\coordinate (b\rotno) at (-5:2.5);
\coordinate (F\rotno) at (17.5:2.5);
\coordinate (A\rotno) at (10:\radin);
\coordinate (B\rotno) at (25:\radin);

\draw (0:\radin) arc (-180:0:0.5);
\draw (\offset:\radin) arc (215:\offset:0.5);

\end{scope}
}

\def\bratio{1.325}
\def\Fratio{1.35}
\def\FFratio{1}
\def\bsize{\footnotesize}
\def\Fsize{\footnotesize}
\draw ($\bratio*(b1)$) node {\bsize $b_j$};
\draw ($1.05*\bratio*(b3)$) node {\bsize $b_{i+1}$};
\draw ($\bratio*(b4)$) node {\bsize $b_i$};
\draw ($\bratio*(b5)$) node {\bsize $b_2$};
\draw ($\bratio*(b6)$) node {\bsize $b_1$};
\draw ($\bratio*(b7)$) node {\bsize $b_{r}$};
\draw ($1.1*\bratio*(b8)$) node {\bsize $b_{j+1}$};

\draw ($\FFratio*(F2)$) node {\Fsize $F_{i+1}$};
\draw ($\FFratio*(F3)$) node {\Fsize $F_i$};
\draw ($\FFratio*(F5)$) node {\Fsize $F_{1}$};
\draw ($\FFratio*(F6)$) node {\Fsize $F_r$};
\draw ($\Fratio*(F8)$) node {\Fsize $F_{j}$};

\draw [white,thick] (145:\radin) arc (145:160:\radin);


\coordinate (A) at ($0.5*(B8)+0.5*(A3)+0.5*(A8)-0.5*(B3)+0.1*(B8)-0.1*(A3)$);
\coordinate (B) at ($(A8)+0.1*(B8)-0.1*(A3)$);
\coordinate (C) at ($0.5*(A)+0.5*(B)-0.05*(B8)+0.05*(A3)$);
\coordinate (Clab) at ($0.5*(A)+0.5*(B)-0.175*(B8)+0.175*(A3)$);

\def\labelh{1.45}
\coordinate (SS) at ($\labelh*0.5*(B8)+\labelh*0.5*(A3)$);
\draw (SS) node {$\mathcal{S}_4$};

\coordinate (Azz) at ($0.5*(A8)+0.5*(B3)$);
\coordinate (Bzz) at ($0.5*(B8)+0.5*(A3)$);
\coordinate (COR) at ($0.25*(B8)+0.25*(A3)+0.25*(A8)+0.25*(B3)$);

\begin{scope}[rotate around={31:(COR)}]
\foreach \rotno in {0,1,...,10}
{
\begin{scope}[rotate={\rotang*\rotno}]
\coordinate (bz\rotno) at (-5:2.5);
\coordinate (Fz\rotno) at (17.5:2.5);
\coordinate (Az\rotno) at (10:\radin);
\coordinate (Bz\rotno) at (25:\radin);
\end{scope}
}
\coordinate (Az) at ($0.5*(Bz8)+0.5*(Az3)+0.5*(Az8)-0.5*(Bz3)+0.05*(B8)-0.05*(Az3)$);
\coordinate (Bz) at ($(Az8)+0.05*(Bz8)-0.05*(Az3)$);
\coordinate (Cz) at ($0.5*(Az)+0.5*(Bz)-0.05*(B8)+0.05*(Az3)$);
\coordinate (Czlab) at ($0.5*(Az)+0.5*(Bz)-0.175*(Bz8)+0.175*(Az3)$);

\end{scope}

\coordinate (Bzzz) at ($2*(b1)-(Az)$);


\coordinate (By) at ($2*(Bz)-(Bzzz)$);

\draw (Az)  to [out=30,in=60]  (By);
\draw (Azz) to [out=-7,in=225] (By);
\draw (Bzz) to [out=-7,in=225] (Az);

\def\labelh{-0.6}
\coordinate (S) at ($\labelh*0.5*(B8)+\labelh*0.5*(A3)$);
\draw (S) node {$\mathcal{S}_3$};

\draw (Azz) -- (B3);
\draw (Bzz) -- (A3);



\end{tikzpicture}



\caption{The sequence $\cS$ split into $\cS_1$ and $\cS_2$ on the left and $\cS_3$ and $\cS_4$ on the right.\label{fig:SSS}}
\end{figure}
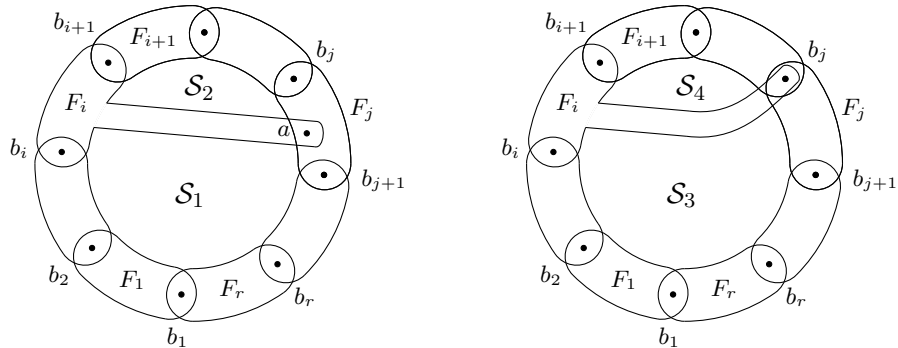

Define $\pi(\mathcal{S}_1)$ and $\pi(\mathcal{S}_2)$ similarly to $\pi(\mathcal{S})$ in \ref{S3}, that is, let
$$
\pi (\mathcal{S}_1)= \sum_{1\leq i'<i}\pi(b_{i'},b_{i'+1},F_{i'})+\pi(b_i,a,F_i)+\pi(a,b_{j+1},F_j)+\sum_{j< i'\leq r}\pi(b_{i'},b_{i'+1},F_{i'}),
$$
and
$$
\pi(\mathcal{S}_2)=\sum_{i< i'<j}\pi(b_{i'},b_{i'+1},F_{i'})+\pi(b_{j},a,F_j)+\pi(a,b_{i+1},F_i).
$$
Then,
\begin{align}
\pi (\mathcal{S}_1)&+\pi(\mathcal{S}_2)-\pi(\mathcal{S})\label{moredetail} \\
&=\pi(b_i,a,F_i)+\pi(a,b_{j+1},F_j)+\pi(b_{j},a,F_j)+\pi(a,b_{i+1},F_i)
-\pi(b_i,b_{i+1},F_i)-\pi(b_i,b_{j+1},F_j).\nonumber
\end{align}
By \ref{S1}, and as $a\notin\{b_1,\ldots,b_r\}$, $a$, $b_i$ and $b_{i+1}$ are distinct. Thus, by Proposition~\ref{simpleclass}, $\pi(b_i,a,F_i)+\pi(b_{i+1},a,F_i)-\pi(b_i,b_{i+1},F_i)\in \{0,2\}$. Similarly, $\pi(b_j,a,F_j)+\pi(b_{j+1},a,F_j)-\pi(b_j,b_{j+1},F_j)\in \{0,2\}$. Therefore, \eqref{moredetail} implies that $\pi(\mathcal{S}_1)+\pi(\mathcal{S}_2)-\pi(\mathcal{S})\in \{0,2,4\}$.

Now, as $i\neq j\pm 1\mod r$, both $\pi(\mathcal{S}_1)$ and $\pi(\mathcal{S}_2)$ are sums of at least three numbers from $\{1,2\}$. Therefore, if $\pi(\mathcal{S}_1)$ or $\pi(\mathcal{S}_2)$ is even, they must be at least $4$. As $\pi(\mathcal{S})$ is odd, if $\pi(\mathcal{S}_1)$ is even, then $\pi(\mathcal{S}_2)\leq 4+\pi(\mathcal{S})-\pi(\mathcal{S}_1)\leq \pi(\mathcal{S})$ and $\pi(\mathcal{S}_2)$ is odd. Similarly, if $\pi(\mathcal{S}_2)$ is even, then $\pi(\mathcal{S}_1)\leq \pi(\mathcal{S})$ and $\pi(\mathcal{S}_1)$ is odd. Therefore, one of $\pi(\mathcal{S}_i)$, $i\in[2]$, is odd.

Including the repetition of $b_1=b_{r+1}$ and $b_{i+1}$ respectively, as $i\neq j\pm 1 \mod r$, the sequence $\mathcal{S}_1$ contains $i+1+(r+1)-j<r+1$ vertices and the sequence $\mathcal{S}_2$ contains $j-i+2<r+1$ vertices. Thus, $r(\cS_1),r(\cS_2)<r=r(\cS)$. Therefore, taking the sequence $\mathcal{S}_{i'}$ with $i'\in [2]$ such that $\pi(\mathcal{S}_{i'})\leq \pi(\mathcal{S})$ is odd (so that the corresponding versions of \ref{S1}--\ref{S3} hold for $\mathcal{S}_{i'}$) contradicts~\ref{S4}.

\medskip

\noindent\textbf{Case II.} Suppose then that $F_i$ and $F_j$ share some vertex, say $b_{i'}$, in $\{b_1,\ldots,b_r\}$. We first show that we can assume that $b_j\in V(F_i)$. If $i'=j$, then this already holds. If $i'=i$ then switch $i$ and $j$ to get the previous case of $i'=j$. If $i'\neq i-1,i,i+1\mod r$, then keep $i$ unchanged and relabel $j=i'$. If $i'\neq j-1,j,j+1\mod r$, then relabel $j=i$ and $i=i'$. In each case we then get $i\neq j\pm 1\mod r$, $i\neq j$ and $b_j\in V(F_i)$.

This leaves only the case that $i'\in \{i-1,i+1\}\mod r$ and $i'\in \{j-1,j+1\}\mod r$. Therefore, switching $i$ and $j$ if necessary, we have that $i+1=i'=j-1\mod r$. We have that $b_{i'}=b_{i+1}$ is in $V(F_j)$, and now reverse the sequence $\cS$ so that, roughly speaking, $b_{i+1}$ and $F_i$ are assigned the same index. That is, consider the sequence $\cS'$ created by reversing $\mathcal{S}$, taking that
\[
\mathcal{S}'=b'_1F_1'b_2'F_2'b_3'\ldots b'_rF_r'b'_{r+1}=b_{r+1}F_rb_rF_{r-1}b_{r-1}\ldots b_2F_1b_1,
\]
where $b'_{j'}=b_{r+2-j'}$ for each $j'\in [r+1]$ and $F'_{j'}=F_{r+1-j'}$ for each $j'\in [r]$.
Then, $b'_{r+1-i}=b_{i+1}\in V(F_j)=V(F'_{r+1-j})$. Thus, relabelling $i=r+1-j$ and $j=r+1-i$ we have $i\neq j\pm 1\mod r$, $i\neq j$ and $b'_j\in V(F'_i)$. Using the corresponding definition for $\cS'$ in \ref{S3}, we have $\pi(\cS')=\pi(\cS)$ and $r(\cS')=r=r(\cS)$, so that $\mathcal{S}'$ satisfies its corresponding versions of \ref{S1}--\ref{S4}.
Thus, taking $\cS'$ instead of $\cS$ if necessary, we can always assume that $i\neq j\pm 1\mod r$, $i\neq j$ and $b_j\in V(F_i)$.

Now we show that we can assume that $j>i$. If $j<i$, then consider the rotated sequence of $\cS$ in which the index $i$ is shifted to 1, that is,
\[
\mathcal{S}''=b''_1F_1''b_2''F_2''b_3''\ldots b''_rF_r''b''_{r+1}=b_{i}F_ib_{i+1}\ldots b_rF_{r}b_{r+1}F_1b_2\ldots b_{i-1}F_{i-1}b_i,
\]
where $b''_{j'}=b_{j'+i-1}$ and $F''_{j'}=b_{j'+i-1}$ for each $j'\in [r]$ with addition modulo $r$ in the indices and $b''_{r+1}=b_i$. 
Let $i''=1$ and $j''=r+j-i+1$, and note that $i''< j''$ and $i''\neq j''\pm 1\mod r$. Furthermore, $b''_{j''}=b_{j}\in V(F_i)=V(F''_{i''})$. Similarly to $\cS'$, $\mathcal{S}''$ satisfies its corresponding versions of \ref{S1}--\ref{S4}, and therefore,  taking $\cS''$ instead of $\cS$ if necessary, we can assume that $i<j$.

Thus, we have that $i<j$, $i\neq j\pm 1\mod r$ and $b_j\in V(F_i)$.
As depicted in Figure~\ref{fig:SSS}, consider  now the sequences
$$
\mathcal{S}_3=b_1F_1b_2F_2\ldots b_iF_ib_jF_{j}b_{j+1}\ldots F_{r}b_{r+1}\quad\text{ and }\quad \mathcal{S}_4=b_jF_ib_{i+1}F_{i+1}b_{i+2}\ldots F_{j-1}b_j.
$$
We will show that one of these sequences satisfies \ref{S1}--\ref{S3} in place of $\mathcal{S}$ and contradicts the minimality of $\mathcal{S}$ in \ref{S4}. That each sequence $\mathcal{S}_3$ and $\mathcal{S}_4$ satisfies the corresponding versions of \ref{S1} and \ref{S2} follows immediately from \ref{S1} and \ref{S2} as $b_j\in V(F_j)$.

Writing
$$
\pi (\mathcal{S}_3)= \sum_{1\leq i'<i}\pi(b_{i'},b_{i'+1},F_{i'})+\pi(b_i,b_j,F_i)+\sum_{j\leq  i'\leq r}\pi(b_{i'},b_{i'+1},F_{i'}),
$$
and
$$
\pi(\mathcal{S}_4)=\pi(b_j,b_{i+1},F_i)+\sum_{i< i'<j}\pi(b_{i'},b_{i'+1},F_{i'}),
$$
we have
\begin{align}
\pi (\mathcal{S}_3)+\pi(\mathcal{S}_4)-\pi(\mathcal{S})=\pi(b_i,b_j,F_i)+\pi(b_j,b_{i+1},F_i)-\pi(b_i,b_{i+1},F_i)\label{moredetail2}.
\end{align}
By \ref{S1} and Proposition~\ref{simpleclass}, we have $\pi(b_i,b_j,F_i)+\pi(b_j,b_{i+1},F_i)-\pi(b_i,b_{i+1},F_i)\in \{0,2\}$. Thus, by \eqref{moredetail2}, we have $\pi(\mathcal{S}_3)+\pi(\mathcal{S}_4)-\pi(\mathcal{S})\in \{0,2\}$.

 Now, as $i\neq j\pm 1\mod r$, both $\pi(\mathcal{S}_3)$ and $\pi(\mathcal{S}_4)$ are sums of at least two numbers from $\{1,2\}$, and are therefore at least 2.  As $\pi(\mathcal{S})$ is odd, if $\pi(\mathcal{S}_3)$ is even then $\pi(\mathcal{S}_4)\leq 2+\pi(\mathcal{S})-\pi(\mathcal{S}_3)\leq \pi(\mathcal{S})$ and $\pi(\mathcal{S}_4)$ is odd. Similarly, if $\pi(\mathcal{S}_4)$ is even, then $\pi(\mathcal{S}_3)\leq \pi(\mathcal{S})$ and $\pi(\mathcal{S}_3)$ is odd. Therefore, one of $\pi(\mathcal{S}_i)$, $i\in\{3,4\}$ is odd.

Including the repetition of $b_1=b_{r+1}$ and $b_j$ respectively, the sequence $\mathcal{S}_3$ contains $i+(r+1)-j+1<r+1$ vertices and the sequence $\mathcal{S}_4$ contains $j-i+1<r+1$ vertices. Thus, $r(\cS_3),r(\cS_4)<r=r(\cS)$. Therefore, taking the sequence $\mathcal{S}_{i'}$ with $i'\in \{3,4\}$ such that $\pi(\mathcal{S}_{i'})$ is odd (so that the corresponding versions of \ref{S1}--\ref{S3} hold for $\mathcal{S}_{i'}$) contradicts~\ref{S4}, completing the proof of the claim.
\end{poc}

Using Claim~\ref{oddtwo}, we can show that each graph $F_i$ in the minimal sequence $\mathcal{S}$ is distinct.

\begin{claim}\label{norep}
$F_1,\ldots,F_r$ are distinct graphs in $\{H_1,\ldots,H_s\}$, and thus are pairwise edge disjoint.
\end{claim}
\begin{poc}
If $r=2$ and $F_1=F_2$, then \ref{S1} implies that $\sum_{i=1}^{r}\pi(b_i,b_{i+1},F_i)=\pi(b_1,b_2,F_1)+\pi(b_2,b_1,F_2)=2\pi(b_1,b_2,F_1)=0\mod 2$, violating \ref{S3}. Thus, the claim holds for $r=2$ by \ref{S2}.

Now suppose $r\ge 3$ and that, for distinct $i,j\in [r]$, $F_i=F_j$. Then, by Claim~\ref{oddtwo}, it must be that $i=j\pm 1\mod r$. Say then, without loss of generality, that $i=j-1 \mod r$. Furthermore, by relabelling the indices in the sequence $\cS$ cyclically (as for $\cS''$ above), we can assume that $i< r$, and hence $j=i+1$, so that $F_i=F_{i+1}$. Consider the sequence $\cS'$ formed from $\cS$ by replacing $F_ib_{i+1}F_{i+1}$ by $F_i$ in $\cS$, so that
$$
\cS'= b_1F_1 b_2F_2 b_3\ldots b_{i-1}F_{i-1}b_iF_ib_{i+2}F_{i+2}\ldots  b_r F_r b_{r+1},
$$
and note that $\cS'$ satisfies its corresponding versions of \ref{S1} and \ref{S2}.
By Proposition~\ref{simpleclass} and \ref{S1}, as $F_i=F_{i+1}$, we have  that $\pi(b_i,b_{i+1},F_i)+\pi(b_{i+1},b_{i+2},F_{i+1})-\pi(b_i,b_{i+2},F_i)\in \{0,2\}$. Then, defining $\pi(\mathcal{S}')$ as in \ref{S3},
$$
\pi(\mathcal{S}')=\pi(\cS)-\pi(b_i,b_{i+1},F_i)-\pi(b_{i+1},b_{i+2},F_{i+1})+\pi(b_i,b_{i+2},F_i)\in \{\pi(\cS),\pi(S)-2\}.
$$
As $\pi(\cS)$ is odd, $\pi(\mathcal{S}')\leq \pi(\cS)$ and $\pi(\mathcal{S}')$ is odd. Therefore, as $r(\cS')=r-1$, this contradicts~\ref{S4}.
\end{poc}

\subsubsection{The right subgraphs in which to vary path lengths}\label{sec:tovary}

Using the previous claims in this section, we can now choose a large subcollection of graphs $F_i$ in $\cS$, varying the paths in which leads to many odd cycles of distinct lengths. Firstly, by Claim~\ref{norep} and relabelling, we can assume that $F_i=H_i$ for each $i\in [r]$. Recall then that we have positive integers $\ell_i$, $i\in [r]$, such that~\ref{bigprop} holds with $H_i=F_i$.

Now, partition $[r]$ as $I_1\cup I_2\cup I_3$ and find distinct vertices $u_1=u_{r+1},u_2,\ldots,u_r$ with $u_i,u_{i+1}\in V(F_i)$ for each $i\in [r]$, and paths $P_i$, $i\in I_2\cup I_3$, such that the following hold.

\stepcounter{propcounter}
\begin{enumerate}[label = {\bfseries \Alph{propcounter}\arabic{enumi}}]
  \item $3\sum_{i\in I_1}(\ell_i+1)\geq \sum_{i\in [r]}(\ell_i+1)$.\label{bang0}
  \item Any collection of paths containing exactly one $u_i,u_{i+1}$-path in $F_i$, for each $i\in I_1$, and $\{P_i\}_{i\in I_2\cup I_3}$ form a cycle.\label{bang1}
  \item The paths $P_i$, $i\in I_2\cup I_3$, have total length $\ell_0\leq \sum_{i\in I_2\cup I_3}(\ell_i+1)$, where $\ell_0$ is such that $\ell_0+\sum_{i\in I_1}\pi(u_i,u_{i+1},F_i)$ is odd. \label{bang2}
\end{enumerate}


We find the partition, vertices and paths differently depending on the length of the sequence $\cS$. If $r\geq 4$, then we are in \textbf{Case 1}, if $r=2$, then we are in \textbf{Case 2}, and if $r=3$ we are in \textbf{Case 3}.

\medskip

\noindent\textbf{Case 1.} Suppose first that $r\geq 4$.
 Partition $[r]=I_1\cup I_2 \cup I_3$ so that, for each $j\in [3]$, there is no solution to $x=y+1\mod r$ with $x,y\in I_j$. Thus, by Claim~\ref{oddtwo}, for every $j\in [3]$, graphs in $\{F_i:i\in I_j\}$ are pairwise vertex disjoint.
By averaging, there exists some $j\in [3]$ such that $\sum_{i\in I_j}(\ell_i+1)\geq \sum_{i\in [r]}(\ell_i+1)/3$. By relabelling, we may assume that $j=1$, and hence \ref{bang0} holds.

 Now, for each $i\in I_2$, find a shortest path $P_i$ between $V(F_{i-1})$ and $V(F_{i+1})$ in $F_i$, and label vertices so that this is a $u_{i},u_{i+1}$-path with $u_i\in V(F_{i-1})$ and $u_{i+1}\in V(F_{i+1})$. Note that, by minimality of $P_i$, all internal vertices of $P_i$ lie in $V(F_i)\setminus (V(F_{i-1})\cup V(F_{i+1}))$. Furthermore, as $r\geq 4$, by Claim~\ref{oddtwo}, we have that $u_i\neq u_{i+1}$.

 For each $i\in I_3$, if $i-1,i+1\in I_2$, let $P_i$ be a shortest $u_i,u_{i+1}$-path in $F_i$. If $i-1\in I_2$ and $i+1\notin I_2$, let $P_i$ be a shortest path from $u_i$ to $V(F_{i+1})$ in $F_i$ and label its endpoint in $V(F_{i+1})$ by $u_{i+1}$.  If $i-1\notin I_2$ and $i+1\in I_2$, let $P_i$ be a shortest path from $V(F_{i-1})$ to $u_{i+1}$ in $F_i$ and label its endpoint in $V(F_{i-1})$ by $u_{i}$. If $i-1,i+1\notin I_2$, let  $P_i$ be a shortest path between $V(F_{i-1})$ and $V(F_{i+1})$ in $F_i$, and label vertices so that this is a $u_{i},u_{i+1}$-path with $u_i\in V(F_{i-1})$ and $u_{i+1}\in V(F_{i+1})$.

Note that we have chosen a vertex $u_i$ for each $i\in [r]$ with $i$ or $i-1$ in $I_2\cup I_3$, and, for each $i\in I_2\cup I_3$, a path $P_i$. As there is no solution to $x=y+1\mod r$ in $I_1$, $\{i,i+1:i\in I_2\cup I_3\}=[r]$. Thus, we have in fact chosen vertices $u_i$ for all $i\in[r]$. To see that vertices $u_i$, $i\in [r]$, are distinct, note that, for each $i\in [r]$, we chose $u_i\in V(F_{i-1})\cap V(F_{i})$.  Therefore, if there is some $i\neq j$ with $u_i=u_j$, then as $u_j\in V(F_{i-1})\cap V(F_{i})$ and $r\geq 4$, by Claim~\ref{oddtwo}, we must have $j=i-1\mod r$. However, similarly, as $u_i\in V(F_{j-1})\cap V(F_{j})$, we must have $i=j-1\mod r$, a contradiction.

 Let $u_{r+1}=u_1$. For each $i\in I_2\cup I_3$, $P_i$ was a shortest path between two vertices in $F_i$, and thus by \ref{bigprop}, has length at most $\ell_i+1$.
 Observe crucially that, by our careful selection of shortest paths, the paths in $\{P_i:{i\in I_2\cup I_3}\}$ are pairwise internally disjoint, and furthermore internally disjoint from any $u_i,u_{i+1}$-path in $F_i$ for any $i\in I_1$. In particular, this, together with the graphs $F_i$, $i\in I_1$, being pairwise vertex disjoint, implies that \ref{bang1} holds. It is left to show that \ref{bang2} holds. Let $\ell_0$ be the total length of the paths $P_i$, $i\in I_2\cup I_3$. As, for each $i\in I_2\cup I_3$, $P_i$ has length at most $\ell_i+1$, we have $\ell_0\le \sum_{i\in I_2\cup I_3}(\ell_i+1)$, as required in \ref{bang2}.

For the rest of \ref{bang2}, we need to show that $\ell_0+\sum_{i\in I_1}\pi(u_i,u_{i+1},F_i)$ is odd. Now, as $r\geq 4$, we have $\pi(\cS)\geq 5$ by \ref{S3}, and hence $\pi(\cS)+r(\cS)\geq 9$. Now, suppose for some $i,j\in [r]$ with $j=i-1 \mod r$, $\pi(u_i,b_i,F_i)+\pi(u_i,b_i,F_{j})$ is odd. Then, $u_i\neq b_i$, and therefore $\pi(u_i,b_i,F_i)+\pi(u_i,b_i,F_{j})=3$. Thus, the sequence $\cS'=u_iF_ib_iF_ju_i$ has $\pi(\cS')+r(\cS')=5$ and satisfies its corresponding version of \ref{S1}--\ref{S3}, contradicting~\ref{S4}. Therefore, working mod $r$ in the indices, for each $i\in [r]$, we have
 \begin{equation}
 \pi(u_i,b_i,F_i)+\pi(u_i,b_i,F_{i-1})=0\mod 2.\label{evenly}
\end{equation}
 Futhermore, for each $i\in [r]$, we have by two applications of (the second part of) Proposition~\ref{simpleclass}, that
\begin{equation*}
 \pi(u_i,u_{i+1},F_i)=\pi(u_i,b_i,F_i)+\pi(b_i,u_{i+1},F_i)=\pi(u_i,b_i,F_i)+\pi(b_i,b_{i+1},F_i)+\pi(b_{i+1},u_{i+1},F_i)\!\!\!\! \mod 2.
\end{equation*}
Thus, working mod $r$ in the indices, we have
\begin{align*}
\sum_{i\in [r]} \pi(u_i,u_{i+1},F_i)&=\sum_{i\in [r]}\pi(u_i,b_i,F_i)+\sum_{i\in [r]}\pi(b_i,b_{i+1},F_i)+\sum_{i\in [r]}\pi(b_{i+1},u_{i+1},F_i)\mod 2\\
&=\sum_{i\in [r]}\pi(b_i,b_{i+1},F_i)+\sum_{i\in [r]}(\pi(u_i,b_i,F_i)+\pi(u_i,b_i,F_{i-1}))\mod 2\\
&\overset{\eqref{evenly}}{=} \sum_{i\in [r]}\pi(b_i,b_{i+1},F_i)\mod 2.
\end{align*}

For each $i\in I_2\cup I_3$, $P_i$ is a $u_i,u_{i+1}$-path in $F_i$, and therefore has length equal to $\pi(u_i,u_{i+1},F_i)\mod 2$. Thus, $\ell_0$ is equal to $\sum_{i\in I_2\cup I_3}\pi(u_i,u_{i+1},F_i) \mod 2$. Therefore,
$$
\ell_0+\sum_{i\in I_1}\pi(u_i,u_{i+1},F_i)=\sum_{i\in [r]} \pi(u_i,u_{i+1},F_i)=\sum_{i\in [r]}\pi(b_i,b_{i+1},F_i)\mod 2.
$$
Combined with \ref{S3}, we have that $\ell_0+\sum_{i\in I_1}\pi(u_i,u_{i+1},F_i)$ is odd, completing the proof of \ref{bang2}.

\medskip

\noindent\textbf{Case 2.} Suppose then that $r=2$. Assume, by relabelling that $\ell_1\geq \ell_2$ and let $I_1=\{1\}$, $I_2=\{2\}$ and $I_3=\varnothing$. Note that \ref{bang0} holds. By \ref{S3}, we have that $\pi(b_1,b_2,F_1)+\pi(b_1,b_2,F_2)$ is odd. Find distinct vertices $u_1,u_2\in V(F_1)\cap V(F_2)$ and a $u_1,u_2$-path $P_2$ in $F_2$ so that $\ell(P_2)+\pi(u_1,u_2,F_1)$ is odd, and subject to this $P_2$ has the shortest possible length. Note that this is possible as taking $u_1=b_1$, $u_2=b_2$ and letting $P_2$ be any $u_1,u_2$-path in $F_2$ (which exists due to \ref{bigprop}) satisfies these conditions, and, furthermore, by \ref{bigprop}, $P_2$ has length at most $\ell_2+1$. Therefore, \ref{bang2} holds for $P_2$.

Now, suppose $P_2$ has some internal vertex $u$ in $F_1$. Then, note $u\in V(P_2)\subseteq V(F_2)$, and split $P_2$ as a $u_1,u$-path $Q_1$ and a $u,u_2$-path $Q_2$. By Proposition~\ref{simpleclass}, we have, working mod $2$, that
$$
1=\ell(P_2)+\pi(u_1,u_2,F_1)=\ell(Q_1)+\pi(u_1,u,F_1)+\ell(Q_2)+\pi(u,u_2,F_1) \mod 2.
$$
Therefore, one of $\ell(Q_1)+\pi(u_1,u,F_1)$ or $\ell(Q_2)+\pi(u,u_2,F_1)$ must be odd, contradicting the minimality of $P_2$. Therefore, $P_2$ has no internal vertices in $F_1$, and hence \ref{bang1} holds.

\medskip

\noindent\textbf{Case 3.} Suppose finally that $r=3$. Assume, by relabelling, that $\ell_1=\max_{i\in [r]}\ell_i$ and let $I_1=\{1\}$, $I_2=\{2\}$ and $I_3=\{3\}$. Note that \ref{bang0} holds.

Let us first show that $F_1\cup F_2$, $F_2\cup F_3$, and $F_1\cup F_3$ are bipartite. Suppose, for contradiction, that $F_2\cup F_3$ is not bipartite, and let $A,B$ and $A',B'$ be the bipartitions of $F_2$ and $F_3$, labelled so that $b_3\in A\cap A'$. As $A\cup A', B\cup B'$ is not a bipartition of $F_2\cup F_3$, there must be some vertex, $u$ say, in $A\cap B'$ or $A'\cap B$. Then, $\pi(b_3,u,F_2)+\pi(u,b_3,F_3)=3$. Thus, $\cS'=b_3F_2uF_3b_3$ satisfies the corresponding version of \ref{S1}--\ref{S3} with $r(\cS')+\pi(\cS')=5$. By \ref{S1} and \ref{S3}, $\pi(S)\geq 3$, and therefore $\pi(\cS)+r(\cS)\geq 6$, contradicting \ref{S4}. Therefore, $F_2\cup F_3$ must be bipartite. Similarly, $F_1\cup F_2$ and $F_1\cup F_3$ are bipartite.

As $F_2\cup F_3$ is bipartite, we have, by Proposition~\ref{simpleclass}, that
\[
 \pi(b_2,b_4,F_2\cup F_3)=\pi(b_2,b_3,F_2\cup F_3)+\pi(b_3,b_4,F_2\cup F_3)=\pi(b_2,b_3,F_2)+\pi(b_3,b_4,F_3)\mod 2.
\]

Let $Q$ be a shortest $b_2,b_4$-path in $F_2\cup F_3$, so that $\ell(Q)= \pi(b_2,b_4,F_2\cup F_3) \mod 2$. As $F_2$ and $F_3$ have diameter at most $\ell_2+1$ and $\ell_3+1$ respectively by \ref{bigprop}, we have $\ell(Q)\leq \ell_2+\ell_3+2$. As $b_4=b_1$, $Q$ is a $b_1,b_2$-path. Now, find distinct vertices $u_1,u_2$ with $u_1\in V(F_1)\cap V(F_3)$ and $u_2\in V(F_1)\cap V(F_2)$, and a $u_1,u_2$-path $P_2$ in $F_2\cup F_3$ so that $\ell(P_2)+\pi(u_1,u_2,F_1)$ is odd, and subject to this $\ell(P_2)$ has the shortest possible length. Note that such a shortest path $P_2$ indeed exists as the path $Q$ satisfies the other requirements with $u_1=b_1$ and $u_2=b_2$ as $\ell(Q)+\pi(b_1,b_2,F_1)$ is odd due to~\ref{S3}. Thus, $\ell(P_2)\le\ell(Q)\le \ell_2+\ell_3+2$. Let $P_3$ be the path with only the vertex $u_2$. Therefore, \ref{bang2} holds for $P_2$ and $P_3$.

Suppose to the contrary that $P_2$ has some internal vertex $u$ in $F_1$. Then, note $u\in V(P_2)\subseteq V(F_2)$, and split $P_2$ as a $u_1,u$-path $Q_1$ and a $u,u_2$-path $Q_2$. By Proposition~\ref{simpleclass}, we have
$$
1=\ell(P_2)+\pi(u_1,u_2,F_1)=\ell(Q_1)+\pi(u_1,u,F_1)+\ell(Q_2)+\pi(u,u_2,F_1) \mod 2.
$$
Therefore, one of $\ell(Q_1)+\pi(u_1,u,F_1)$ or $\ell(Q_2)+\pi(u,u_2,F_1)$ must be odd, contradicting the minimality of $P_2$. Therefore, $P_2$ has no internal vertices in $F_1$, and hence \ref{bang1} holds.

Now, combining $P_2$ with any $u_1,u_2$-path in $F_1$ with length equivalent to $\pi(u_1,u_2,F_1)\mod 2$ gives an odd cycle. Thus, as both $F_1\cup F_2$ and $F_1\cup F_3$ are bipartite, $P_2$ must have some edge from $F_2$ and some edge from $F_3$, and hence we can pick $u_3$ as an arbitrary internal vertex of $P_2$ in $V(F_2)\cap V(F_3)$. This completes the partition $[3]=I_1\cup I_2\cup I_3$, distinct vertices $u_1,u_2$ and $u_3$, with $u_1,u_2\in V(F_1)$, $u_2,u_3\in V(F_2)$ and $u_3,u_1\in V(F_3)$, and paths $P_2$ and $P_3$ such that \ref{bang0}--\ref{bang2} hold.

\subsubsection{Varying paths in the chosen subgraphs}\label{sec:vary}

Thus, in each of the three cases, we have a partition $[r]=I_1\cup I_2\cup I_3$, distinct vertices $u_1=u_{r+1},u_2,\ldots,u_r$  with $u_i,u_{i+1}\in V(F_i)$ for each $i\in [r]$, and paths $P_i$, $i\in I_2\cup I_3$, for which \ref{bang0}--\ref{bang2} hold. Let
\begin{align}\label{eq-arewethereyet}
\ell&=\ell_0+\sum_{i\in I_1}(\ell_i+1)
\overset{\text{\ref{bang2}}}{\leq}\sum_{i\in [r]}(\ell_i+1)
\overset{\text{\ref{bang0}}}{\leq} 3\sum_{i\in I_1}(\ell_i+1).
\end{align}

Now, if  $\ell_i'$, $i\in I_1$, is a collection of integers satisfying $\ell_i'\in [\ell_i,\ell_i\cdot k^{1-\eps/2}]$ and $\ell_i'=\pi(u_i,u_{i+1},F_i)\mod 2$, for each $i\in I_1$, then, by taking a $u_i,u_{i+1}$-path in $F_i$ with length $\ell_i'$ by \ref{bigprop}, for each $i\in I_1$, and combining these paths with the paths $P_i$, $i\in I_2\cup I_3$, by \ref{bang1} we get a cycle with length $\ell_0+\sum_{i\in I_1}\ell_i'$.

As, by \ref{bang2}, $\ell_0+\sum_{i\in I_1}\pi(u_i,u_{i+1},F_i)$ is odd, there are sets of such numbers $\ell_i'$, $i\in I_1$, with $\ell_0+\sum_{i\in I_1}\ell_i'=t$ for any odd number $t$ such that
\[
\ell_0+\sum_{i\in I_1}(\ell_i+1)\leq t\leq \ell_0+\sum_{i\in I_1}(\lfloor \ell_i\cdot k^{1-\eps/2}\rfloor-1).
\]
Now, we have $\ell=\ell_0+\sum_{i\in I_1}(\ell_i+1)$, so, as $k^{\eps/2}\geq d_0^{\eps/2}\geq 8$, we have
\[
\ell_0+\sum_{i\in I_1}(\lfloor \ell_i\cdot k^{1-\eps/2}\rfloor-1)
\geq \sum_{i\in I_1}(\ell_i\cdot k^{1-\eps/2}-2)\geq \sum_{i\in I_1}(8\ell_i\cdot k^{1-\eps}-2)\geq \sum_{i\in I_1}(3\ell_i+3)\cdot k^{1-\eps}\overset{\eqref{eq-arewethereyet}}{\geq } \ell\cdot k^{1-\eps}.
\]
Thus, for each odd integer $t\in [\ell,\ell\cdot k^{1-\eps}]$, $G$ contains a cycle with length $t$.
This completes the proof of Theorem~\ref{thm:chrom}.


\section*{Acknowledgements}
We are very grateful to Benny Sudakov, and the referee, for suggestions that improved the presentation of this paper.


\begin{thebibliography}{10}

\bibitem{probmethod}
N.~Alon and J.~H. Spencer.
\newblock {\em \emph{\textbf{The {P}robabilistic {M}ethod}}}.
\newblock John Wiley \& Sons, 2004.

\bibitem{bollobas1977cycles}
B.~Bollob{\'a}s.
\newblock Cycles modulo {$k$}.
\newblock {\em Bulletin of the London Mathematical Society}, 9(1):97--98, 1977.

\bibitem{BT98}
B.~Bollob{\'a}s and A.~Thomason.
\newblock Proof of a conjecture of {M}ader, {E}rd{\H{o}}s and {H}ajnal on
  topological complete subgraphs.
\newblock {\em European Journal of Combinatorics}, 19(8):883--887, 1998.

\bibitem{erdos1975-42}
P.~Erd\H{o}s.
\newblock Some recent progress on extremal problems in graph theory.
\newblock {\em Congr. Numer}, 14:3--14, 1975.

\bibitem{erdos1981-20}
P.~Erd\H{o}s.
\newblock Problems and results in graph theory.
\newblock {\em The theory and applications of graphs ({K}alamazoo, {MI},
  1980)}, pages 331--341, 1981.

\bibitem{erdos1984-21}
P.~Erd{\H{o}}s.
\newblock Some new and old problems on chromatic graphs.
\newblock {\em Combinatorics and applications}, pages 118--126, 1984.

\bibitem{Erd95}
P.~Erd{\H{o}}s.
\newblock Some old and new problems in various branches of combinatorics.
\newblock {\em Discrete Mathematics}, 165:227--231, 1997.

\bibitem{EH66}
P.~Erd{\H{o}}s and A.~Hajnal.
\newblock On chromatic number of graphs and set-systems.
\newblock {\em Acta Mathematica Hungarica}, 17(1-2):61--99, 1966.

\bibitem{gyarfas}
A.~Gy{\'a}rf{\'a}s.
\newblock Graphs with {$k$} odd cycle lengths.
\newblock {\em Discrete Mathematics}, 103(1):41--48, 1992.

\bibitem{gyarfas1984distribution}
A.~Gy{\'a}rf{\'a}s, J.~Koml{\'o}s, and E.~Szemer{\'e}di.
\newblock On the distribution of cycle lengths in graphs.
\newblock {\em Journal of Graph Theory}, 8(4):441--462, 1984.

\bibitem{KLSS17}
J.~Kim, H.~Liu, M.~Sharifzadeh, and K.~Staden.
\newblock Proof of {K}oml{\'o}s's conjecture on {H}amiltonian subsets.
\newblock {\em Proceedings of the London Mathematical Society},
  115(5):974--1013, 2017.

\bibitem{K-Sz-1}
J.~Koml{\'o}s and E.~Szemer{\'e}di.
\newblock Topological cliques in graphs.
\newblock {\em Combinatorics, Probability and Computing}, 3(2):247--256, 1994.

\bibitem{K-Sz-2}
J.~Koml{\'o}s and E.~Szemer{\'e}di.
\newblock Topological cliques in graphs {II}.
\newblock {\em Combinatorics, Probability and Computing}, 5(1):79--90, 1996.

\bibitem{Kur30}
K.~Kuratowski.
\newblock Sur le probleme des courbes gauches en topologie.
\newblock {\em Fund. Math.}, 16:271--283, 1930.

\bibitem{liu2017proof}
H.~Liu and R.~Montgomery.
\newblock A proof of {M}ader's conjecture on large clique subdivisions in
  {$C_4$}-free graphs.
\newblock {\em Journal of the London Mathematical Society}, 95(1):203--222,
  2017.

\bibitem{Mad67}
W.~Mader.
\newblock Homomorphieeigenschaften und mittlere {K}antendichte von {G}raphen.
\newblock {\em Mathematische Annalen}, 174(4):265--268, 1967.

\bibitem{Mad72}
W.~Mader.
\newblock Hinreichende {B}edingungen f{\"u}r die {E}xistenz von {T}eilgraphen,
  die zu einem vollst{\"a}ndigen {G}raphen hom{\"o}omorph sind.
\newblock {\em Mathematische Nachrichten}, 53(1-6):145--150, 1972.

\bibitem{SV08}
B.~Sudakov and J.~Verstra{\"e}te.
\newblock Cycle lengths in sparse graphs.
\newblock {\em Combinatorica}, 28:357--372, 2008.

\bibitem{SV11}
B.~Sudakov and J.~Verstra{\"e}te.
\newblock Cycles in graphs with large independence ratio.
\newblock {\em Journal of Combinatorics}, 2(1):83--102, 2011.

\bibitem{Tho84}
C.~Thomassen.
\newblock Subdivisions of graphs with large minimum degree.
\newblock {\em Journal of Graph Theory}, 8(1):23--28, 1984.

\bibitem{Tho85}
C.~Thomassen.
\newblock Problems 20 and 21.
\newblock In {\em Graphs, Hypergraphs and Applications}. H. Sachs, Ed.: 217.
  Teubner. Leipzig., 1985.

\bibitem{thomassen}
C.~Thomassen.
\newblock Configurations in graphs of large minimum degree, connectivity, or
  chromatic number.
\newblock In {\em Proceedings of the third international conference on
  Combinatorial mathematics}, pages 402--412, 1989.

\bibitem{Ver05}
J.~Verstra{\"e}te.
\newblock Unavoidable cycle lengths in graphs.
\newblock {\em J. Graph Theory}, 49:151--167, 2005.

\bibitem{SurVer}
J.~Verstra{\"e}te.
\newblock Extremal problems for cycles in graphs.
\newblock In {\em Recent trends in combinatorics}, pages 83--116. Springer,
  2016.

\end{thebibliography}
\end{document}